\documentclass[1 leqno,11pt,psfig]{amsart}
\usepackage{amssymb, amsmath,amsmath,latexsym,amssymb,amsfonts,amsbsy, amsthm,mathtools,graphicx,CJKutf8,CJKnumb,CJKulem,color}
\usepackage{graphicx,epsfig}
\usepackage{graphicx}
\usepackage{amssymb}
\usepackage{graphicx}
\usepackage{graphicx,epsfig}
\usepackage{amsmath}
\usepackage{mathrsfs}
\usepackage{amssymb}

\usepackage{amssymb,mathrsfs,amsfonts,amsmath,amsbsy,tikz}\usepackage{fontenc}\usepackage{textcomp}

\renewcommand{\theequation}{\thesection.\arabic{equation}}
\numberwithin{equation}{section}

\allowdisplaybreaks

\let\al=\alpha
\let\b=\beta

\let\pa=\partial

\def\bbT{\mathbb{T}}

\def\R{\mathbf R}
\def\Z{\mathbf Z}

\def\no{\noindent}

\newcommand{\beq}{\begin{equation}}
\newcommand{\eeq}{\end{equation}}
\newcommand{\ben}{\begin{eqnarray}}
\newcommand{\een}{\end{eqnarray}}
\newcommand{\beno}{\begin{eqnarray*}}
\newcommand{\eeno}{\end{eqnarray*}}
\newcommand\numberthis{\addtocounter{equation}{1}\tag{\theequation}}

\newtheorem{theorem}{Theorem}[section]
\newtheorem{definition}[theorem]{Definition}
\newtheorem{lemma}[theorem]{Lemma}
\newtheorem{proposition}[theorem]{Proposition}
\newtheorem{corol}[theorem]{Corollary}
\newtheorem{example}[theorem]{Example}
\newtheorem{remark}[theorem]{Remark}

\begin{document}
\begin{CJK*}{UTF8}{gkai}
\title[The number of traveling wave families]{The number of traveling wave families in a running water with Coriolis force}

\author{Zhiwu Lin}
\address{School of Mathematics, Georgia Institute of Technology, 30332, Atlanta, GA, USA}
\email{zlin@math.gatech.edu}

\author{Dongyi Wei}
\address{School of Mathematical Science, Peking University, 100871, Beijing, P. R. China}
\email{jnwdyi@163.com}

\author{Zhifei Zhang}
\address{School of Mathematical Science, Peking University, 100871, Beijing, P. R. China}
\email{zfzhang@math.pku.edu.cn}

\author{Hao Zhu}
\address{Department of Mathematics, Nanjing University,  210093, Nanjing, Jiangsu, P. R. China}
\email{haozhu@nju.edu.cn}

\date{\today}

\maketitle
\begin{abstract}
In this paper, we study the number of traveling wave families  near a shear flow  under the influence of  Coriolis force, where the traveling speeds lie outside the range of the flow $u$.  Under  the $\beta$-plane approximation,
if
  the flow $u$ has a critical point at which $u$ attains its minimal (resp. maximal) value,
then  a unique transitional $\beta$ value exists in the positive (resp. negative) half-line
  such that the number of  traveling wave families  near the shear flow changes suddenly  from finite to infinite when  $\beta$ passes through it.
On the other hand, if
  $u$ has no such critical points, then the number
is always finite for positive (resp.  negative)  $\beta$ values.
  This is true  for general shear flows under  mildly technical assumptions, and for a large   class of shear flows including a cosine jet $u(y) = {1+\cos(\pi y)\over
2}$ (i.e. the sinus profile) and analytic monotone flows
   unconditionally.
  The sudden change of the number of  traveling wave families indicates that  long time dynamics around the  shear flow is much richer than the
  non-rotating case, where no such traveling wave families exist.
\end{abstract}
\section{Introduction}

 The earth's rotation influences dynamics of  large-scale flows significantly. Under  the $\beta$-plane approximation,
  the motion for such a flow could be described by 2-D incompressible Euler equation with rotation
\begin{align}\label{Euler equation}
		\partial_{t}\vec{v}+(\vec{v}\cdot\nabla)\vec{v}=-\nabla P-\beta yJ\vec
{v},\quad  \nabla\cdot\vec{v}=0,
	\end{align}
where $\vec{v}=(v_{1},v_{2})$ is the fluid velocity, $P$ is the pressure,
$
J=%
\begin{pmatrix}
0 & -1\\
1 & 0
\end{pmatrix}
$
is the rotation matrix, and $\beta$ is the Rossby number. Here we study the fluid in a {\it periodic} finite channel
$\Omega=D_T=\bbT_T\times{[y_1,y_2]}$, $\bbT_T=\R/(T\Z)$
with non-permeable  boundary condition on $\pa\Omega$:
\begin{align}\label{boundary condition for euler}
v_2=0\quad \text{on}\quad  y=y_1, y_2.
\end{align}
The $\beta$-plane approximation is commonly used for large-scale motions in geophysical fluid dynamics \cite{McWilliams,Pedlosky1987}.
 The
vorticity form of (\ref{Euler equation}) takes
\begin{equation}
\partial_{t}\omega+(\vec{v}\cdot\nabla)\omega+\beta v_{2}%
=0,\label{vorticity-eqn}%
\end{equation}
where $\omega=\partial_{x}v_{2}-\partial_{y}v_{1}$. By the incompressible condition, we introduce the stream function $\psi$ such that  $\vec{v}=\nabla^{\perp}\psi=(\pa_y\psi,-\pa_x\psi)$.
Consider the shear flow $(u(y),0)$, which is a steady solution of (\ref{vorticity-eqn}). The
linearized equation of (\ref{vorticity-eqn}) around $(u(y),0)$ takes
\begin{align}\label{linearized Euler equation}
	\partial_{t}\omega+u\partial_{x}\omega-(\beta-u^{\prime\prime})\partial_x\psi=0,
	\end{align}
which was derived in \cite{Rossby1939}.

In the study of long time dynamics near a shear flow,
 the most rigid case is  the  nonlinear inviscid damping (to a shear
flow), a kind of asymptotic stability. This means that if the initial velocity is taken close enough to the given shear flow in some function space, then the
velocity tends asymptotically to a nearby shear flow in this space.
The existence of nearby non-shear steady states or
 traveling waves means that nonlinear inviscid damping (to a shear
flow) is not true, and long time dynamics near the shear flow may be richer and fruitful.
 To understand the richer   long time dynamics near the shear flow in this situation, an important step is  to clarify whether the number of
curves of nearby
 traveling waves with traveling speeds converging to different points is infinite.
Indeed,
if the number is finite, then the velocity might tend asymptotically  to some nonlinear superpositions of finite  such non-shear states  when the initial data is taken close  to the flow, and  quasi-periodic nearby solutions are expected, which indicates new but not so complicated dynamics. If
the number is infinite, then the evolutionary velocity might tend asymptotically  to  superpositions of infinite such non-shear states, and   almost periodic nearby solutions  potentially exist, which predicts complicated even chaotic long time behavior  near the flow. Similar phenomena were observed numerically in the study of Vlasov-Poisson system, a  model describing
collisionless plasmas \cite{DZ90,BD93,BD94,LD09}. This model shares many similarities with the 2D incompressible Euler equation. By numerical simulations, it was found that for some initial perturbation near homogeneous states, the asymptotic state toward which the system evolves can be described by a superposition of BGK modes \cite{DZ90}. This offers a hint for further numerical study in the 2D Euler case.
It is very challenging to study  long time dynamics near a shear flow   in a fully analytic way when such
 non-shear steady states or 
 traveling waves  exist. The first step towards this direction is to construct nonlinear superpositions of traveling waves as in the Vlasov-Poisson case.

When there is no Coriolis force, long time dynamics near monotone flows is relatively rigid in strong topology, while it is still highly non-trivial to give a mathematical confirmation. A first step is to understand  the linearized equation. Orr \cite{Orr} observed the linear damping for  Couette flow, and Case  predicted   the decay of velocity for monotone shear
flows. Recently,  their predictions are confirmed in \cite{LZ,Zillinger-1,Zillinger-2,WZZ-1,GNRS,Jia2020-1,Jia2020-2} and are extended to non-monotone flows in \cite{WZZ-2,WZZ-3}. Meanwhile,  great progress has been made in the study of nonlinear dynamics near shear flows. Bedrossian and Masmoudi \cite{BM15} proved nonlinear inviscid damping near  Couette flow for the initial perturbation in some  Gevrey space on $\mathbb{T}\times\mathbf{R}$.  Ionescu and  Jia \cite{Ionescu-Jia20} extended the  above  asymptotic stability to a periodic finite channel $\mathbb{T}\times[-1,1]$ under the assumption that the initial vorticity perturbation is compacted supported in the interior of the channel. Later, nonlinear inviscid damping was proved near a class of Gevrey smooth monotone shear
flows  in a periodic finite channel if the perturbation is
taken in a suitable Gevrey space, where  $u''(y)$ is compactly supported
\cite{Ionescu-Jia2001,Masmoudi-Zhao2001}. It is still challenging to study the long time behavior near general, rough,
monotone  or non-monotone shear flows. On the other hand, inviscid damping (to a shear flow) depends on the regularity of the perturbation, and the existence of non-shear stationary structures is shown near some specific flows. Lin and Zeng  \cite{LZ} found cats' eyes flows near Couette for $H^{<{3\over2}}$  vorticity perturbation in a periodic  finite channel, while no non-shear  traveling waves near Couette exist if the regularity is $H^{>{3\over2}}$,   in contrast to the linear level, where  damping is always true for any initial vorticity in $L^2$. For  Kolmogorov flows, which is non-monotone,
 Coti Zelati,  Elgindi and Widmayer \cite{CEW20} constructed non-shear stationary states near Kolmogorov at analytic regularity on the square torus, while there are no nearby non-shear steady states
at regularity $H^{3}$ for velocity on a rectangular torus.
They also proved that any traveling wave  near Poiseuille must be shear for $H^{>5}$  vorticity perturbation in a periodic  finite channel.

As indicated in \cite{Pedlosky1987}, the study of the dynamics of large-scale oceanic or atmospheric motions must include the Coriolis force  to be geophysically
relevant, and once the Coriolis force  is included a host of subtle and fascinating dynamical phenomena are possible.
By numerical computation,
Kuo \cite{Kuo1974} found the boundary of  barotropic instability for the sinus profile, which is far from  linear instability in no Coriolis case. Later, based on Hamiltonian index theory and spectral analysis,  Lin, Yang and Zhu  theoretically confirmed large parts of the boundary and corrected the rest. New traveling waves, which are purely due to the Coriolis effects,
are found near the sinus profile \cite{LYZ}. Barotropic instability of other geophysical  shear flows has also attracted much attention. For instance, by looking for the neutral solutions, most of
the stability boundary, which is again different from   no Coriolis case, of bounded and unbounded Bickley jet is found  numerically and analytically in  \cite{Lipps,Howard-Drazin,Maslowe,Balmforth-Piccolo,Engevik}.  More fruitful geophysical fluid dynamics, such as Rossby wave and baroclinic instability, could be found in \cite{Kuo1974,Pedlosky1987,Drazin,Majda,McWilliams}. On the other hand, similar to  no Coriolis case, linear inviscid damping is still true for a large class of flows and moreover,  the same decay estimates of the velocity can be obtained  for a class of monotone flows \cite{WZZ}.
 Elgindi and Widmayer \cite{EW} viewed  Coriolis effect as one mechanism helping to stabilize the
motion of an ideal fluid, and proved the almost global stability of the zero solution  for
the $\beta$-plane equation.  Global stability of the zero solution is further to be confirmed in \cite{PW}.

When  Coriolis force is involved,  long time dynamics near a   shear flow becomes fruitful. One of the main reasons is that, comparing with no Coriolis case,  there are new traveling waves  with  fluid trajectories moving in one direction.
 This paper is devoted to studying  the number of such  traveling wave families near a general shear flow $u$   under  the influence of  Coriolis force.
Here, a traveling wave family  roughly includes the sets of nearby traveling waves with traveling speeds converging to a same number  outside the  range of the flow, see Definition \ref{traveling wave family} for details.
  Precisely, we
 prove that if the flow $u$ has a  critical point at which $u$ attains its minimal (resp. maximal) value,  then a unique transitional $\beta$ value  $\beta_{+}$
 (resp. $\beta_{-}$)  exists in the positive (resp. negative) half-line, through which the number of traveling wave families changes suddenly from
 finite to infinite.
The  transitional $\beta$ values are defined in (\ref{def-beta+})--(\ref{def-beta-}).
   If
 the flow $u$ has no such  critical  points, then  the number of
 traveling wave families is always finite   for positive (resp. negative) $\beta$ values.
  This is true  for general shear flows under  mildly technical assumptions.
 Based on Hamiltonian structure and index theory, we unconditionally prove the above results for a flow in class $\mathcal{K}^+$, which is
 defined as follows.
\begin{definition}\label{def-flow-in-class K+}
A flow $u$ in class $\mathcal{K}^+$ means that $u\in H^4(y_1,y_2)$
is not a constant function, and for any $\beta\in\text{Ran} (u'')$, there exists $u_\beta\in\text{Ran} (u)$ such that $K_\beta=(\beta-u'')/(u-u_\beta)$
is positive and bounded on $[y_1,y_2]$.
\end{definition}

A typical example of such a flow is a cosine jet $u(y)={\frac{1+\cos(\pi y)}{2}},$ $ y\in\lbrack-1,1]$ (i.e. the sinus profile), which was studied in  geophysical literature \cite{Kuo1949,Kuo1974,Pedlosky1987}.
For $\beta=0$ and a general shear flow $u\in C^2([y_1,y_2])$,  Rayleigh \cite{Rayleigh1880} gave  a necessary  condition for spectral instability that $u''(y_0)=0$ for some $y_0\in(y_1,y_2)$, and even under this condition, Fj{\o}rtoft \cite{Fjortoft} provided a sufficient  condition for spectral stability  that     $(u-u(y_0))u''\geq0$  on $(y_1,y_2)$.
For $\beta\neq0$ and $u\in C^2([y_1,y_2])$,   the above two conditions   can be extended  as $\beta-u''(y_\beta)=0$ for some  $y_{\beta}\in(y_1,y_2)$   and  $(\beta-u'')(u-u(y_\beta))\leq 0$ on $(y_1,y_2)$, respectively. See, for example, (6.3)--(6.4) in \cite{Kuo1974}. For a flow in class $\mathcal{K}^+$, the extended Rayleigh's  condition implies that $\beta\in\text{Ran} (u'')$ is necessary for spectral instability, but the flow does not satisfy the extended Fj{\o}rtoft's  sufficient condition for spectral stability. The sharp condition for spectral stability indeed depends on $\beta$ and the wave number $\alpha$, which was obtained in \cite{Lin03} for $\beta=0$  and  in \cite{LYZ} for $\beta\neq0$.

Consider a class of general shear  flows satisfying
\beno
(\textbf{H1}) \quad u\in H^{4}(y_{1},y_{2}), \,\,  u''\neq0  \,\, \text{on}\,\, u\text{'s}\,\,  \text{critical level}\; \{u'=0\}.
\eeno
A flow $u$  in class $\mathcal{K}^+$ satisfies
the assumption ({\bf{H1}}).
In fact, it is trivial for $0\notin \text{Ran}(u'')$; if $0\in \text{Ran}(u'')$ and  $y_0\in[y_1,y_2]$ satisfies $u'(y_0)=0$ and $u''(y_0)=0$, then $u(y_0)-u_0=-{1\over K_0(y_0)}u''(y_0)=0$.
Thus, $\varphi\equiv u-u_0$ solves $\varphi''+K_0\varphi=0$, $\varphi(y_0)=\varphi'(y_0)=0$. Then $u\equiv u_0$, which is a contradiction.

To state
  our main results with few restriction, we first consider  flows in class $\mathcal{K}^+$, and left the extension to general shear flows satisfying
  ({\bf{H1}}) in Section 2.

\begin{theorem}\label{traveling wave construction-classK+}
Let 
$\beta\neq0$  and the flow $u$ be in class $\mathcal{K}^+$.
\begin{itemize}
\item[(1)] If
$\{u'=0\}\cap\{u=u_{\min}\}\neq\emptyset$, then
there exists $\beta_+\in(0,\infty)$ such that there exist at most finitely  many  traveling wave families  near $(u,0)$ for  $\beta\in(0,\beta_+]$, and infinitely  many traveling wave families  near $(u,0)$ for $\beta\in(\beta_+,\infty)$. Moreover, $\beta_+$ is specified in
\eqref{def-beta+}.
\item[(2)]
If
$\{u'=0\}\cap\{u=u_{\max}\}\neq\emptyset$,
 then
there exists $\beta_-\in(-\infty,0)$ such that there exist at most finitely  many  traveling wave families  near $(u,0)$ for  $\beta\in[\beta_-,0)$, and infinitely  many traveling wave families  near $(u,0)$ for $\beta\in(-\infty,\beta_-)$. Moreover,
$\beta_-$ is specified in
\eqref{def-beta-}.
\end{itemize}
\begin{itemize}
\item[(3)] If
$\{u'=0\}\cap\{u=u_{\min}\}=\emptyset$, then    there exist at most finitely many traveling wave families  near  $(u,0)$ for   $\beta\in(0,\infty)$.
\item[(4)]
 If
$\{u'=0\}\cap\{u=u_{\max}\}=\emptyset$, then  there exist at most finitely  many traveling wave families  near  $(u,0)$ for   $\beta\in(-\infty,0)$.
\end{itemize}
Here,
 the precise description of a traveling wave family  near $(u,0)$ is given in Definition $\ref{traveling wave family}$.
\end{theorem}

Unless otherwise specified,  ``near $(u,0)$" always means ``in a (velocity) $H^3$ neighborhood of $(u,0)$" in Theorem \ref{traveling wave construction-classK+} and the rest of this paper, as indicated in Definition \ref{a set of traveling wave solutions}.
These traveling wave families do not exist if there is no  Coriolis force.
By Theorem \ref{traveling wave construction-classK+}, Coriolis force and its magnitude indeed  bring fascinating dynamics near the shear flow.
  On the one hand,
 for flows having no  critical point which is meanwhile a minimal point, the number of traveling wave families is always finite  no matter how much magnitude of Coriolis force, which is   a mild Coriolis effect.
 On the other hand,
 for flows having  such a critical  point, there is a surprisingly sharp difference, namely, when the  Coriolis parameter   passes through the transitional point $\beta_{+}$, the number of traveling wave families  changes  suddenly from finite to infinite. In particular, quasi-periodic solutions to (\ref{Euler equation})--(\ref{boundary condition for euler}) can be expected near the shear flow for  $\beta\in(0,\beta_+]$, while
almost periodic solutions potentially exist  for  $\beta\in(\beta_+,\infty)$.
 This could be regarded as  a strong Coriolis effect and predicts chaotic long time dynamics near these flows.

 The same dynamical phenomena are true for  general shear flows under some mildly technical assumptions.  The explicit result is stated in Theorem \ref{traveling wave construction}. For $\beta>0$, the  technical assumption for flows having a critical and meanwhile minimal point is that $u_{\min}$ is not an embedding eigenvalue of the linearized Euler operator for small wave numbers. The  assumption for flows having no such critical points is some regularized  condition near the endpoints of $u$. Note that the first spectral assumption has only restriction for one  point $u_{\min}$, no matter whether the interior of $\text{Ran}\,(u)$ has embedding eigenvalues.
 The second assumption is more generic and quite easy to verify. Both the two technical assumptions are used for ruling out eigenvalues' oscillation for Rayleigh-Kuo  boundary value problem (BVP) as the parameter $c$ tends to $u_{\min}$, see Subsection 2.2 for details.

 Let us give some remarks on properties of such traveling waves near the flow $u$.
\begin{itemize}
\item  The traveling waves
 have fluid trajectories moving in one direction, see \eqref{horizontal-velocity does not change sign} in the proof of Lemma \ref{traveling wave construction-lemma}. Thus unlike  the  constructed steady flow near Couette flow in \cite{LZ}, the streamlines here have
 no cat's eyes structure.
\item
The traveling waves can be constructed near a smooth shear flow for $H^{\geq{3}}$ (including $H^{>6}$) velocity perturbation when the Coriolis parameter is large, see Corollary \ref{traveling wave construction-corollary}.
In contrast, in the case of    no Coriolis force,  no traveling waves  could  be found  near Couette flow   for $H^{>{5\over2}}$  velocity  perturbation \cite{LZ} and near Poiseuille  flow  for $H^{>6}$  velocity   perturbation \cite{CEW20}.
\item
Let $\{u'=0\}\cap\{u=u_{\min}\}\neq \emptyset$ and $\beta>{9\over8}\kappa_+$.
The directions of  vertical velocities of the nearby traveling waves
might
change  frequently  with small  amplitude as the traveling speeds converge to $u_{\min}^-$, see  Remark \ref{asymptotic behavior of vertical velocities}.
\end{itemize}

We  apply the main results to analytic monotone flows (including Couette flow) and the sinus profile. For  an analytic monotone  flow, there exist at most finitely many nearby traveling wave families   for   $\beta\neq0$, see  Corollary \ref{thm-analytic-monotone flow}.
  For the  sinus profile, as   mentioned above, it is in class $\mathcal{K}^+$, and so  applying Theorem \ref{traveling wave construction-classK+} (1)-(2) we get that $\{u'=0\}\cap\{u=u_{\min}\}=\{\pm1\}$, $\{u'=0\}\cap\{u=u_{\max}\}=\{0\}$,  $\beta_+={9\over16}\pi^2$, $\beta_-=-{1\over2}\pi^2$,  there exist at most finitely  many  traveling wave families  near the sinus profile for  $\beta\in[-{1\over2}\pi^2,{9\over16}\pi^2]$, and infinitely  many nearby traveling wave families  for $\beta\notin[-{1\over2}\pi^2,{9\over16}\pi^2]$. Moreover, we will give a systematical study on  the  number  of  isolated real eigenvalues of  the linearized Euler operator and traveling wave families near the sinus profile on the whole  $(\alpha,\beta)$'s region in Section 7 (here $\alpha$ is the wave number in the $x$-direction), which plays an
  important role in further study on its  long time dynamics. We make a  comparison with  the  previous work in \cite{LYZ}. By Theorem \ref{the explicit number of traveling wave families}, the number of isolated real eigenvalues of  the linearized Euler operator  (i.e. non-resonant modes)
  determines that of traveling wave families.
 The explicit  number of isolated real eigenvalues  in the region $(\alpha,\beta)\in(0,\infty)\times[-{\pi^2\over2},{\pi^2\over2}]$ can be obtained in  \cite{LYZ}, but no information can be concluded outside this region, see the discussion below Figure 4 in  \cite{LYZ}. Our new contribution for the sinus profile in this paper is that we  calculate  the  explicit  number of isolated real eigenvalues in the remaining area $(\alpha,\beta)\in(0,\infty)\times(-\infty,-{\pi^2\over2})\cup({\pi^2\over2},\infty)$, and thus completely get the number of   traveling wave families near the sinus profile on the whole  $(\alpha,\beta)$'s region.
 For the sinus profile,  the  novelty is that we give the asymptotic behavior of the $n$-th eigenvalue $\lambda_n(c)$ of the Rayleigh-Kuo BVP \eqref{sturm-Liouville} as $c\to 0^-$ for $\beta\in({\pi^2\over2},\infty)$ and as $c\to 1^+$ for $\beta\in(-\infty,-{\pi^2\over2})$, from which we find the transitional $\beta$ values such that the number of traveling wave families changes suddenly from finite
to infinite. For general shear flows satisfying ${\bf (H1)}$, the key   is to study whether $\lambda_n(c)$ is  unbounded from  below as $c$ is close to $u_{\min}$ (or $u_{\max}$) in Theorem \ref{eigenvalue asymptotic behavior-bound} and to rule out the oscillation of $\lambda_n(c)$ in Theorems \ref{main-result}-\ref{number for sinus flow}. In this paper, we focus on the
description of the eigenvalues of the Rayleigh-Kuo BVP \eqref{sturm-Liouville}, which in turn, by Theorem \ref{the explicit number of traveling wave families}, yields information  on traveling wave families.

The rest of this paper is organized as follows. In Section $2$, we extend  Theorem \ref{traveling wave construction-classK+}  to general shear flows and give the outline of the proof.  In Sections 3-4, we study the
 asymptotic behavior of the $n$-th eigenvalue of Rayleigh-Kuo BVP, where we determine the transitional  values for the $n$-th eigenvalue of Rayleigh-Kuo BVP in Section 3, and rule out oscillation of the $n$-th eigenvalue in Section 4.
In Section $5$, we establish the correspondence between a traveling wave family  and an isolated real eigenvalue of the linearized Euler operator.
In Section $6$, we prove the main Theorems \ref{traveling wave construction} and  \ref{traveling wave construction-classK+}.
As a concrete application, we thoroughly study the number of traveling wave families near the sinus profile in the last section.

\vspace{0.2cm}

\no\textbf{Notation}

\vspace{0.2cm}

We provide the notations that we use in this
paper.  Let $u_{\min}=\min(u)$ and $u_{\max}=\max(u)$ for $u\in C([y_1,y_2])$. For    a shear  flow $u$ satisfying
$
(\textbf{H1})$, we use the following characteristic quantities of the flow.
  If $\{u'=0\}\cap\{u=u_{\min}\}\neq\emptyset,$ we define
\begin{align}
\label{def-kappa+}
\kappa_+&:=
		\min \{u''(y)|y\in [y_1,y_2]\text{ such that } u'(y)=0\text{ and } u(y)=u_{\min}\}.
\end{align}
If $\{u'=0\}\cap\{u=u_{\max}\}\neq\emptyset,$ we define
\begin{align}
\label{def-kappa-}
 \kappa_-&:=
		\max \{u''(y)|y\in[y_1,y_2]\text{ such that } u'(y)=0\text{ and } u(y)=u_{\max}\}.
\end{align}
Note that $\kappa_+\in (0,\infty)$ and $\kappa_-\in (-\infty,0)$ in  \eqref{def-kappa+}--\eqref{def-kappa-}. In fact,
${\bf (H1)}$ implies $u''(y_0)>0$ for  $y_0\in A:=\{y\in [y_1,y_2]| u'(y)=0\text{ and } u(y)=u_{\min}\}$. Then $y_0$ is an isolated point of $A$. Thus, $A$ is a finite set and $\kappa_+\in (0,\infty)$ in \eqref{def-kappa+}. Similarly, $\kappa_-\in (-\infty,0)$ in \eqref{def-kappa-}.
Besides \eqref{def-kappa+}--\eqref{def-kappa-}, we define
\begin{align}
\label{def-kappa-infty}
\kappa_+:=
		\infty,\quad \text{ if }&\{u'=0\}\cap\{u=u_{\min}\}=\emptyset, \\
 \label{def-kappa-negative-infty}\kappa_-:=-\infty, \;\text{ if }&\{u'=0\}\cap\{u=u_{\max}\}=\emptyset.
 \end{align}
If $\{u=u_{\min}\}\cap(y_1,y_2)\neq \emptyset,$ we define
\begin{align}
\label{def-mu+}
\mu_+&:=\min\{ u''(y)|y\in(y_1,y_2)\text{ such that } u(y)=u_{\min}\}.
\end{align}
If $\{u=u_{\max}\}\cap(y_1,y_2)\neq \emptyset,$ we define
\begin{align}
\label{def-mu-}
\mu_-&:=\max\{ u''(y)|y\in(y_1,y_2)\text{ such that }u(y)=u_{\max}\}.
\end{align}
Note that $\mu_+\in[\kappa_+,\infty)$ and $\mu_-\in (-\infty,\kappa_-]$ in \eqref{def-mu+}--\eqref{def-mu-}.
Then we define
\begin{align}\label{def-beta+}
\beta_+:=
\begin{cases}
\min\{{9\over8}\kappa_+,\mu_+\},  &
\text{if } \{u=u_{\min}\}\cap(y_1,y_2)\neq\emptyset,\\
{9\over8}\kappa_+, &
\text{if } \{u=u_{\min}\}\cap(y_1,y_2)=\emptyset,
\end{cases}\end{align}
and
\begin{align}\label{def-beta-}
\beta_-:=
\begin{cases}
\max\{{9\over8}\kappa_-,\mu_-\},  &
\text{if } \{u=u_{\max}\}\cap(y_1,y_2)\neq\emptyset,\\
{9\over8}\kappa_-, &
\text{if } \{u=u_{\max}\}\cap(y_1,y_2)=\emptyset.
\end{cases}
\end{align}
We denote
\begin{align}\label{E+}
(\textbf{E$_+$}) \quad \quad\quad u_{\min} \text{ is not an embedding eigenvalue of }\mathcal{R}_{\alpha,\beta},\\\label{E-}
(\textbf{E$_-$}) \quad \quad\quad u_{\max} \text{ is not an embedding eigenvalue of }\mathcal{R}_{\alpha,\beta},
\end{align}
where  $\mathcal{R}_{\alpha,\beta}$ is defined in \eqref{defRab}.
Moreover, we define
\begin{align}
\label{def-m0finitebetapositive}
m_{\beta}&:=\left\{\begin{array}{ll}
\sharp\{a\in(y_1,y_2)|u(a)=u_{\min},u''(a)-\beta<0\},&\text{if}\;0<\beta\leq{9\over8}\kappa_+,\\
		\sharp\{a\in(y_1,y_2)|u(a)=u_{\max},u''(a)-\beta>0\},&\text{if}\;{9\over8}\kappa_-\leq\beta<0,
		\end{array}\right.
\end{align}
and
\begin{align}
\label{def-M0finitebetapositive}
M_{\beta}&:=\left\{\begin{array}{ll}
-\inf\limits_{c\in(-\infty,u_{\min})}\lambda_{m_{\beta}+1}(c),&\text{if}\;0<\beta\leq{9\over8}\kappa_+,\\
-\inf\limits_{c\in(u_{\max},\infty)}\lambda_{m_{\beta}+1}(c),&\text{if}\;{9\over8}\kappa_-\leq\beta<0,
		\end{array}\right.
\end{align}
where $\lambda_{m_{\beta}+1}(c)$ is the $(m_{\beta}+1)$-th eigenvalue of the Rayleigh-Kuo  BVP
\eqref{sturm-Liouville}.

$\mathbf{R}$, $\mathbf{Z}$ and $\mathbf{Z}^+$  denote the set of all the real numbers, integers and positive integers, respectively.
$\sharp(K)$ or $\sharp\,K$  is the cardinality of the set $K$.
Let  $L$ be a linear operator from  a Banach space $X$  to $X$.  $X^*$ is the dual space of $X$.
$\sigma(L)$,  $\sigma_e(L)$ and $\sigma_d(L)$ are the spectrum, essential spectrum and discrete spectrum of the operator $L$, respectively.
For $\psi\in L^2(D_T)$,   the Fourier transform of $\psi$ in $x$ is denoted by $\widehat \psi$.

\section{Extension to general shear flows and  outline of the proof}

In this section, we first extend the main Theorem \ref{traveling wave construction-classK+} to  general shear flows under mild assumptions, and then discuss our approach  in its proof.

\subsection{Main results for general shear flows}
For a shear flow in $H^4(y_1,y_2)$, we give the exact number of traveling wave families near the flow.
\begin{theorem}\label{the explicit number of traveling wave families}
 Let $\alpha=2\pi/T$, $\beta\neq 0$ and $u\in H^4(y_1,y_2)$. Then
$\sharp(\bigcup_{k\geq1}(\sigma_d(\mathcal{R}_{k\alpha,\beta})\cap\mathbf{R}))$ is exactly  the number of traveling wave families near $(u,0)$, where
$\mathcal{R}_{k\alpha,\beta}$ is defined in \eqref{defRab} and the precise description of a traveling wave family  near $(u,0)$ is given in Definition
$\ref{traveling wave family}$.
\end{theorem}
Then we state our main theorem for a shear flow satisfying $
(\textbf{H1})$.
\begin{theorem}\label{traveling wave construction}
Let $\beta\neq0$  and $u$ satisfy ${\bf (H1)}$.
\begin{itemize}
\item[(1)] If
$\{u'=0\}\cap\{u=u_{\min}\}\neq\emptyset$
and  $\bf{(E_+)}$ holds for every $\alpha\in(0, \sqrt{M_{\beta}}]\cap\{{2k\pi\over T}|k\in\mathbf{Z}^+\}$ and $\beta\in(0,{9\over8}\kappa_+)$, then
there exists $\beta_+\in(0,\infty)$ such that there exist at most finitely  many  traveling wave families  near $(u,0)$ for  $\beta\in(0,\beta_+)$, and infinitely  many traveling wave families  near $(u,0)$ for $\beta\in(\beta_+,\infty)$, where
$\kappa_+$, $\bf{(E_+)}$
 and $ M_{\beta}$ are defined in \eqref{def-kappa+}, \eqref{E+} and \eqref{def-M0finitebetapositive}, respectively. Moreover, $\beta_+$  is specified in \eqref{def-beta+}.
\item[(2)]
If
$\{u'=0\}\cap\{u=u_{\max}\}\neq\emptyset$
and  $\bf{(E_-)}$ holds  for every $\alpha\in(0, \sqrt{M_{\beta}}]\cap\{{2k\pi\over T}|k\in\mathbf{Z}^+\}$ and $\beta\in({9\over8}\kappa_-,0)$,
 then
there exists $\beta_-\in(-\infty,0)$ such that there exist at most finitely  many  traveling wave families  near $(u,0)$ for  $\beta\in(\beta_-,0)$, and infinitely  many traveling wave families  near $(u,0)$ for $\beta\in(-\infty,\beta_-)$,
where
 $\kappa_-$ and
$\bf{(E_-)}$
are defined in \eqref{def-kappa-} and \eqref{E-}, respectively. Moreover, $\beta_-$  is specified in \eqref{def-beta-}.
\end{itemize}
Assume that $u(y_1)\neq u(y_2)$ and for $i=1,2$, there exist $ \delta>0,\ C>0$ and $m_i> 0$ such that $(\rm{i})$ $u''(y)=\beta_i$ for $y\in(y_i-\delta,y_i+\delta)\cap[y_1,y_2]$
or $(\rm{ii})$ $C^{-1}|y-y_i|^{m_i}\leq |u''(y)-\beta_i|\leq C|y-y_i|^{m_i}$ for $y\in(y_i-\delta,y_i+\delta)\cap[y_1,y_2]$
or $(\rm{iii})$ $ \beta_iu'(y_i)(-1)^i\geq 0$, where $\beta_i=u''(y_i)$.
\begin{itemize}
\item[(3)] If
$\{u'=0\}\cap\{u=u_{\min}\}=\emptyset$, then    there exist at most finitely many traveling wave families  near  $(u,0)$ for   $\beta\in(0,\infty)$.
\item[(4)]
 If
$\{u'=0\}\cap\{u=u_{\max}\}=\emptyset$, then  there exist at most finitely  many traveling wave families  near  $(u,0)$ for   $\beta\in(-\infty,0)$.
\end{itemize}
Here, the precise description of a traveling wave family  near $(u,0)$ is given in Definition $\ref{traveling wave family}$.
\end{theorem}

As mentioned in Introduction, $\bf{(E_+)}$ or $\bf{(E_-)}$ is  ``one spectral point" assumption for small wave numbers. Note that if ${2\pi\over T}>\sqrt{M_{\beta}}$, then $(0, \sqrt{M_{\beta}}]\cap\{{2k\pi\over T}|k\in\mathbf{Z}^+\}=\emptyset$, and $\bf{(E_{\pm})}$ is not needed in Theorem \ref{traveling wave construction} (1)--(2). One of
the conditions (i)--(iii) is the ``good''   endpoints assumption  and rather generic.
For example, if $u\in C^{m}([y_1,y_2])$, $m\geq 3$ and $ u^{(k_i)}(y_i)\neq 0$ for some $ 3\leq k_i\leq m$, then $(\rm{ii})$ is true for $m_i=k_i-2$. Thus, for analytic flows, (ii) holds if $ u^{(k_i)}(y_i)\neq 0$ for some $  k_i\geq 3$ and (i) holds otherwise.
Applying  Theorem \ref{traveling wave construction} (3)-(4) to analytic monotone flows, we have the following result.
\begin{corol}\label{thm-analytic-monotone flow}
Let $u$ be an analytic monotone flow: $u'(y)\neq 0$ for $y\in[y_1,y_2]$. Then
 there exist at most finitely many traveling wave families  near  $(u,0)$  for  $\beta\neq0$.
\end{corol}
\subsection{Outline and our approach in the proof}
Non-parallel steady flows or traveling
waves may be bifurcated from a shear flow if the linearized Euler operator has an
embedding or isolated real eigenvalues \cite{BD01,LZ,LYZ}. Based on the existence of an embedding eigenvalue for a class of monotone shear flows near Couette flow, cat's eyes steady states
are  bifurcated  from  these flows \cite{LZ}. When the Coriolis force is involved,     non-parallel traveling waves
are bifurcated from the sinus profile on account of the existence of an isolated real eigenvalue \cite{LYZ}. The traveling  speeds lie outside the
range of the sinus profile and are contiguous to the isolated real eigenvalue.
Now, we consider such bifurcation theorem for general shear flows, namely, using an isolated real eigenvalue of the linearized Euler operator, we prove that
such traveling waves can be bifurcated from general shear flows. We use the following concept.
\begin{definition}\label{a set of traveling wave solutions}
$\{\vec{u}_{\varepsilon}\left(
x-c_{\varepsilon}t,y\right)  =(  u_{\varepsilon}\left(  x-c_{\varepsilon
}t,y\right), v_{\varepsilon}\left(  x-c_{\varepsilon}t,y\right)  ) |\varepsilon\in(0,\varepsilon_0)\text{ for some }\varepsilon_0>0\}$ is called a set of traveling wave solutions  near $(u,0)$ with traveling speeds converging to $ c_0$, if for each $\varepsilon\in(0,\varepsilon_0)$,
$ \vec{u}_{\varepsilon}\left(
x-c_{\varepsilon}t,y\right)  =(  u_{\varepsilon}\left(  x-c_{\varepsilon
}t,y\right), v_{\varepsilon}\left(  x-c_{\varepsilon}t,y\right)  ) $ is a
traveling wave solution
to \eqref{Euler equation}--\eqref{boundary condition for euler} which has period $T$ in $x$ such that
\begin{align}\label{traveling wave construction-lemma-H3}
 \|(u_\varepsilon,v_\varepsilon)-(u,0)\|_{H^3(D_T)} \leq\varepsilon,
\end{align}
 $\|v_{\varepsilon}\|_{L^2\left(  D_T\right)}\neq0$,
 $c_{\varepsilon}\notin \text{Ran}(u)$ and $c_{\varepsilon}\rightarrow c_0$.
\end{definition}

Then we give the bifurcation result for general shear flows.
\begin{lemma}\label{traveling wave construction-lemma}
 Let $\alpha=2\pi/T$, $\beta\neq 0$ and $u\in H^4(y_1,y_2)$.
 Assume that $c_0\in\bigcup_{k\geq1}(\sigma_d(\mathcal{R}_{k\alpha,\beta})\cap\mathbf{R})$, where $\mathcal{R}_{k\alpha,\beta}$ is defined in \eqref{defRab}.
Then there
exists a set of traveling wave solutions  near $(u,0)$ with traveling speeds converging to $ c_0$.
Moreover, we have
$u_{\varepsilon}\left(  x,y\right)-c_\varepsilon  \neq0$.
\end{lemma}

Here, we mention some differences  from   the construction of traveling waves in the literature.
First, the  horizontal period of  constructed traveling waves in Proposition 7 of \cite{LYZ} is not the given period  $T$, and for the sinus profile, the  period of  traveling waves is modified to $T$  by  adjusting the traveling speed in Theorem 7 of \cite{LYZ}. But the price is an additional condition, namely, the isolated eigenvalue $c_0$ can not be
an extreme point of $\lambda_{1}$ (i.e. $\alpha_0\neq \sqrt{\Lambda_\beta}$ in Theorem 7 (ii) of \cite{LYZ}), where  $\lambda_{n}$ is the $n$-th eigenvalue of
  \eqref{sturm-Liouville}.
  In Lemma \ref{traveling wave construction-lemma},
 we can construct traveling waves for general flows no matter whether $c_0$ is an extreme point  of $\lambda_{n}$,
and thus improve the result  in Theorem 7 of \cite{LYZ} even for the sinus profile.
Second, it is possible that $c_0\in\sigma_d(\mathcal{R}_{0,\beta})\cap\mathbf{R}$, which makes it subtle to guarantee that  the bifurcated solutions near the flow $u$ is not a shear flow.
 Thus, the extension of the bifurcation result for the sinus profile in \cite{LYZ} to general shear flows in Lemma \ref{traveling wave construction-lemma} is still non-trivial,  since we have to treat the unsolved case  that  $c_0$ is  an extreme point of $\lambda_{n_0}$ for some  $n_0\in\mathbf{Z}^+$ or $c_0\in\sigma_d(\mathcal{R}_{0,\beta})\cap\mathbf{R}$.
 To overcome the difficulty, we carefully modify the flow $u$ to a suitable shear flow, which satisfies that $\lambda_{n_0}$ is locally monotone near $c_0$ and    $c_0\notin\sigma_d(\mathcal{R}_{0,\beta})\cap\mathbf{R}$, and then study the  bifurcation at the suitable shear flow. Finally,
the minimal horizontal periods of constructed traveling waves are possibly  less than ${2\pi/\alpha}$ if $c_0\in\sigma_d(\mathcal{R}_{\alpha,\beta})\cap\mathbf{R}.$
In fact, the  Sturm-Liouville operator $\mathcal{L}_{c_0}$ could indeed have more than one  negative eigenvalues (e.g., if $\kappa_+<\infty$, $\beta>{9\kappa_+/8}$ and $c_0$ is close to $u_{\min}$), where $\mathcal{L}_{c_0}$ is defined in \eqref{sturm-Liouville}.
In this case, we  give sufficient condition to guarantee that  the minimal period  is  ${2\pi/\alpha}$ in Lemma
\ref{traveling wave concentration ad}. In contrast,
the minimal period must be ${2\pi/\alpha}$  in Theorem 5.1 of \cite{Li-Lin} and Theorem 1 of \cite{LZ}, since the corresponding Sturm-Liouville operator has only one negative eigenvalue.

Since the isolated real eigenvalue $c_0$ lies outside the range of the flow $u$, by a similar proof of Lemma \ref{traveling wave construction-lemma} we can improve the regularity of traveling waves as follows.
\begin{corol}\label{traveling wave construction-corollary}
 Let $\alpha=2\pi/T$, $\beta\neq 0$, $u\in C^\infty([y_1,y_2])$ and $s\geq3$.
 Assume that $c_0\in\bigcup_{k\geq1}(\sigma_d(\mathcal{R}_{k\alpha,\beta})\cap\mathbf{R})$, where $\mathcal{R}_{k\alpha,\beta}$ is defined in \eqref{defRab}.
Then the conclusion in Lemma \ref{traveling wave construction-lemma} holds true with \eqref{traveling wave construction-lemma-H3} replaced by
$
 \|(u_\varepsilon,v_\varepsilon)-(u,0)\|_{H^s(D_T)} \leq\varepsilon.
$
\end{corol}

One naturally asks whether the assumption $c_0\in\bigcup_{k\geq1}(\sigma_d(\mathcal{R}_{k\alpha,\beta})\cap\mathbf{R})$ in Lemma \ref{traveling wave construction-lemma} is necessary.
By studying the asymptotic behavior of traveling speeds and $L^2$ normalized vertical velocities for nearby traveling waves, we confirm that it is true.

\begin{lemma}\label{traveling wave concentration}
Let $\alpha=2\pi/T$, $\beta\neq 0$ and $u\in H^4(y_1,y_2)$.
Assume that
$\{\vec{u}_{\varepsilon}\left(
x-c_{\varepsilon}t,y\right)  =(  u_{\varepsilon}\left(  x-c_{\varepsilon
}t,y\right), v_{\varepsilon}\left(  x-c_{\varepsilon}t,y\right)  ) |\varepsilon\in(0,\varepsilon_0)\}$ is a set of traveling wave solutions  near $(u,0)$
 with traveling speeds converging to $ c_0$.
  Then
 $c_0\in\bigcup_{k\geq1}(\sigma_d(\mathcal{R}_{k\alpha,\beta})\cap\mathbf{R})\cup\{u_{\min},u_{\max}\},$ where $\mathcal{R}_{k\alpha,\beta}$ is defined in \eqref{defRab}. Moreover,
if $c_0\in\bigcup_{k\geq1}(\sigma_d(\mathcal{R}_{k\alpha,\beta})\cap\mathbf{R})$, then
there exists  $\varphi_{c_0}\in\ker(\mathcal{G}_{c_0})$ such that
 \begin{align}\label{eigenfunction-convergence}
 {\tilde{v}_{\varepsilon}}\longrightarrow\varphi_{c_0} \ \ \text{in} \ \ {H^{2}\left(D_T\right)  },
 \end{align}
 where the operator $\mathcal{G}_{c_0}$ is defined by
\begin{align}\label{Glinearized elliptic operator}
 \mathcal{G}_{c_0}=-\Delta-\frac{\beta
-u^{\prime\prime}\left(  y\right)  }{u\left(  y\right)  -c_0}\; : \; H^2(D_T)\to L^2(D_T)
\end{align}with periodic boundary condition in $x$ and Dirichlet boundary condition in $y$, and
  $\tilde{v}_{\varepsilon}={{v}_{\varepsilon}/\|{v}_{\varepsilon}\|_{L^2\left( D_T\right)}}$.
\end{lemma}

The limit function $\varphi_{c_0}$ in
 Lemma
 \ref{traveling wave concentration} is a superposition of finite normal modes, see Remark \ref{The function varphi c0 superposition of finite normal modes}.
If $c_0\in\bigcup_{k\geq1}(\sigma_d(\mathcal{R}_{k\alpha,\beta})\cap\mathbf{R})$ in Lemma \ref{traveling wave concentration}, the vertical velocities of the nearby traveling waves  have simple asymptotic behavior as seen in \eqref{eigenfunction-convergence}.
However, if $c_0\in\{u_{\min},u_{\max}\}$, then the asymptotic behavior
 of vertical velocities might be complicated, see Remark \ref{asymptotic behavior of vertical velocities}.
The proofs of
Lemmas \ref{traveling wave construction-lemma}-\ref{traveling wave concentration} are given  in Section 5.

By Lemma
\ref{traveling wave concentration}, for any set  of traveling waves  near $(u,0)$ with traveling speeds converging to $ c_0$, $c_0$ must be an isolated real eigenvalue of the linearized Euler operator (besides $u_{\min}$ and $u_{\max}$).
By Lemma \ref{traveling wave construction-lemma}, every isolated real eigenvalue is contiguous  to the speeds of nearby traveling waves.
As the minimal periods of traveling waves in $x$ can be less than ${2\pi/\alpha}$, there might be two or more sets of  traveling wave solutions near $(u,0)$
 with traveling speeds converging to a same isolated real eigenvalue.
For example, if $((i+1)\alpha)^2=-\lambda_{n_i}(c_0)$, $\lambda_{n_i}$ is monotone near $c_0$ for $i=1,2$, and $(k\alpha)^2\neq \lambda_n(c_0)$ for $k\notin\{2,3\}$
 and $n\notin\{n_1,n_2\}$, then an application to Lemma \ref{traveling wave construction-lemma} (see Case 1 in its proof) gives two sets of traveling wave
 solutions, which has minimal periods $\pi/\alpha$ and $2\pi/ (3\alpha)$ respectively, near $(u,0)$ with traveling speeds converging to $c_0$. Moreover,
 traveling wave solutions could be bifurcated from nearby shear flows, which might induce more sets of  traveling wave solutions near $(u,0)$ with traveling
 speeds converging to $ c_0$.
This suggests us to  define a traveling wave family  near $(u,0)$ by an equivalence class as follows.
\begin{definition}\label{traveling wave family} A traveling wave family  near $(u,0)$
is defined  by an equivalence class under $\sim$, where
if $\{\vec{u}_{i,\varepsilon}\left(
x-c_{i,\varepsilon}t,y\right)=(  u_{i,\varepsilon}\left(  x-c_{i,\varepsilon
}t,y\right),v_{i,\varepsilon}\left(  x-c_{i,\varepsilon}t,y\right)  ) |\varepsilon\in(0,\varepsilon_i)\}$, $i=1,2$, are two sets of traveling wave solutions
 near $(u,0)$ with traveling speeds converging to  $ c_i\notin \text{Ran}(u)$,
  then
$\{\vec{u}_{1,\varepsilon}\left(
x-c_{1,\varepsilon}t,y\right)|\varepsilon\in(0,\varepsilon_1)\}$ and $\{\vec{u}_{2,\varepsilon}\left(
x-c_{2,\varepsilon}t,y\right)|\varepsilon\in(0,\varepsilon_2)\}$ are equivalent,  $\{\vec{u}_{1,\varepsilon}\left(
x-c_{1,\varepsilon}t,y\right)|\varepsilon\in(0,\varepsilon_1)\}\sim\{\vec{u}_{2,\varepsilon}\left(
x-c_{2,\varepsilon}t,y\right)|\varepsilon\in(0,\varepsilon_2)\}$, if
$ c_1=c_2.$
\end{definition}


By Lemma \ref{traveling wave concentration}, there exists $\varphi_{i}\in\ker(\mathcal{G}_{c_i})$ such that   ${\tilde{v}_{i,\varepsilon}}\longrightarrow\varphi_{i}$ {in}  ${H^{2}\left(
D_T\right)}$, where $\tilde{v}_{i,\varepsilon}={{v}_{i,\varepsilon}/\|{v}_{i,\varepsilon}\|_{L^2\left(  D_T\right)}}$ and ${v}_{i,\varepsilon}, \varepsilon\in(0,\varepsilon_i),$ are given in Definition \ref{traveling wave family}.
By Lemmas \ref{traveling wave construction-lemma} and \ref{traveling wave concentration}, we obtain the exact number of traveling wave families near a flow $u\in H^4(y_1,y_2)$ in Theorem \ref{the explicit number of traveling wave families}.

Thus, the number of isolated real eigenvalues of the linearized Euler operator plays an important role in counting the traveling wave families near the shear flow.
In terms of the stream function $\psi$, (\ref{linearized Euler equation}) can be written as
$
\partial_{t}\Delta\psi+u\partial_{x}\Delta\psi+(\beta-u^{\prime\prime})\partial_x\psi=0.
$
 By taking Fourier transform in $x$, we have
 \begin{align}\label{equation after Fourier transform}
 (\partial^2_y-\alpha^2)\partial_t\widehat{\psi}=i\alpha((u''-\beta)-u(\partial^2_y-\alpha^2))\widehat{\psi}.
 \end{align}
For  $\alpha>0$ and $\beta\in \mathbf{R}$,  the linearized Euler operator is given by
\begin{align}\label{defRab}
\mathcal{R}_{\alpha,\beta}:=-(\partial^2_y-\alpha^2)^{-1}((u''-\beta)-u(\partial^2_y-\alpha^2)).
\end{align}
Then (\ref{equation after Fourier transform}) becomes
$
-\frac{1}{ i\alpha}\partial_t\widehat{\psi}=\mathcal{R}_{\alpha,\beta}\widehat{\psi}.
$
Recall that $\sigma_e(\mathcal{R}_{\alpha,\beta})=\text{Ran} (u)$. Then the set of isolated real eigenvalues $\sigma_d(\mathcal{R}_{\alpha,\beta})\cap\mathbf{R}\subset(-\infty,u_{\min})\cup(u_{\max}, \infty)$. Moreover, it is well-known that
$\sigma_d(\mathcal{R}_{\alpha,\beta})\cap\mathbf{R}=\emptyset$ if $\beta=0$; $\sigma_d(\mathcal{R}_{\alpha,\beta})\cap(u_{\max},\infty)=\emptyset$ if $\beta>0$; and
$\sigma_d(\mathcal{R}_{\alpha,\beta})\cap(-\infty,u_{\min})=\emptyset$ if $\beta<0$, see \cite{Kuo1949,Tung1981,Pedlosky1987,LYZ}.
Therefore,  we only need to study $\sharp(\sigma_d(\mathcal{R}_{\alpha,\beta})\cap(-\infty,u_{\min}))$ for $\beta>0$ and $\sharp(\sigma_d(\mathcal{R}_{\alpha,\beta})\cap(u_{\max},\infty))$ for $\beta<0$. We mainly study $\sharp(\sigma_d(\mathcal{R}_{\alpha,\beta})\cap(-\infty,u_{\min}))$ for $\beta>0$, since the other is similar.
$c\in \sigma_d(\mathcal{R}_{\alpha,\beta})$ if and only if its corresponding eigenfunction $\psi_c$ satisfies the
Rayleigh-Kuo  BVP:
\begin{align}\label{sturm-Liouville}
\mathcal{L}_c\phi:=-\phi''+{u''-\beta\over u-c}\phi=\lambda\phi, \;\;\;\;\phi(y_1)= \phi(y_2)=0,
\end{align}
where $\phi\in H_0^1\cap H^2(y_1,y_2)$ and $\lambda=-\alpha^2$.  This equation  is  formulated by Kuo \cite{Kuo1949}.
For $c<u_{\min}$,
it follows from \cite{RS} that the $n$-th eigenvalue of (\ref{sturm-Liouville})  is
\begin{align}\label{variation2}
\lambda_n(c)
=&\inf_{\dim V_n=n}\;\;\sup_{\phi\in H_0^1, \phi\in V_n, \|\phi\|_{L^2}=1}\langle\mathcal{L}_c\phi,\phi\rangle\\\nonumber
=&\inf_{\dim V_n=n}\;\;\sup_{\phi\in H_0^1, \phi\in V_n, \|\phi\|_{L^2}=1}\int_{y_1}^{y_2}\left(|\phi'|^2+{u''-\beta\over u-c}|\phi|^2\right)dy.
\end{align} In this way,  we have
\begin{align*}
\sigma_d(\mathcal{R}_{\alpha,\beta})\cap(-\infty,u_{\min})=\bigcup_{n\geq1}\{c<u_{\min}:\lambda_n(c)=-\alpha^2\}.
\end{align*}
To determine whether $\sharp(\sigma_d(\mathcal{R}_{\alpha,\beta})\cap(-\infty,u_{\min}))$ is finite, we need to study the number of solutions $c<u_{\min}$ such that $\lambda_n(c)=-\alpha^2$ for   $n\geq1$. Since $\lim_{c\to-\infty}\lambda_n(c)={{n^2}\over4} \pi^2>0$ by Proposition 4.2 in \cite{LYZ} and $\lambda_n(c)$ is real-analytic on $(-\infty,u_{\min})$, the only possibility such that $\sharp(\sigma_d(\mathcal{R}_{\alpha,\beta})\cap(-\infty,u_{\min}))=\infty$ is that there exists a sequence $\{c_j(\alpha,\beta)\}_{j=1}^\infty\subset\bigcup_{n\geq1}\{c<u_{\min}:\lambda_n(c)=-\alpha^2 \} $ such that $c_j(\alpha,\beta)\to u_{\min}^-$. Thus, the key is to study the asymptotic behavior of $\lambda_n(c)$ as $c\to u_{\min}^-$.
We divide it into two steps.
 \\
 {\bf Step 1.} We study how many $n$'s such that $\lambda_n(c)\to -\infty$ as $c\to u_{\min}^-$. We determine a transitional $\beta$ value such that the number
 $\sharp\{n\geq1:\lambda_n(c)\to -\infty \text{ as }c\to u_{\min}^-\}$ changes suddenly from finite to infinite when $\beta$ passes through it.

\begin{theorem}\label{eigenvalue asymptotic behavior-bound}
Let  $u$ satisfy ${\bf (H1)}$.
$(1)$ Let  $0<\beta\leq{9\over8}\kappa_+$.
Then
\begin{itemize}
\item[(i)]
$
\lim\limits_{c\to u_{\min}^-}\lambda_n(c)=-\infty,\;\;1\leq n\leq m_{\beta};
$

\item[(ii)]
$
M_{\beta}<\infty;
$

\item[(iii)]
 there exists an integer $N_{\beta}> m_{\beta}$ such that
$
\inf\limits_{c\in(-\infty,u_{\min})}\lambda_{N_{\beta}}(c)\geq 0.
$
\end{itemize}
$(2)$ Let  ${9\over 8}\kappa_-\leq\beta<0$.
Then
\begin{itemize}
\item[(i)]
$
\lim\limits_{c\to u_{\max}^+}\lambda_n(c)=-\infty,\;\;1\leq n\leq m_{\beta};
$

\item[(ii)]
$
M_{\beta}<\infty;
$

\item[(iii)]
 there exists an integer $N_{\beta}> m_{\beta}$ such that
$
\inf\limits_{c\in(u_{\max},\infty)}\lambda_{N_{\beta}}(c)\geq 0.
$
\end{itemize}
$(3)$ Let  $\beta>{9\over8}\kappa_+$. Then
$
\lim_{c\to u_{\min}^-}\lambda_n(c)= -\infty
$ for $n\geq1$.\\
$(4)$ Let  $\beta<{9\over8}\kappa_-$. Then
$
\lim_{c\to u_{\max}^+}\lambda_n(c)= -\infty
$ for $n\geq1$.\\
Here, $\kappa_{\pm}$, $m_{\beta}$ and $M_{\beta}$ are defined in  \eqref{def-kappa+}--\eqref{def-kappa-negative-infty}, \eqref{def-m0finitebetapositive} and \eqref{def-M0finitebetapositive}, respectively.
\end{theorem}

The transitional  value $\beta={9\over8}\kappa_+$ is illustrated   in Figure 1. We give a simple example to explain why such a transitional  value exists. Consider the flow $u={1\over2} y^2$ on $[0,1]$ and $\beta>0$.
If $c<0$ is very close to $0$, then the energy quadratic form in \eqref{variation2} roughly looks like
\begin{align*}
\langle\mathcal{L}_c\phi,\phi\rangle\sim\int_0^1 |\phi'|^2+{2-2\beta\over y^2}|\phi|^2dy.
\end{align*}
Thus, if $2-2\beta>-{1\over4}\Leftrightarrow\beta<{9\over8}$, by  Hardy type inequality (Lemma \ref{Hardy type inequality2}) we have $\langle\mathcal{L}_c\phi,\phi\rangle$ is bounded from below for any test functions $\phi$ with $\phi(0)=0$. From this formal observation, we may expect $\lambda_1(c)$ is bounded from below. If $2-2\beta<-{1\over4}\Leftrightarrow\beta>{9\over8}$, $\langle\mathcal{L}_c\phi,\phi\rangle$ is unbounded from below by looking at the test functions $y^{{1\over2}+\varepsilon}$ with $\varepsilon\to0^+$. We will construct  test functions motivated by the function $y^{1\over2}$ to show that all the eigenvalues are unbounded from below.
\begin{figure}[!ht]
	\centering
	\begin{tabular}{ll}
	   \includegraphics[width=2in]{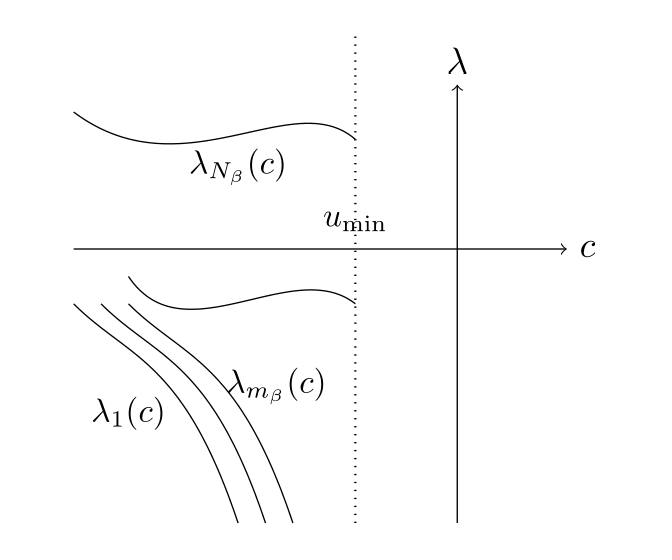}\quad\quad  \includegraphics[width=2in]{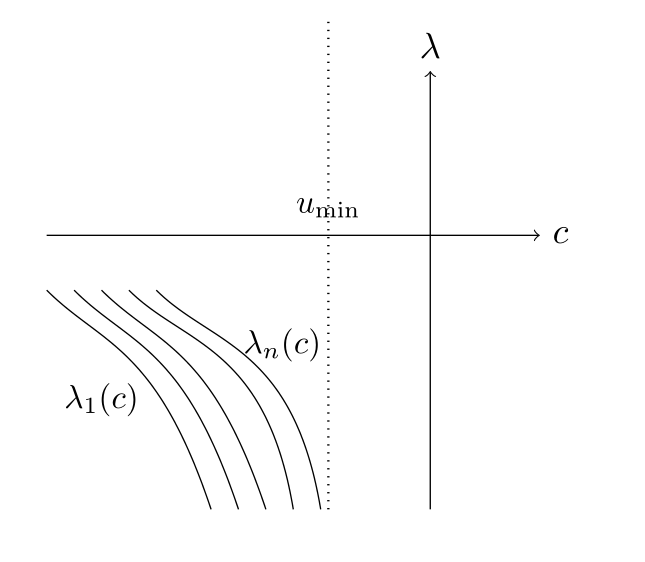} \\
	   \quad\quad  $0<\beta\leq{9\over8}\kappa_+$ \;\;\;\;\;\;\;\;\;\;\;\;\;\;\quad\quad\quad\quad\quad\quad\quad\quad $\beta>{9\over8}\kappa_+$
	\end{tabular}
 \begin{center}\vspace{-0.2cm}
   {\small {\bf Figure 1.} }
  \end{center}
\end{figure}

For general flows, the main idea in the proof of Theorem \ref{eigenvalue asymptotic behavior-bound} (1)-(2)  is  to control $\int{u''-\beta\over u-c}|\phi|^2dy$ using the $L^2$ norm of $\phi'$ near a singular point (see Lemma \ref{aH1}), which involves  very delicate and careful localized analysis. The transitional $\beta$ values ${9\over8} \kappa_{\pm}$ are essentially due to  the optimal Hardy type inequality \eqref{Hardy-finite-interval}.
The idea in the proof of Theorem \ref{eigenvalue asymptotic behavior-bound} (3)-(4) is to construct suitable test functions such that the functional in  \eqref{variation2} converges  to $-\infty$ as $c\to u_{\min}^-$ or $u_{\max}^+$, see \eqref{the point we prove}. This is inspired by the
``eigenfunction" $y^{1\over2}$ for the  optimal Hardy type equality and a support-separated technique. The proof of Theorem \ref{eigenvalue asymptotic behavior-bound} is given in Section 3.

Then we give sharp criteria for $\lambda_1(c)\to -\infty$ as $c\to u_{\min}^-$ if $\beta\in[{9\over 8}\kappa_-,{9\over8}\kappa_+]$.
By Theorem \ref{the explicit number of traveling wave families},
 the number of traveling wave families is to count the union of  $\sharp(\sigma_d(\mathcal{R}_{k\alpha,\beta})\cap\mathbf{R})$ for all $k\geq1$.
Thus,   the number of traveling wave families is infinity provided that  $\lambda_1(c)\to -\infty$ as $c\to u_{\min}^-$.
By Theorem \ref{eigenvalue asymptotic behavior-bound} (1)-(2), we get the sharp criteria for $\lambda_1(c)\to -\infty$ as $c\to u_{\min}^-$.
 \begin{corol}\label{sturm-liouville-eigen-asym}
 Let  $u$ satisfy ${\bf (H1)}$.
\begin{itemize}
\item[(1)] If $\{u=u_{\min}\}\cap(y_1,y_2)\neq \emptyset$,  then  a  transitional $\beta$ value $\min\{{9\over8}\kappa_+,\mu_+\}$ exists in
$(0,{9\over8}\kappa_+]$ such that
$\inf_{c\in(-\infty,u_{\min})}\lambda_1(c)>-\infty$ for $\beta\in(0,\min\{{9\over8}\kappa_+,\mu_+\}]$ and
 $\lim_{c\to u_{\min}^-}$ $\lambda_1(c)=-\infty$ for $\beta\in(\min\{{9\over8}\kappa_+,\mu_+\},{9\over8}\kappa_+]$.
\item[(2)] If $\{u=u_{\max}\}\cap(y_1,y_2)\neq \emptyset$, then  a  transitional $\beta$ value $\max\{{9\over8}\kappa_-,\mu_-\}$ exists in $[{9\over8}\kappa_-,0)$ such that
$\inf_{c\in(u_{\max}, \infty)}\lambda_1(c)>-\infty$ for $\beta\in[\max\{{9\over8}\kappa_-,\mu_-\},0)$ and $\lim_{c\to u_{\max}^+}\\ \lambda_1(c)=-\infty$ for $\beta\in[{9\over8}\kappa_-,\max\{{9\over8}\kappa_-,\mu_-\})$.
\item[(3)] If $\{u=u_{\min}\}\cap(y_1,y_2)= \emptyset$,  then
$\inf_{c\in(-\infty,u_{\min})}\lambda_1(c)>-\infty$ for $\beta\in(0,{9\over8}\kappa_+]$.
\item[(4)] If $\{u=u_{\max}\}\cap(y_1,y_2)= \emptyset$,  then
$\inf_{c\in(u_{\max},\infty)}\lambda_1(c)>-\infty$ for $\beta\in[{9\over8}\kappa_-,0)$.
\end{itemize}
Here, $\kappa_{\pm}$ and $\mu_{\pm}$ are defined in \eqref{def-kappa+}--\eqref{def-mu-}.
\end{corol}
Here, a key point for Corollary \ref{sturm-liouville-eigen-asym} (1) and (3) is that $\inf_{c\in(-\infty,u_{\min})}\lambda_1(c)>-\infty$ if and only if $m_{\beta}=0$ and $\beta\leq{9\over8}\kappa_+.$

 \noindent{\bf Step 2.} We rule out the oscillation of $\lambda_n(c)$  as $c\to u_{\min}^-$ (or $c\to u_{\max}^+$).
By Theorem \ref{eigenvalue asymptotic behavior-bound} (1), we get that for $1\leq n\leq m_{\beta}$,  $\lambda_n(c)=-\alpha^2$ has only finite number of solutions  $c$ on $(-\infty, u_{\min})$. Moreover, if $n\geq N_{\beta}$,  no solutions exist for $\lambda_n(c)=-\alpha^2$ on $c\in(-\infty, u_{\min})$. Now, we consider
  whether  $\sharp(\{\lambda_n(c)=-\alpha^2,c\in(-\infty,u_{\min})\})<\infty$  for $m_{\beta}<n<N_{\beta}$. Indeed, we rule out the oscillation of $\lambda_n(c)$ under   the spectral assumption $\bf{(E_\pm)}$, or under the ``good" endpoints assumption (i.e. one of the conditions  (i)--(iii) in Theorem \ref{traveling wave construction}), or for flows in class $\mathcal{K}^+$.
The oscillation of $\lambda_n(c)$ is illustrated in Figure 2.
\begin{figure}[!ht]
 \centering
 \begin{tabular}{ll}
   \includegraphics[width=2in]{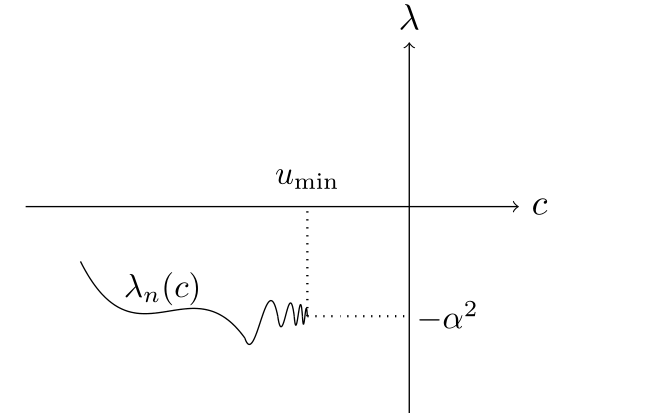}
 \end{tabular}
 \begin{center}\vspace{-0.2cm}
   {\small {\bf Figure 2.} }
  \end{center}
\end{figure}
\\
{\bf Case 1.}
Under  the spectral assumption,
the main argument to rule out oscillation is to prove  uniform $H^1$ bound for corresponding eigenfunctions, and the proof is  in Subsection 4.1.  In this case, ${9\over8} \kappa_{\pm}$ are also   transitional $\beta$ values
 for the number of isolated real eigenvalues of the linearized Euler operator if $|\kappa_{\pm}|<\infty$.

\begin{theorem}\label{main-result}
Assume that $u$ satisfies $\bf{(H1)}$ and $\alpha>0$.
\begin{itemize}
\item[(1)]
If $ 0<\beta<{9\over 8}\kappa_+$, $0<\alpha^2\leq M_{\beta}$ and $\bf{(E_+)}$ holds for this $ \alpha$, then
\begin{align}\label{num-beta+}
m_{\beta}\leq\sharp(\sigma_d(\mathcal{R}_{\alpha,\beta})\cap(-\infty,u_{\min}))<\infty.
\end{align}
If $ 0<\beta\leq{9\over 8}\kappa_+$, then  \eqref{num-beta+} holds for $\alpha^2>M_{\beta}$.
\item[(2)]
If $ {9\over8}\kappa_-<\beta<0$, $0<\alpha^2\leq M_{\beta}$ and $\bf{(E_-)}$ holds for this $ \alpha$, then
\begin{align}\label{num-beta-}
m_{\beta}\leq\sharp(\sigma_d(\mathcal{R}_{\alpha,\beta})\cap(u_{\max},\infty))<\infty.
\end{align}
If $ {9\over8}\kappa_-\leq\beta<0$, then  \eqref{num-beta-} holds for  $\alpha^2>M_{\beta}$.
 \item[(3)]
 If $\beta>{9\over 8}\kappa_+$, then $\sharp(\sigma_d(\mathcal{R}_{\alpha,\beta})\cap(-\infty,u_{\min}))=\infty$.
 \item[(4)]
If $\beta<{9\over 8}\kappa_-$, then $\sharp(\sigma_d(\mathcal{R}_{\alpha,\beta})\cap(u_{\max},\infty))=\infty$.

\end{itemize}
Here, $\kappa_{\pm}$, $\bf{(E_{\pm})}$, $m_{\beta}$ and $M_{\beta}$ are defined in \eqref{def-kappa+}--\eqref{def-kappa-negative-infty},
\eqref{E+}--\eqref{E-}, \eqref{def-m0finitebetapositive} and \eqref{def-M0finitebetapositive}, respectively.
\end{theorem}

In fact, by Theorem \ref{eigenvalue asymptotic behavior-bound} we have $M_{\beta}<\infty$ for $0<\beta\leq{9\over8}\kappa_+$ or ${9\over 8}\kappa_-\leq\beta<0$.
Here, we focus on sufficient conditions of \eqref{num-beta+} and \eqref{num-beta-}, it is unclear whether \eqref{num-beta+} is true for the case $\beta={9\over8}\kappa_+$ with $0<\alpha^2\leq M_{\beta}$, or the case $ 0<\beta<{9\over8}\kappa_+$ with $0<\alpha^2\leq M_{\beta}$ but no assumption $\bf{(E_+)}$.

Note that Theorem \ref{main-result} (3)-(4) is a direct consequence of Theorem \ref{eigenvalue asymptotic behavior-bound} (3)-(4).
\\
{\bf Case 2.}
Under  the ``good" endpoints assumption (i.e.  one of the conditions $(\rm{i})$--$(\rm{iii})$ in Theorem \ref{traveling wave construction}),  a delicate analysis near the endpoints is involved
to rule out oscillation, and the proof is in Subsection 4.2. In this case, we get that  no transitional $\beta$ values exist if $|\kappa_{\pm}|=\infty$.

\begin{theorem}\label{thm-flows good endpoints}
Let $\alpha>0$ and $u$ satisfy ${\bf (H1)}$.
 Assume that $u(y_1)\neq u(y_2)$, and one of the conditions $(\rm{i})$--$(\rm{iii})$ in Theorem \ref{traveling wave construction} holds.
Then
\begin{itemize}
\item[(1)] $\sharp(\sigma_d(\mathcal{R}_{\alpha,\beta})\cap\mathbf{R})<\infty$ for all $\beta\in(0,\infty)$ if and only if
$\{u'=0\}\cap\{u=u_{\min}\}=\emptyset$;

 \item[(2)] $\sharp(\sigma_d(\mathcal{R}_{\alpha,\beta})\cap\mathbf{R})<\infty$ for all $\beta\in(-\infty,0)$ if and only if
 $\{u'=0\}\cap\{u=u_{\max}\}=\emptyset$.
     \end{itemize}
 Consequently, $\sharp(\sigma_d(\mathcal{R}_{\alpha,\beta})\cap\mathbf{R})<\infty$ for all $\beta\in\mathbf{R}$ if and only if
 $\{u'=0\}\cap\{u=u_{\min}\}=\{u'=0\}\cap\{u=u_{\max}\}=\emptyset$.
\end{theorem}

Note that if $\alpha>\sqrt{M_{\beta}}$, then \eqref{num-beta+} and \eqref{num-beta-} are true, and the ``good" endpoints assumption (i.e.  one of the conditions $(\rm{i})$--$(\rm{iii})$ in Theorem \ref{traveling wave construction}) is not needed in Theorem \ref{thm-flows good endpoints}. Consequently, if ${2\pi\over T}>\sqrt{M_{\beta}}$, then Theorem \ref{traveling wave construction} (3)--(4) hold true without this assumption (see their proof).

Let $u$ be an analytic monotone flow and $\alpha>0$. Then
$\sharp(\sigma_d(\mathcal{R}_{\alpha,\beta})\cap\mathbf{R})<\infty$  for $\beta\neq0$.
This is a corollary of Theorem \ref{thm-flows good endpoints},  and  can also be deduced by the method used in  Lemma 3.2 and Theorem 4.1 of \cite{Rosencrans-Sattinger}.


 \noindent{\bf Case 3.} For flows in class $\mathcal{K}^+$,
the main tools
to rule out oscillation are Hamiltonian structure and index formula, and the proof is in Subsection 4.3.
This is also the main reason why the spectral and ``good" endpoints  assumptions can be removed  in Theorem \ref{traveling wave construction-classK+}.
\begin{theorem}\label{number for sinus flow}
Let $u$ be  in class $\mathcal{K}^+$ and $\alpha>0$.  Then
$m_{\beta}\leq\sharp(\sigma_d(\mathcal{R}_{\alpha,\beta})\cap(-\infty,u_{\min}))<\infty$ for $ 0<\beta\leq{9\over 8}\kappa_+$;
$m_{\beta}\leq\sharp(\sigma_d(\mathcal{R}_{\alpha,\beta})\cap(u_{\max},\infty))<\infty$ for $ {9\over8}\kappa_-\leq\beta<0$;
and
$\sharp(\sigma_d(\mathcal{R}_{\alpha,\beta})\cap\mathbf{R})=\infty$ for $\beta\notin[{9\over 8}\kappa_-,{9\over 8}\kappa_+]$.
Here, $\kappa_{\pm}$ and $m_{\beta}$ are defined in  \eqref{def-kappa+}--\eqref{def-kappa-negative-infty} and \eqref{def-m0finitebetapositive}, respectively.
\end{theorem}

The idea of the  proof   is as follows.  The linearized equation has  Hamiltonian structure and the energy quadratic form has finite negative directions. The key observation is that  oscillation of $\lambda_n(c)$ brings  infinite times  of   sign-changes of $\lambda_n'(c)$. This contributes infinite negative directions of quadratic form for  non-resonant neutral modes, which is a contradiction to the index formula. Thus,   the oscillation of $\lambda_n(c)$ can be ruled out unconditionally for flows in class  $\mathcal{K}^+$.

\section{Transitional  values $\beta={9\over8}\kappa_{\pm}$ for the $n$-th eigenvalue of Rayleigh-Kuo BVP}

We begin to study the asymptotic behavior of the $n$-th eigenvalue $\lambda_n(c)$  of Rayleigh-Kuo BVP.
In this section, we focus on the number $\sharp\{n\geq1:\lambda_n(c)\to -\infty \text{ as }c\to u_{\min}^-\;(\text{or }c\to u_{\max}^+)\}$. We prove that  the number is finite for $\beta\in [{9\over8}\kappa_{-},{9\over8}\kappa_{+}]$ and it is infinite  for $\beta\notin [{9\over8}\kappa_{-},{9\over8}\kappa_{+}]$, which is stated precisely in Theorem \ref{eigenvalue asymptotic behavior-bound}.
\subsection{Finite number for $\beta\in [{9\over8}\kappa_{-},{9\over8}\kappa_{+}]$}
The optimal constant in the following Hardy type inequality plays an important role in discovering the transitional  values $\beta={9\over8}\kappa_{\pm}$.

\begin{lemma}\label{Hardy type inequality2}
Let $\phi\in H^1(a,b)$ and $\phi(y_0)=0$ for some $y_0\in [a,b]$. Then
\begin{align}\label{Hardy-finite-interval}
\left\|\phi\over y-y_0\right\|_{L^2(a,b)}^2+\frac{\int_{a}^{b}
|y-y_0|^{-1}{\phi(y)^2}{}dy}{\max(b-y_0,y_0-a)}\leq 4\|\phi'\|_{L^2(a,b)}^2.
\end{align}
Here the constant $4$ is optimal.
\end{lemma}

\begin{proof}\def\ve{\varepsilon} Suppose that $\phi$ is real-valued   without loss of generality. Let $ \ve=1/\max(b-y_0,y_0-a)$. First, we consider the integration on  $[y_0,b]$ (if $y_0<b$).
Since
\begin{align}\label{phiyto0}\left|\phi(y)^2\over y-y_0\right|\leq {\|\phi'\|^2_{L^2(y_0,y)}(y-y_0)\over y-y_0}=\|\phi'\|^2_{L^2(y_0,y)}\to0,\;\;as \;\;y\to y_0^+,
\end{align}
we have
\begin{align*}
&\left\|\phi\over y-y_0\right\|_{L^2(y_0,b)}^2=-\int_{y_0}^{b}
\phi(y)^2d\left(1\over y-y_0\right)
=-{\phi(y)^2\over y-y_0}\big|_{y_0}^{b}+\int_{y_0}^{b}{2\phi(y)\phi'(y)\over y-y_0}dy, \\ &\int_{y_0}^{b}{2\phi(y)\phi'(y)}dy={\phi(y)^2}\big|_{y_0}^{b}=\phi(b)^2,\\&\left\|\phi\over y-y_0\right\|_{L^2(y_0,b)}^2+\left\|{\phi\over y-y_0}-\ve\phi-2\phi'\right\|_{L^2(y_0,b)}^2\\
=& -2{\phi(b)^2\over b-y_0}+2\ve\phi(b)^2-2\ve\int_{y_0}^{b}
\frac{\phi(y)^2}{y-y_0}dy+\ve^2\|\phi\|_{L^2(y_0,b)}^2+4\|\phi'\|_{L^2(y_0,b)}^2\\ \leq&-\ve\int_{y_0}^{b}
\frac{\phi(y)^2}{y-y_0}dy+4\|\phi'\|_{L^2(y_0,b)}^2.
\end{align*}
Here we used $ \ve\leq {1\over b-y_0}$ and ${\phi(y)^2\over y-y_0}\big|_{y_0}^{b}={\phi(b)^2\over b-y_0}. $ Thus, $\big\|{\phi\over y-y_0}\big\|_{L^2(y_0,b)}^2+\ve\int_{y_0}^{b}
\frac{\phi(y)^2}{y-y_0}dy\leq 4\|\phi'\|_{L^2(y_0,b)}^2.$
Similarly, $\big\|{\phi\over y-y_0}\big\|_{L^2(a,y_0)}^2+\ve\int_a^{y_0}
\frac{\phi(y)^2}{y_0-y}dy\leq 4\|\phi'\|_{L^2(a,y_0)}^2$. This gives (\ref{Hardy-finite-interval}).
Letting $y_0=a$, $\phi(y)=(y-a)^{{1\over2}+\varepsilon_1}$ and sending $\varepsilon_1\to0^+$, we see that the constant $4$ is optimal.
\end{proof}

For other versions of Hardy type inequality, the readers are referred to \cite{Godfrey-Hardy-Littlewood,Masmoudi}.
To study the lower bound  of the $n$-th eigenvalue $\lambda_n(c)$ of Rayleigh-Kuo BVP  for  $c$  close to $u_{\min}^-$, it is important to estimate the energy expression \eqref{variation2} near singular points. To this end, we need the following lemma.
\begin{lemma}\label{aH1}
Assume that $a\in [y_1,y_2],\ u(a)=u_{\min},\ \phi\in H_0^1(y_1,y_2),\ c<u_{\min}$ and $\beta>0$. Then there exists a constant $ \delta_0>0$ (depending only on $u$ and $a$) such that for $0<\delta\leq \delta_0$,
\begin{itemize}
\item[(1)]
 if $(\rm{i})$ $u'(a)\neq 0$ or $(\rm{ii})$ $u'(a)= 0,\ \beta\leq 9u''(a)/8,\ \phi(a)=0$, then
\begin{align}\label{a1}
\int_{[a-\delta,a+\delta]\cap[y_1,y_2]}\left(|\phi'|^2+{u''-\beta\over u-c}|\phi|^2\right)dy\geq 0;
\end{align}
\item[(2)] if $(\rm{iii})$ $a\in(y_1,y_2),\ u'(a)= 0,\ \beta\leq u''(a)$, then\begin{align}\label{a2}
\int_{[a-\delta,a+\delta]\cap[y_1,y_2]}\left(|\phi'|^2+{u''-\beta\over u-c}|\phi|^2\right)dy\geq -C_{\delta}\|\phi\|_{L^2([a-\delta,a+\delta]\cap[y_1,y_2])}^2.
\end{align}
Here  $C_{\delta}$ is a positive constant depending only on $u,a$ and $\delta$.
\end{itemize}
\end{lemma}
\begin{proof}\def\b{\widetilde{\beta}}First, we assume (i). Then we have $a\in\{y_1,y_2\}$ and thus $\phi(a)=0$. Without loss of generality, we assume that $a=y_1$. In this case, $u'(y_1)>0$. Choose $\delta_1\in(0,y_2-y_1)$ small enough such that $u'(y)>{u'(y_1)\over2}$ for $y\in(y_1,y_1+\delta_1)$, and thus, there exists $\xi_y\in(y_1,y)$ such that
$u(y)-c>u(y)-u_{\min}=u'(\xi_y)(y-y_1)>{u'(y_1)\over2}(y-y_1)>0$ for $c<u_{\min}$.
Note that for $y\in(y_1,y_1+\delta)$,
\begin{align*}
|\phi(y)|^2=\left|\int_{y_1}^{y}\phi'(s)ds\right|^2\leq \|\phi'\|_{L^2(y_1,y)}^2(y-y_1).
\end{align*}
 Now we take $\delta_0\in(0,\delta_1)$ to be small enough such that $\int_{y_1}^{y_1+\delta_0}{2|u''-\beta|\over u'(y_1)}dy\leq1. $ Then for $0<\delta\leq \delta_0$,
\begin{align*}
\left|\int_{y_1}^{y_1+\delta}{u''-\beta\over u-c}|\phi|^2dy\right|\leq\|\phi'\|_{L^2(y_1,y_1+\delta)}^2 \int_{y_1}^{y_1+\delta}{2|u''-\beta|\over u'(y_1)}dy\leq \|\phi'\|_{L^2(y_1,y_1+\delta)}^2,
\end{align*}which implies \eqref{a1} since $[a-\delta,a+\delta]\cap[y_1,y_2]= [y_1,y_1+\delta]$.

Now we assume (ii), then $u''(a)>0 $. Let $\delta_1\in(0,\max(y_2-a,a-y_1))$ be small enough such that $u''(y)>{u''(a)\over2}>0$ for $y\in[a-\delta_1,a+\delta_1]\cap[y_1,y_2]$. Since $u\in H^4(y_1,y_2)\subset C^3([y_1,y_2])$, we have $|u''(y)-u''(a)|\leq C|y-a|$ and $ |1/u''(y)-1/u''(a)|\leq C|y-a| $ for $y\in[a-\delta_1,a+\delta_1]\cap[y_1,y_2]$.
Then there exists $\xi_y\in\{z:|z-a|<|y-a|\}$ such that
$u(y)-c>u(y)-u_{\min}={u''(\xi_y)\over2}(y-a)^2$ $>{2\beta\over9}(y-a)^2>0$ for $c<u_{\min}$ and $y\in[a-\delta_1,a+\delta_1]\cap[y_1,y_2]$, and thus\begin{align*}
0&<{1\over u(y)-c}<{2\over u''(\xi_y)(y-a)^2}\leq{2+C|\xi_y-a|\over u''(a)(y-a)^2}\leq{2+C|y-a|\over u''(a)(y-a)^2},\\ u''(y)-\beta&\geq u''(a)-C|y-a|-9u''(a)/8
=-u''(a)/8-C|y-a|,\\ {u''(y)-\beta\over u(y)-c}&\geq{-u''(a)/8-C|y-a|\over u(y)-c}\geq-{u''(a)\over 8}{2+C|y-a|\over u''(a)(y-a)^2}-{C|y-a|\over (2\beta/9)(y-a)^2}
\\ &\geq-{1+C_0|y-a|\over 4(y-a)^2}.
\end{align*}
Now we take $\delta_0=\min(\delta_1,C_0^{-1})>0$. For $0<\delta\leq \delta_0$,  we have
\begin{align}\label{hardy-inequality-control-potential}
&-\int_{[a-\delta,a+\delta]\cap[y_1,y_2]}{u''-\beta\over u-c}|\phi|^2dy\leq \int_{[a-\delta,a+\delta]\cap[y_1,y_2]}{1+C_0|y-a|\over 4(y-a)^2}|\phi|^2dy\\\nonumber \leq&\int_{[a-\delta,a+\delta]\cap[y_1,y_2]}{1+\delta^{-1}|y-a|\over 4(y-a)^2}|\phi|^2dy\leq \left\|\phi'\right\|_{L^2([a-\delta,a+\delta]\cap[y_1,y_2])}^2,
\end{align}which implies \eqref{a1}. Here we used Lemma \ref{Hardy type inequality2} in the last step.

Next, we assume (iii).  In this case, $u'(a)=0$. Let $\b=u''(a)>0$. Then $\b\geq\beta$. Let $\delta_1\in(0,\min(y_2-a,a-y_1))$ be small enough such that $u''(y)>{u''(a)\over2}>0$ and $ |u(y)-u_{\min}|\leq1/2$ for $y\in[a-\delta_1,a+\delta_1]$. Then $0<u(y)-c<1$ for $y\in[a-\delta_1,a+\delta_1]$ and $c\in(u_{\min}-1/2,u_{\min})$.
Now we assume $0<\delta\leq \delta_1$. Direct computation implies
 \begin{align*}
 &\int_{a-\delta}^{a+\delta}{u''-\b\over u-c}|\phi|^2dy=\int_{a-\delta}^{a+\delta}{u''-\b\over u'}|\phi|^2d(\ln(u-c))\\
 =&\left({u''-\b\over u'}|\phi|^2(\ln(u-c))\right)\big|_{a-\delta}^{a+\delta}-\int_{a-\delta}^{a+\delta}\ln(u-c)\left({u''-\b\over u'}|\phi|^2\right)'dy=I_{c,\delta}(\phi)+II_{c,\delta}(\phi).
 \end{align*}
Since $\b-u''(a)=0$, it follows from the proof of Lemma 3.7 in \cite{WZZ} that ${u''-\b\over u'}\in H^1(a-\delta_1,a+\delta_1)$.
By interpolation, we have $\|\phi\|_{L^\infty(a-\delta,a+\delta)}\leq C_{\delta}\|\phi\|_{L^2(a-\delta,a+\delta)}+\|\phi'\|_{L^2(a-\delta,a+\delta)}$, and thus
\begin{align*}
&\left|\left({u''-\b\over u'}|\phi|^2\right)({a+\delta})-\left({u''-\b\over u'}|\phi|^2\right)({a-\delta})\right|
\leq C\left\|\left({u''-\b\over u'}|\phi|^2\right)'\right\|_{L^2(a-\delta,a+\delta)}\delta^{1\over2}\\
\leq&C\left\|\left({u''-\b\over u'}\right)'|\phi|^2+\left({u''-\b\over u'}\right)2\phi\phi'\right\|_{L^2(a-\delta,a+\delta)}\delta^{1\over2}\\
\leq&C\left(\|\phi\|^2_{L^\infty(a-\delta,a+\delta)}+\|\phi\|_{L^\infty(a-\delta,a+\delta)}\|\phi'\|_{L^2(a-\delta,a+\delta)}\right)\delta^{1\over2}\\
\leq&(C_{\delta}\|\phi\|^2_{L^2(a-\delta,a+\delta)}+C\|\phi'\|^2_{L^2(a-\delta,a+\delta)})\delta^{1\over2}
\end{align*}
 and $\left|\left({u''-\b\over u'}|\phi|^2\right)(a-\delta)\right|\leq C\|\phi\|^2_{L^\infty(a-\delta,a+\delta)}\leq C_{\delta}\|\phi\|^2_{L^2(a-\delta,a+\delta)}+C\|\phi'\|^2_{L^2(a-\delta,a+\delta)}$. Then
\begin{align*}
I_{c,\delta}(\phi)=&\left({u''-\b\over u'}|\phi|^2\right)\big|_{a-\delta}^{a+\delta}\ln(u(a+\delta)-c)+\left({u''-\b\over u'}|\phi|^2\right)|_{a-\delta}\ln(u-c)\big|_{a-\delta}^{a+\delta}\\ \nonumber
\leq &(C_{\delta}\|\phi\|^2_{L^2(a-\delta,a+\delta)}+C\|\phi'\|^2_{L^2(a-\delta,a+\delta)})
(\delta^{1\over2}|\ln(u(a+\delta)-c)|+|\ln(u-c)\big|_{a-\delta}^{a+\delta}|).
\end{align*}
Note that $C'(y-a)^2\leq |u(y)-u(a)|\leq C''(y-a)^2$ for $y\in [a-\delta,a+\delta]$. Then $|\ln(u(y)-c)|=-\ln(u(y)-c)\leq-\ln (u(y)-u(a))\leq-\ln( C'(y-a)^2)$ and $|\ln(u(a+\delta)-c)|\leq-\ln( C'\delta^2)\leq C(|\ln \delta|+1) $ for $|c-u_{\min}|<1/2$ and $y\in[a-\delta,a+\delta]$. Let $u_+=\max(u(a+\delta), u(a-\delta))$ and $u_-=\min(u(a+\delta), u(a-\delta))$. Then $u_+\geq u_->u(a)$, and\begin{align*}
|\ln(u-c)\big|_{a-\delta}^{a+\delta}|=\ln{u_+-c\over u_--c}\leq\ln{u_+-u(a)\over u_--u(a)}=\left|\ln{u(a+\delta)-u(a)\over u(a-\delta)-u(a)}\right|
\end{align*}for $c<u_{\min}=u(a)$. Thus,\begin{align}\label{estamite I-c-delta}
&|I_{c,\delta}(\phi)|\leq (C_{\delta}\|\phi\|^2_{L^2(a-\delta,a+\delta)}+C\|\phi'\|^2_{L^2(a-\delta,a+\delta)})\Psi(\delta)
\end{align}
for $c\in(u_{\min}-1/2,u_{\min})$,
where
\begin{align*}
\Psi(\delta)=\delta^{1\over2}(|\ln \delta|+1)+\left|\ln{u(a+\delta)-u(a)\over u(a-\delta)-u(a)}\right|.
\end{align*} Note that\begin{align*}
\lim_{\delta\to0^+}\left|\ln{u(a+\delta)-u(a)\over u(a-\delta)-u(a)}\right|=\lim_{\delta\to0^+}\left|\ln{u''(\xi_{a+\delta})\over u''(\xi_{a-\delta})}\right|=0,
\end{align*}
where $\xi_{a+\delta}\in (a, a+\delta)$ and $\xi_{a-\delta}\in (a-\delta,a)$.
Therefore $\lim_{\delta\to0^+}\Psi(\delta)=0$.

Next, we claim  that $\ln(u-c)$, $c\in(u_{\min}-1/2,u_{\min})$, is uniformly bounded in $L^p(a-\delta,a+\delta)$ for $1<p<\infty$. The proof is similar as that in Lemma 3.7 of \cite{WZZ}. Note that $|\ln(u(y)-c)|\leq-\ln( C'(y-a)^2)$ for $|c-u_{\min}|<1/2$ and $y\in[a-\delta,a+\delta]$. Therefore,
\begin{align*}
\int_{a-\delta}^{a+\delta}|\ln(u-c)|^pdy\leq C\int_{a-\delta}^{a+\delta}(|\ln(y-a)^2|^p+1)dy\leq C\int_{-\delta}^{\delta}(|\ln|z|^2|^p+1)dz\leq C.
\end{align*}

Now, we consider $II_{c,\delta}(\phi)$.
\begin{align}\label{estamite II-c-delta}
|II_{c,\delta}(\phi)|\leq& (2\delta)^{1\over4}\|\ln(u-c)\|_{L^4(a-\delta,a+\delta)}\left\|\left({u''-\beta\over u'}|\phi|^2\right)'\right\|_{L^2(a-\delta,a+\delta)}\\ \nonumber
\leq &(C_{\delta}\|\phi\|^2_{L^2(a-\delta,a+\delta)}+C\|\phi'\|^2_{L^2(a-\delta,a+\delta)})\delta^{1\over4}.
\end{align}
Combining (\ref{estamite I-c-delta}) and (\ref{estamite II-c-delta}), we get
 for  $c\in(u_{\min}-1/2, u_{\min})$,
\begin{align*}
\left|\int_{a-\delta}^{a+\delta}{u''-\b\over u-c}|\phi|^2dy\right|\leq &(C_{\delta}\|\phi\|^2_{L^2(a-\delta,a+\delta)}+
C_1\|\phi'\|^2_{L^2(a-\delta,a+\delta)})(\delta^{1\over4}+\Psi(\delta)).
\end{align*}
Since $\lim_{\delta\to0^+}\Psi(\delta)=0$, we have $ \lim_{\delta\to0^+}(\delta^{1\over4}+\Psi(\delta))=0$, and there exists $ \delta_0\in(0,\delta_1)$ such that
$ C_1(\delta^{1\over4}+\Psi(\delta))<1$. Then for $0<\delta\leq \delta_0$ and $c\in(u_{\min}-1/2,u_{\min})$, we have\begin{align*}
&\int_{[a-\delta,a+\delta]\cap[y_1,y_2]}{u''-\beta\over u-c}|\phi|^2dy\geq \int_{[a-\delta,a+\delta]\cap[y_1,y_2]}{u''-\b\over u-c}|\phi|^2dy\\ \geq& -(C_{\delta}\|\phi\|^2_{L^2(a-\delta,a+\delta)}+C_1\|\phi'\|^2_{L^2(a-\delta,a+\delta)})(\delta^{1\over4}+\Psi(\delta))\\
\geq&-C_{\delta}\|\phi\|^2_{L^2(a-\delta,a+\delta)}(\delta^{1\over4}+\Psi(\delta))-\|\phi'\|^2_{L^2(a-\delta,a+\delta)},
\end{align*}which implies \eqref{a2} since $[a-\delta,a+\delta]\cap[y_1,y_2]= [a-\delta,a+\delta]$. On the other hand, for $0<\delta\leq \delta_0$ and $c\leq u_{\min}-1/2$, we have\begin{align*}
&\int_{[a-\delta,a+\delta]\cap[y_1,y_2]}{u''-\beta\over u-c}|\phi|^2dy\geq \int_{[a-\delta,a+\delta]\cap[y_1,y_2]}{u''-\b\over u-c}|\phi|^2dy\\ \geq& -2\int_{[a-\delta,a+\delta]\cap[y_1,y_2]}|u''-\b||\phi|^2dy
\geq-C\|\phi\|^2_{L^2(a-\delta,a+\delta)},
\end{align*}which implies \eqref{a2}. This completes the proof.\end{proof}


Now, we are ready to prove  Theorem \ref{eigenvalue asymptotic behavior-bound} (1)-(2).
\begin{proof}[Proof of Theorem \ref{eigenvalue asymptotic behavior-bound} (1)-(2)] We first give the proof of (1), and (2) can be proved similarly.
Consider $1\leq n\leq m_{\beta}$. It suffices to show that $\lim_{c\to u_{\min}^-}\lambda_{m_{\beta}}(c)=-\infty$.
Let
\begin{align}\label{def-a1am0}\{a\in (y_1,y_2):u=u_{\min},u''(a)-\beta<0\}=\{a_1,\cdots,a_{m_{\beta}}\},\end{align}
and
\begin{align}\label{def-eta}
\eta(x)=\left\{\begin{array}{ll}
		\mu\exp\left({-1\over1-x^2}\right),\;\;\;\;\;\;\;\;\;\; x\in(-1,1),\\
		0,\;\;\;\;\;\;\;\;\;\;\;\;\;\;\;\;\;\;\;\;\;\;\;\;\;\;\;\;\;x\notin(-1,1),
		\end{array}\right.
\end{align}
where $\mu>0$ is a constant such that $\int_{-1}^1\eta(x)^2dx=1$. Then $\eta\in C_0^\infty(\mathbf{R})$.
Define
\begin{align*}
\varphi_i(y)=\delta_0^{-{1\over2}}\eta\left({y-a_i\over\delta_0}\right),\;\;y\in[y_1,y_2],
\end{align*}
where $1\leq i\leq m_{\beta}$, and $\delta_0>0$ is  small enough such that
$(a_i-\delta_0,a_i+\delta_0)\cap(a_j-\delta_0,a_j+\delta_0)=\emptyset$ for $i\neq j$ and $u''(y)-\beta<0$ for all $y\in\cup_{1\leq i\leq m_{\beta}}(a_i-\delta_0,a_i+\delta_0)\subset(y_1,y_2)$. Then $\|\varphi_i\|_{L^2(y_1,y_2)}=1$ and $\text{supp}\,(\varphi_i)=(a_i-\delta_0,a_i+\delta_0)$. Thus,
$\varphi_i\bot\varphi_j$ in the $L^2$ sense for $i\neq j$.
Let
$V_{m_{\beta}}=\text{span}\{\varphi_1,\cdots,\varphi_{m_{\beta}}\}.$ Then $V_{m_{\beta}}\subset H_0^1(y_1,y_2)$.
By (\ref{variation2}), there exist $b_{i,c}\in\mathbf{R}$, $i=1,\cdots,m_{\beta}$, with $\sum_{i=1}^{m_{\beta}}|b_{i,c}|^2=1 $ such that $\varphi_c=\sum_{i=1}^{m_{\beta}}b_{i,c}\varphi_i\in V_{m_{\beta}}$ with $\|\varphi_c\|^2_{L^2}=1$, and
\begin{align}\nonumber
\lambda_{m_{\beta}}(c)\leq &\sup_{\|\phi\|_{L^2}=1,\phi\in V_{m_{\beta}}}\int_{y_1}^{y_2}\left(|\phi'|^2+{u''-\beta\over u-c}|\phi|^2\right)dy=
\int_{y_1}^{y_2}\left(|\varphi_c'|^2+{u''-\beta\over u-c}|\varphi_c|^2\right)dy\\ \nonumber
=&\sum_{i=1}^{m_{\beta}}|b_{i,c}|^2\int_{a_i-\delta_0}^{a_i+\delta_0}\left(|\varphi_i'|^2+{u''-\beta\over u-c}|\varphi_i|^2\right)dy\\ \label{eigenvaluetendtonegativeinfty}
\leq&\max_{1\leq i\leq m_{\beta}}\int_{a_i-\delta_0}^{a_i+\delta_0}\left(|\varphi_i'|^2+{u''-\beta\over u-c}|\varphi_i|^2\right)dy\to-\infty\;\;\text{as}\;\;c\to u_{\min}^-.
\end{align}

Next, we prove (ii).
Let $\delta_1>0$ be a sufficiently small  constant such that $(a-\delta_1,a+\delta_1)\subset[y_1,y_2] $ for $a\in\{u=u_{\min}\}\setminus\{y_1,y_2\} $, and
$|a-b|>2\delta_1$ for $a,b\in\{u=u_{\min}\}$ and $ a\neq b$. 
There are four cases for zeros of $a\in\{u=u_{\min}\}$ as follows:

 Case 1. $a\in\{y_1,y_2\}$ and  $u'(a)\neq0$;

 Case 2. $a\in\{y_1,y_2\}$,   $u'(a)=0$ (thus $\beta\leq{9\over8}\kappa_+\leq 9u''(a)/8 $);

 Case 3. $a\in(y_1,y_2)$ and  $\beta\leq u''(a)$;

 Case 4. $a\in(y_1,y_2)$ and  $u''(a)<\beta\leq 9u''(a)/8$.\\
Then we divide our proof into four cases as above.
In fact, for  Cases 1--2, by Lemma \ref{aH1} (1) there exists $ \delta(a)>0$ such that for $0<\delta\leq \delta(a)$, $c<u_{\min}$ and $\phi\in H_0^1$,
\begin{align}\label{singular-point-integrate-estimate}
\int_{[a-\delta,a+\delta]\cap[y_1,y_2]}\left(|\phi'|^2+{u''-\beta\over u-c}|\phi|^2\right)dy\geq 0.
\end{align}
  For  Case 3, by Lemma \ref{aH1} (2) there exists $ \delta(a)>0$ such that for $0<\delta\leq \delta(a)$,  $c<u_{\min}$ and $\phi\in H_0^1$,
 \begin{align}\label{singular-point-integrate-estimate2}
\int_{[a-\delta,a+\delta]\cap[y_1,y_2]}\left(|\phi'|^2+{u''-\beta\over u-c}|\phi|^2\right)dy\geq-C(\delta,a)\int_{[a-\delta,a+\delta]\cap[y_1,y_2]}|\phi|^2dy.
\end{align}Here $C(\delta,a)$ depends only on $u,a,\delta$. Moreover, if $ \phi(a)=0$, then by Lemma \ref{aH1} (1),\begin{align}\label{a3}
\int_{[a-\delta,a+\delta]\cap[y_1,y_2]}\left(|\phi'|^2+{u''-\beta\over u-c}|\phi|^2\right)dy\geq0.
\end{align}
For  Case 4, we have $a\in \{a_1,\cdots,a_{m_{\beta}}\}.$ By Lemma \ref{aH1} (1), there exists $ \delta(a)>0$ such that for $0<\delta\leq \delta(a)$,
$c<u_{\min}$, $\phi\in H_0^1$ and  $\phi(a)=0$,
 \begin{align}\label{singular-point-integrate-estimate-case5}
\int_{[a-\delta,a+\delta]\cap[y_1,y_2]}\left(|\phi'|^2+{u''-\beta\over u-c}|\phi|^2\right)dy\geq 0.
\end{align}Now let $ \delta_0=\min(\{\delta(a):a\in\{u=u_{\min}\}\}\cup\{\delta_1\})$. Define
\begin{align*}
(q_1(y),q_1^0(y))&=\begin{cases}
({u''(y)-\beta\over u(y)-c},1) &
\;y\in [y_1,y_2]\setminus\cup_{a\in\{u=u_{\min}\}}(a-\delta_0,a+\delta_0),  \\
(0,0)  &
y\in\cup_{a\in\{u=u_{\min}\}}\left((a-\delta_0,a+\delta_0)\cap [y_1,y_2]\right),
\end{cases}\\
(q_2(y),q_2^0(y))&=({u''(y)-\beta\over u(y)-c}-q_1(y),1-q_1^0(y)) \;\;y\in[y_1,y_2].
\end{align*}Then there exists $C_0>0$ such that for $c<u_{\min}$,
 \begin{align*}
 |q_1(y)|\leq C_0\;\; \text{for}\;\; y\in[y_1,y_2].
 \end{align*}
For $\phi\in H_0^1$ and  $\|\phi\|_{L^2}=1$,
\begin{align*}&\int_{y_1}^{y_2}\left(|\phi'|^2+{u''-\beta\over u-c}|\phi|^2\right)dy\\
\nonumber=&
\int_{y_1}^{y_2}\left(q_1^0|\phi'|^2+q_1|\phi|^2\right)dy+\int_{y_1}^{y_2}\left(q_2^0|\phi'|^2+q_2|\phi|^2\right)dy
=I_c(\phi)+II_c(\phi).
\end{align*}

Let us first consider  $I_c(\phi)$.
For   $\phi\in H_0^1$ and $\|\phi\|_{L^2}=1$, we have
\begin{align}\label{estimate-for-I-new-ii}
I_c(\phi)\geq \int_{y_1}^{y_2}\left(-C_0|\phi|^2\right)dy
\geq-C_0
\end{align}
for  $c<u_{\min}$.
 We proceed to consider $II_c(\phi)$.
\begin{align}\label{decomposition-IIc}
II_c(\phi)=& \sum_{a\in\{u=u_{\min}\}}\int_{[a-\delta_0,a+\delta_0]\cap[y_1,y_2]}\left(|\phi'|^2+{u''-\beta\over u-c}|\phi|^2\right)dy\\ \nonumber
=&\left(\sum_{\text{Case 1}}+\cdots+\sum_{\text{Case 4}}\right)\int_{[a-\delta_0,a+\delta_0]\cap[y_1,y_2]}\left(|\phi'|^2+{u''-\beta\over u-c}|\phi|^2\right)dy.
\end{align}
Recall that $a_1,\cdots,a_{m_{\beta}}$ are defined in (\ref{def-a1am0}). For any $(m_{\beta}+1)$-dimensional subspace $V=\text{span}\{\psi_1,\cdots,\psi_{m_{\beta}+1}\}$ in $H_0^1(y_1,y_2)$,
there exists  $0\neq(\xi_1,\cdots,\xi_{m_{\beta}+1})\in\mathbf{R}^{m_{\beta}+1}$ such that
$
\xi_1\psi_1(a_i)+\cdots+\xi_{m_{\beta}+1}\psi_{m_{\beta}+1}(a_i)=0, i=1,\cdots,m_{\beta}.
$
Define
$
\tilde\psi=\xi_1\psi_1+\cdots+\xi_{m_{\beta}+1}\psi_{m_{\beta}+1}.
$
Then $\tilde\psi(a_i)=0$, $i=1,\cdots,m_{\beta}$. We normalize $\tilde\psi$ such that $\|\tilde\psi\|_{L^2(y_1,y_2)}=1$. Then by \eqref{singular-point-integrate-estimate}, \eqref{singular-point-integrate-estimate2}, \eqref{singular-point-integrate-estimate-case5} and \eqref{decomposition-IIc}, we have\begin{align*}
II_c(\tilde\psi)\geq& \sum_{\text{Case 3}}\int_{[a-\delta_0,a+\delta_0]\cap[y_1,y_2]}\left(|\tilde\psi'|^2+{u''-\beta\over u-c}|\tilde\psi|^2\right)dy\\
\geq& -\sum_{\text{Case 3}}C(\delta_0,a)\int_{[a-\delta_0,a+\delta_0]\cap[y_1,y_2]}|\tilde\psi|^2dy\geq -\max_{\text{Case 3}}C(\delta_0,a).
\end{align*}
This, along with (\ref{variation2}) and  (\ref{estimate-for-I-new-ii}), yields that
$
\inf_{c\in(-\infty,u_{\min})}\lambda_{m_{\beta}+1}(c)\geq -\max\limits_{\text{Case 3}}C(\delta_0,a)-C_0.
$
This proves (ii).

Finally, we prove (iii).\def\M{M}\def\Ma{M_1}
Let $q_1, q_2, I_c(\phi), II_c(\phi)$ and $ C_0$ be defined as in (ii). Let $\mu_1([a,b])$ be the principal eigenvalue of
\begin{align*}
-\phi''=\lambda\phi, \;\;\phi(a)=\phi(b)=0.
\end{align*}Then we have $\mu_1([a,b])=|\pi /(b-a)|^2,$ and \begin{align}\label{slp1}
\int_a^b|\phi'|^2dy\geq \mu_1([a,b])\int_a^b|\phi|^2dy \ \ \text{for}\ \phi\in H_0^1(a,b).
\end{align}
Let $ \delta_2=\pi/C_0^{\frac{1}{2}}$. Then we have $\mu_1([a,b])\geq C_0$ for $0<b-a\leq \delta_2$. Let\begin{align*}
\M=(\{n\delta_2:n\in\Z\}\cup\{a+b:a\in\{u=u_{\min}\},b\in\{-\delta_0,0,\delta_0\}\}\cup\{y_1,y_2\})\cap[y_1,y_2].
\end{align*}
Then $\M$ is a finite set, and we can write its elements in the increasing order\begin{align*}
\M=\{a_0',\cdots,a_{N_{\beta}}'\},\ y_1=a_0'<\cdots<a_{N_{\beta}}'=y_2.
\end{align*}
Then $0<a'_{k+1}-a_k'\leq \delta_2 $ and $ \mu_1([a_k',a'_{k+1}])\geq C_0$ for $0\leq k<N_{\beta}$. Let \begin{align*}
\Ma=\{k\in\Z:0\leq k<N_{\beta},\ [a_k',a'_{k+1}]\cap(a-\delta_0,a+\delta_0)=\emptyset,\ \forall\ a\in\{u=u_{\min}\}\}.
\end{align*}Then we have $
[y_1,y_2]\setminus(\cup_{a\in\{u=u_{\min}\}}(a-\delta_0,a+\delta_0))=\cup_{k\in\Ma}[a_k',a'_{k+1}].
$

For any $N_{\beta}$-dimensional subspace $V=\text{span}\{\psi_1,\cdots,\psi_{N_{\beta}}\}$ in $H_0^1(y_1,y_2)$,
there exists  $0\neq(\xi_1,\cdots,\xi_{N_{\beta}})\in\mathbf{R}^{N_{\beta}}$ such that
$
\xi_1\psi_1(a_i')+\cdots+\xi_{N_{\beta}}\psi_{N_{\beta}}(a_i')=0, i=1,\cdots,N_{\beta}-1.
$
Define
$
\tilde\psi=\xi_1\psi_1+\cdots+\xi_{N_{\beta}}\psi_{N_{\beta}}.
$
Then $\tilde\psi(a_i')=0$, $i=1,\cdots,N_{\beta}-1$. We normalize $\tilde\psi$ such that $\|\tilde\psi\|_{L^2(y_1,y_2)}=1$. Since $ \tilde\psi\in H_0^1(y_1,y_2)$, we also have $\tilde\psi(a_0')=\tilde\psi(y_1)=0,\ \tilde\psi(a_{N_{\beta}}')=\tilde\psi(y_2)=0$, and thus $\tilde\psi(a_i')=0$ for $i=0,\cdots,N_{\beta},$ i.e.  $\tilde\psi|_{\M}=0.$ By \eqref{slp1}, we have\begin{align}\label{def-k0}\int_{a_k'}^{a_{k+1}'}|\tilde\psi'|^2dy\geq \mu_1([a_k',a'_{k+1}])\int_{a_k'}^{a_{k+1}'}|\tilde\psi|^2dy\geq C_0\int_{a_k'}^{a_{k+1}'}|\tilde\psi|^2dy,\ k=0,\cdots,N_{\beta}-1.\end{align}
First, we consider  $I_c(\tilde\psi)$. By \eqref{def-k0}, we have
\begin{align}\label{estimate-for-I-new}
I_c(\tilde\psi)&\geq \int_{[y_1,y_2]\setminus(\cup_{a\in\{u=u_{\min}\}}(a-\delta_0,a+\delta_0))}\left(|\tilde\psi'|^2-C_0|\tilde\psi|^2\right)dy\\ \nonumber&
=\sum_{k\in\Ma}\int_{a_k'}^{a_{k+1}'}\left(|\tilde\psi'|^2-C_0|\tilde\psi|^2\right)dy\geq0.
\end{align}

Next, we consider $II_c(\tilde\psi)$. For $a\in\{u=u_{\min}\},$ we have $a\in\M$ and $\tilde\psi(a)=0. $ Then by  \eqref{singular-point-integrate-estimate}, \eqref{a3}, \eqref{singular-point-integrate-estimate-case5} and \eqref{decomposition-IIc}, we have $II_c(\tilde\psi)\geq 0.$ This, along with (\ref{variation2}) and  (\ref{estimate-for-I-new}), yields that
$
\inf_{c\in(-\infty,u_{\min})}\lambda_{N_{\beta}}(c)\geq 0.
$
This proves (iii).\if0\begin{align*}
II_c(\tilde\psi)\geq& 0.
\end{align*}
We  decompose  $II_c(\phi)$ as (\ref{decomposition-IIc}) and analyze   five cases for zeros of $u=u_{\min}$ as in (ii).
From (\ref{singular-point-integrate-estimate})--(\ref{singular-point-integrate-estimate2}) in the proof of (ii), it suffices to
 consider $u''(a_i)-\beta<0$ in Case 5.
Recall that $a_1,\cdots,a_{m_{\beta}}$ are defined in (\ref{def-a1am0}) and $k_0$ is determined by (\ref{def-k0}).
For any $(k_0+m_{\beta}+1)$-dimensional subspace $V$ in $H_0^1$, we infer that
\begin{align}\label{dimension formulae}
\dim\left(V\cap\overline{\text{span}\{\phi_n\}_{n=k_0+1}^\infty}\right)\geq m_{\beta}+1.\end{align}

In fact, suppose that $\dim\left(V\cap\overline{\text{span}\{\phi_n\}_{n=k_0+1}^\infty}\right)=\tilde m_0\leq m_{\beta}$ and let $V\cap\overline{\text{span}\{\phi_n\}_{n=k_0+1}^\infty}\\=\text{span}\{\alpha_1,\cdots,\alpha_{\tilde m_0}\}$.
Then  $V$ can be extended by $V=\text{span}\{\alpha_1,\cdots,\alpha_{\tilde m_0},\rho_1,\cdots,\rho_{n_0}\}$ with $\tilde m_0+n_0=k_0+m_{\beta}+1$, where $\rho_i\in V$.
We decompose $\rho_i$ by
$$\rho_i=\gamma_i+\tilde\gamma_i,\;\;\gamma_i\in\text{span}\{\phi_n\}_{n=1}^{k_0},\;\;\tilde\gamma_i\in\overline{\text{span}\{\phi_n\}_{n=k_0+1}^\infty},$$
where $i=1,\cdots,n_0$. Then $\gamma_i\neq0$ for $i=1,\cdots,n_0$. We show that $\gamma_i$, $i=1,\cdots,n_0,$ are linearly independent in $\text{span}\{\phi_n\}_{n=1}^{k_0}$. Assuming this is true, then $n_0\geq k_0+1>\dim\left(\text{span}\{\phi_n\}_{n=1}^{k_0}\right)$, which is a contradiction.
Suppose that $l_1\gamma_1+\cdots+ l_{n_0}\gamma_{n_0}=0$. Since $\gamma_i=\rho_i-\tilde\gamma_i$, we have
$$l_1\rho_1+\cdots+l_{n_0}\rho_{n_0}=l_1\tilde\gamma_1+\cdots+l_{n_0}\tilde\gamma_{n_0}\subset V\cap\overline{\text{span}\{\phi_n\}_{n=k_0+1}^\infty}.$$
Thus, there exist $p_i$, $i=1,\cdots,\tilde m_0$, such that
$l_1\rho_1+\cdots+l_{n_0}\rho_{n_0}-p_1\alpha_1-\cdots-p_{\tilde m_0}\alpha_{\tilde m_0}=0.$
Therefore, $l_1=\cdots=l_{n_0}=p_1=\cdots=p_{\tilde m_0}=0$. This proves (\ref{dimension formulae}).

Thus, we can choose $\hat\psi_1,\cdots,\hat\psi_{m_{\beta}+1}$ to be  linearly independent in $V\cap\overline{\text{span}\{\phi_n\}_{n=k_0+1}^\infty}$.
Then there exists  $0\neq(\zeta_1,\cdots,\zeta_{m_{\beta}+1})\in\mathbf{R}^{m_{\beta}+1}$ such that
$
\zeta_1\hat\psi_1(a_i)+\cdots\zeta_{m_{\beta}+1}\hat\psi_{m_{\beta}+1}(a_i)=0, i=1,\cdots,m_{\beta}.
$
Define
$
\hat\psi=\zeta_1\hat\psi_1+\cdots+\zeta_{m_{\beta}+1}\hat\psi_{m_{\beta}+1}.
$
Then $\hat\psi\in V\cap\overline{\text{span}\{\phi_n\}_{n=k_0+1}^\infty}$ and $\hat\psi(a_i)=0$, $i=1,\cdots,m_{\beta}$. We normalize $\hat\psi$ such that $\|\hat\psi\|_{L^2}=1$.
Using the property $\hat\psi(a_i)=0$, $i=1,\cdots,m_{\beta}$, and with a similar approach  to Case 5 in (ii), we get
 (\ref{singular-point-integrate-estimate-case5}) holds for $c<u_{\min}$, where
  $\tilde \psi$ is replaced  by $\hat\psi$.

 We have shown that for any $(k_0+m_{\beta}+1)$-dimensional subspace $V\subset H_0^1(y_1,y_2)$, there exists $\hat\psi\in V\cap\overline{\text{span}\{\phi_n\}_{n=k_0+1}^\infty}$ such that
$
II_c(\hat \psi)\geq 0
$
 for  $c\in[c_1,u_{\min})$.
This, along with (\ref{estimate-for-I-new}), yields that
$
\inf_{c\in[c_1,u_{\min})}\lambda_{k_0+m_{\beta}+1}(c)\geq 0.
$
Clearly, there exists $N_{\beta}>k_0+m_{\beta}+1$ such that
$
\inf_{c\in(-\infty,c_1)}\lambda_{N_{\beta}}(c)\geq 0.
$
This proves (iii).\fi
\end{proof}

\subsection{Infinite number for $\beta\notin [{9\over8}\kappa_{-},{9\over8}\kappa_{+}]$}
In this subsection, we  prove Theorem \ref{eigenvalue asymptotic behavior-bound} (3)-(4). The proof is based on
construction of suitable test functions such that the energy in  \eqref{variation2} converges  to $-\infty$ as $c\to u_{\min}^-$ or $c\to u_{\max}^+$.

\begin{proof} [Proof of Theorem \ref{eigenvalue asymptotic behavior-bound}  (3)-(4)]
We only prove Theorem \ref{eigenvalue asymptotic behavior-bound}  (3), since (4) can be proved similarly. Let $\beta >{9\over8}\kappa_+$.
Then there exists  $a\in[y_1,y_2]$ such that  $ \beta/u''(a)>9/8$, $u'(a)=0$ and $u(a)=u_{\min}$.
 If $a\in[y_1,y_2)$, our analysis  is completely on   $[a,a+\delta]\subset [y_1,y_2)$  for  $\delta>0$ small enough. If $a=y_2$, the analysis is only on  $[a-\delta,a]$ and the proof is similar as $a\in[y_1,y_2)$. Now we assume that $a=0\in [y_1,y_2)$.  Then $u''(0)>0$ and there exists  $\varepsilon_0>0$ such that
$u''(z)>0$  and
\begin{align*}
{2(u''(y)-\beta)\over u''(z)}<-{1\over 4-\varepsilon_0}
\end{align*}
for $y,z\in[0,\delta]\subset[y_1,y_2)$ and $\delta>0$ small enough.
 Let $\nu_0=\min_{z\in[0,\delta]}\{u''(z)\}>0$ and
$J(x)=\eta(2x-1), x\in\mathbf{R},$
 where $\eta$ is defined in (\ref{def-eta}).
Define $$\varphi_{i,R}(y)=
\begin{cases}
y^{1\over2}J\left({\ln y \over R}+i+1\right), &
y\in(0,y_2],  \\
0,  &
y\in[y_1,0],
\end{cases}$$
 where $i=1,\cdots,n$, and $R$ is large enough such that $e^{-R}<\delta$. Then $\varphi_{i,R}\in H_0^1(y_1,y_2)$ and  $\text{supp}\, \varphi_{i,R}=[e^{-(i+1)R}, e^{-iR}]$, $i=1,\cdots,n$.
 Thus, $\varphi_{i,R}\perp\varphi_{j,R}$ in the $L^2$  sense for $i\neq j$.
 Note that $u''(y)-\beta<0$ for $y\in[0,\delta]$. For $1\leq i\leq n$, we define
  $$\tilde \varphi_{i,R}={1\over\|\varphi_{i,R}\|_{L^2}}\varphi_{i,R},\;\;\text{and}\;\;
  \tilde V_{n,R}=\text{span}\{\tilde\varphi_{1,R},\cdots,\tilde\varphi_{n,R}\}.$$
   Choose $R>0$ such that $2(u_{\min}-c)={\varepsilon_0\nu_0e^{-2(n+1)R}\over8}$. Then $c\to u_{\min}^-\Leftrightarrow R\to\infty$. We shall show that for $1\leq i\leq n,$
 \begin{align}\label{the point we prove}
\lim_{c\to u_{\min}^-}\int_{y_1}^{y_2}\left(|\tilde\varphi_{i,R}'|^2+{u''-\beta\over u-c}|\tilde\varphi_{i,R}|^2\right)dy=-\infty.
\end{align}
Assume that (\ref{the point we prove}) is true. Similar to (\ref{eigenvaluetendtonegativeinfty}), there exist  $d_{i,c}\in\mathbf{R}, i=1,\cdots,n,$ with  $\sum_{i=1}^n|d_{i,c}|^2=1$ such that
\begin{align}\label{kth eigenvalue tends to -infty}
\lambda_n(c)\leq \sum_{i=1}^n|d_{i,c}|^2\int_{y_1}^{y_2}\left(|\tilde\varphi_{i,R}'|^2+{u''-\beta\over u-c}|\tilde\varphi_{i,R}|^2\right)dy\to -\infty\;\;\text{as}\;\;c\to u_{\min}^-.
\end{align}
Now we prove (\ref{the point we prove}).
 Direct computation gives
 \begin{align*}
 {u''(y)-\beta\over u(y)-c}=&{u''(y)-\beta\over u(y)-u_{\min}+u_{\min}-c}={2(u''(y)-\beta)\over u''(\xi_y)y^2+2(u_{\min}-c)}
 <{2(u''(y)-\beta)\over \left(u''(\xi_y)+{\varepsilon_0u''(\xi_y)\over 8}\right)y^2}\\
 <&{-{1\over 4-\varepsilon_0}u''(\xi_y)\over \left(u''(\xi_y)+{\varepsilon_0u''(\xi_y)\over 8}\right)y^2}
  =-{1\over(4-\varepsilon_0) \left(1+{\varepsilon_0\over 8}\right)y^2}
 =-{1\over(4-\varepsilon_1)y^2}
 \end{align*}
 for $y\in[e^{-(i+1)R}, e^{-iR}]$ and $2(u_{\min}-c)={\varepsilon_0\nu_0e^{-2(n+1)R}\over8}$, where $\xi_y\in(0,y)$ and $\varepsilon_1={\varepsilon_0\over 2}+{\varepsilon^2_0\over8}$.
Then
\begin{align}\label{estimate-inf}
\int_{y_1}^{y_2}\left(|\varphi_{i,R}'|^2+{u''-\beta\over u-c}|\varphi_{i,R}|^2\right)dy=&\int_{e^{-(i+1)R}}^{e^{-iR}}\left(|\varphi_{i,R}'|^2+{u''-\beta\over u-c}|\varphi_{i,R}|^2\right)dy\\\nonumber
<&\int_{e^{-(i+1)R}}^{e^{-iR}}\left(|\varphi_{i,R}'|^2-{1\over(4-\varepsilon_1)y^2}|\varphi_{i,R}|^2\right)dy.
\end{align}
Note that for $y\in[e^{-(i+1)R},e^{-iR}]$,
\begin{align*}
|\varphi_{i,R}'(y)|^2=&\left|J(x){-(2x-1)\over 4x^2(x-1)^2}{1\over Ry}y^{1\over2}+{1\over2}y^{-{1\over2}}J(x)\right|^2\\
=&J(x)^2{(2x-1)^2\over 16x^4(x-1)^4}{1\over R^2y}-J(x)^2{2x-1\over 4x^2(x-1)^2}{1\over Ry}+{1\over 4y} J(x)^2,
\end{align*}
where $x={\ln y \over R}+i+1$. Since $\left|J(x)^2{(2x-1)^2\over 16x^4(x-1)^4}\right|\leq C$ and $\left|J(x)^2{2x-1\over 4x^2(x-1)^2}\right|\leq C$ for  $x\in[0,1]$, we get
\begin{align*}
&\left|\int_{e^{-(i+1)R}}^{e^{-iR}}
\left(J(x)^2{(2x-1)^2\over 16x^4(x-1)^4}{1\over R^2y}-J(x)^2{2x-1\over 4x^2(x-1)^2}{1\over Ry}\right)dy\right|\\
\leq&{C\over R^2}\int_{e^{-(i+1)R}}^{e^{-iR}}{1\over y}dy+{C\over R}\int_{e^{-(i+1)R}}^{e^{-iR}}{1\over y}dy={C\over R}+C\leq C.
\end{align*}
Then we infer from (\ref{estimate-inf}) that
\begin{align}\label{estimate-inf-varphi}
&\int_{y_1}^{y_2}\left(|\varphi_{i,R}'|^2+{u''-\beta\over u-c}|\varphi_{i,R}|^2\right)dy
\leq C+
\int_{e^{-(i+1)R}}^{e^{-iR}}\left({1\over 4y} J(x)^2-{1\over(4-\varepsilon_1)y^2}|\varphi_{i,R}|^2\right)dy\\\nonumber
=&C+\int_{e^{-(i+1)R}}^{e^{-iR}}{-\varepsilon_1\over4(4-\varepsilon_1)}J(x)^2{1\over y}dy
= C-{\varepsilon_1R\over 8(4-\varepsilon_1)}<0
\end{align}
when $R$ is large enough.
Direct computation gives
\begin{align}\label{estimate-inf-varphi-normalize}
\|\varphi_{i,R}\|_{L^2}^2=&\int_{e^{-(i+1)R}}^{e^{-iR}}yJ\left({\ln y\over R}+i+1\right)^2dy
=\int_0^1Re^{2R(x-i-1)}J(x)^2dx\\\nonumber
\leq&CR\int_0^1e^{2R(x-i-1)}dx=C(e^{-2iR}-e^{-2(i+1)R})\leq C e^{-2iR}.
\end{align}
Combining (\ref{estimate-inf-varphi}) and (\ref{estimate-inf-varphi-normalize}), we have
\begin{align*}
&\int_{y_1}^{y_2}\left(|\tilde\varphi_{i,R}'|^2+{u''-\beta\over u-c}|\tilde\varphi_{i,R}|^2\right)dy
={1\over\|\varphi_{i,R}\|_{L^2}^2}\int_{y_1}^{y_2}\left(|\varphi_{i,R}'|^2+{u''-\beta\over u-c}|\varphi_{i,R}|^2\right)dy\\
\leq&\left(C-{\varepsilon_1R\over 8(4-\varepsilon_1)}\right){1\over\|\varphi_{i,R}\|_{L^2}^2}
\leq \left(C-{\varepsilon_1R\over 8(4-\varepsilon_1)}\right){e^{2iR}\over C}\to-\infty
\end{align*}
as $R\to\infty$.
This proves (\ref{the point we prove}).
\end{proof}

\section{Rule out oscillation of the $n$-th eigenvalue of Rayleigh-Kuo BVP}

Let $\beta\in[{9\over8}\kappa_-,{9\over8}\kappa_+]$. By Theorem \ref{eigenvalue asymptotic behavior-bound} (1)-(2),   $\lambda_n(c)=-\alpha^2$ has only finite number of solutions  $c$ outside the range of $u$ for  $1\leq n\leq m_{\beta}$, and  no solutions exist for $n\geq N_{\beta}$. It is non-trivial to study whether the  number of solutions is finite for $m_{\beta}<n<N_{\beta}$. Recall that $N_\beta$ is obtained in Theorem \ref{eigenvalue asymptotic behavior-bound} such that $
\inf\limits_{c\in(-\infty,u_{\min})}\lambda_{N_{\beta}}(c)\geq 0
$ for $0<\beta\leq{9\over8}\kappa_+$, $
\inf\limits_{c\in(u_{\max},\infty)}\lambda_{N_{\beta}}(c)\geq 0
$ for ${9\over 8}\kappa_-\leq\beta<0$.
The main difficulty is that  $\lambda_n(c)$ might  oscillate when $c$ is close to $u_{\min}$ or $u_{\max}$. In this section, we rule out the oscillation in the following three cases.
\subsection{Rule out oscillation under the spectral assumption}
We rule out the oscillation of $\lambda_n(c)$ under   the spectral assumption $\bf{(E_\pm)}$, which is stated in Theorem \ref{main-result} (1)-(2).
To this end, we first consider the compactness near a class of singular points.
\begin{lemma}\label{critical-point-not-beta-u-sec-dao-equal0}
Let $c\in \text{Ran}(u)$, $y_0\in u^{-1}\{c\}\cap (y_1,y_2)$, $u'(y_0)=0$  and $\delta>0$ so that $(u''(y_0)-\beta)(u''(y)-\beta)>0$ on $[y_0-\delta,y_0+\delta]\subset[y_1,y_2]$ and $[y_0-\delta,y_0+\delta]\cap u^{-1}\{c\}=\{y_0\}$. Assume that $ \beta/u''(y_0)< 9/8.$ Let $\phi_n,\omega_n\in H^1(y_0-\delta,y_0+\delta)$ and $c_n\in\mathbf{C}$ so that $c_n^i>0$, $c_n\to c$, $\phi_n\rightharpoonup0,\omega_n\to0 $ in $H^1(y_0-\delta,y_0+\delta)$ and
\begin{align*}
(u-c_n)(\phi_n''-\alpha^2\phi_n)-(u''-\beta)\phi_n=\omega_n
\end{align*}
 holds on $[y_0-\delta,y_0+\delta]$.
Then $\phi_n\to0$ in $H^1(y_0-\delta,y_0+\delta)$.
\end{lemma}Here $c_n^i=\text{Im}(c_n)$. The proof of Lemma \ref{critical-point-not-beta-u-sec-dao-equal0} is the same as that of Lemma 3.4 in \cite{WZZ}, where we only used the condition $\beta/u''(y_0)< 9/8 $ rather than the stronger condition:\beno
(\textbf{H}) \quad u\in H^{4}(y_{1},y_{2}), \,\,  u''(y_c)\neq0,\,\, \beta/u''(y_c)<9/8\,\, \text{at critical points } u'(y_c)=0.
\eeno Since otherwise, we can construct $ \widetilde{u}$ such that $\widetilde{u}\in H^{4}(y_{1},y_{2}),\ \widetilde{u}|_{[y_0-\delta,y_0+\delta]}= u|_{[y_0-\delta,y_0+\delta]}$ and $ (\widetilde{u}')^{-1}\{0\}=\{y_0\}$.
 Recall that all the conditions and conclusions depend only on $u|_{[y_0-\delta,y_0+\delta]}$.
Then we prove the uniform $H^1$ bound for the eigenfunctions. More precisely, we have the following result.
\begin{proposition}\label{uniformH1bound}
Let $0<\beta<{9\over8}\kappa_+$.
Assume that $m_{\beta}<n<N_\beta$, $\{c_k\}\subset(-\infty,u_{\min})$, $c_k\to u_{\min}^-$ and $-\lambda_n(c_k)=\alpha^2>0$, where $m_{\beta}$ and $N_{\beta}$ are given in \eqref{def-m0finitebetapositive} and Theorem \ref{eigenvalue asymptotic behavior-bound}, and  $\lambda_n(c_k)$ is the $n$-th eigenvalue of
\begin{align}\label{eigenfunction equation boundary condition}
-\psi_{k}''+{u''-\beta\over u-c_k}\psi_{k}=\lambda_n(c_k)\psi_{k},\;\;\;\;\psi_k(y_1)=\psi_k(y_2)=0
\end{align}
with   the $L^2$ normalized eigenfunction $\psi_k$.
Then
\begin{align}\label{uniform-H1-bound}
\|\psi_{k}\|_{H^1(y_1,y_2)}\leq C,\;\;k\geq1.
\end{align}
\end{proposition}\begin{proof}
Suppose that (\ref{uniform-H1-bound}) is not true. Up to a subsequence, we can assume that $\|\psi_{k}\|_{H^1(y_1,y_2)}\\ \geq k.$ Let $\hat \psi_{k}={\psi_k\over \|\psi_k\|_{H^1(y_1,y_2)}}$ on $[y_1,y_2]$. Then
$
-\hat\psi_{k}''+{u''-\beta\over u-c_k}\hat\psi_{k}=-\alpha^2\hat\psi_{k}
$
on $[y_1,y_2]$, $\|\hat \psi_{k}\|_{H^1(y_1,y_2)}\\=1$ and  $\|\hat \psi_{k}\|_{L^2(y_1,y_2)}=1/\|\psi_{k}\|_{H^1(y_1,y_2)}\leq 1/k\to 0$. Thus, $\hat\psi_k\rightharpoonup0$ in $H^1(y_1,y_2)$.

Similar to Lemma 3.1
 in \cite{WZZ}, we have $\hat \psi_k\to0$ in $H^1((a-\delta,a+\delta)\cap[y_1,y_2])$ for $a\in\{u=u_{\min}\}\cap\{u'\neq0\}.$
Similar to
Lemma 3.5 and Remark 3.6 in \cite{WZZ}, we have $\hat \psi_k\to0$ in $H^1((a-\delta,a+\delta)\cap[y_1,y_2])$ for $a\in\{u=u_{\min}\}\cap\{y_1,y_2\}\cap\{u'=0\}\cap\{u''\neq\beta\}.$ Similar to
Lemma 3.7 and Remark 3.8 in \cite{WZZ}, we have $\hat \psi_k\to0$ in $H^1((a-\delta,a+\delta)\cap[y_1,y_2])$ for $a\in\{u=u_{\min}\}\cap\{u'=0\}\cap\{u''=\beta\}.$
The main difference is that $c_k\in\R$ rather than $\text{Im}({c}_k)>0$, and we can overcome this difficulty by perturbation of $ {c}_k$ as in the next case.

If  $a\in\{u=u_{\min}\}\cap(y_1,y_2)\cap\{u''\neq\beta\}$, then $0<\beta<{9\over8}\kappa_+\leq 9u''(a)/8 $. Take $\delta\in(0,\min(y_2-a,a-y_1))$ small enough so that $(u''(a)-\beta)(u''(y)-\beta)>0$ on $[a-\delta,a+\delta]\subset[y_1,y_2]$ and $[a-\delta,a+\delta]\cap \{u=u_{\min}\}=\{a\}$. Noting that $\hat \psi_{k},{u''-\beta\over u-c_k}\in H^1(a-\delta,a+\delta)$, we have $\hat\psi_{k}''-\alpha^2\hat\psi_{k}={u''-\beta\over u-c_k}\hat\psi_{k}\in H^1(a-\delta,a+\delta), $ and there exists $ \epsilon_k>0$ such that $ \epsilon_k(1+\|\hat\psi_{k}''-\alpha^2\hat\psi_{k}\|_{H^1(a-\delta,a+\delta)})\to 0.$ Let $ \widetilde{c}_k=c_k+i\epsilon_k$ and $\omega_k=-i\epsilon_k(\hat\psi_{k}''-\alpha^2\hat\psi_{k})$. Then we have $(u-\widetilde{c}_k)(\hat \psi_{k}''-\alpha^2\hat \psi_{k})-(u''-\beta)\hat \psi_{k}=\omega_k,\ \|\omega_k\|_{H^1(a-\delta,a+\delta)}\to 0,\ \widetilde{c}_k\to u_{\min},\ \text{Im}(\widetilde{c}_k)>0$ and $\hat\psi_k\rightharpoonup0$ in $H^1(a-\delta,a+\delta)$. By Lemma \ref{critical-point-not-beta-u-sec-dao-equal0}, we have $\hat \psi_{k}\to0$ in $H^1((a-\delta,a+\delta)\cap[y_1,y_2])$. Note that $\hat \psi_{k}\to0$ in $C^2_{loc}([y_1,y_2]\setminus \{u=u_{\min}\})$. Therefore, $\hat \psi_{k}\to0$ in $H^1(y_1,y_2)$, which contradicts $\|\hat \psi_{k}\|_{H^1(y_1,y_2)}=1$. Thus, (\ref{uniform-H1-bound}) is true.\end{proof}

Following Definition 3.10 in \cite{WZZ}, we call $u_{\min}$ (or $u_{\max}$)  to be an embedding eigenvalue of $\mathcal{R}_{\alpha,\beta}$ if
there exists a nontrivial $\psi\in H_0^1(y_1,y_2)$ such that for any $\varphi\in H_0^1(y_1,y_2)$ and $\text{supp} \ \varphi\subset(y_1,y_2)\setminus\{y\in(y_1,y_2):u(y)=u_{\min}, u''(y)\neq\beta\}$,
$$\int_{y_1}^{y_2}(\psi'\varphi'+\alpha^2\psi\varphi)dy+p.v.\int_{y_1}^{y_2}{(u''-\beta)\psi\varphi\over u-u_{\min}}dy=0.$$
Equivalently, $u_{\min}$ (or $u_{\max}$) is an embedding eigenvalue of the linearized operator of (\ref{Euler equation}) (in  velocity form) defined  on $L^2\times L^2$. In fact, $\vec{v}=(\psi',-i\alpha\psi)\neq0$ is the corresponding eigenfunction.

We are now in a position to prove Theorem \ref{main-result} (1)-(2).

\begin{proof}[Proof of Theorem \ref{main-result} (1)-(2)]
First, we prove Theorem \ref{main-result} (1).
Suppose $\sharp(\sigma_d(\mathcal{R}_{\alpha,\beta})\cap(-\infty,u_{\min}))=\infty$. Then by Theorem \ref{eigenvalue asymptotic behavior-bound}, there exist
$m_{\beta}<n<N_{\beta}$ and $\{c_k\}\subset(-\infty,u_{\min})$ with $c_k\to u_{\min}^-$ such that $\lambda_n(c_k)=-\alpha^2$ is the $n$-th eigenvalue of
 (\ref{eigenfunction equation boundary condition}) with   the $L^2$ normalized eigenfunction $\psi_k$. By the definition of $M_{\beta} $ we have $\alpha^2=-\lambda_n(c_k)\leq M_{\beta}$, which implies the second statement of Theorem \ref{main-result} (1). To prove the first statement, we now assume that $0<\alpha^2\leq M_{\beta}$ and $ 0<\beta<{9\over8}\kappa_+$. By Proposition \ref{uniformH1bound}, up to a subsequence, there exists
$\psi_0\in H_0^1(y_1,y_2)$ such that $\psi_k\rightharpoonup\psi_0$ in $H_0^1(y_1,y_2)$.  Similar to
 (3.28) in \cite{WZZ},
$$\lim_{k\to\infty}\int_{a-\delta}^{a+\delta}{(u''-\beta)\psi_k\varphi\over u-c_k}dy=p.v.\int_{a-\delta}^{a+\delta}{(u''-\beta)\psi_0\varphi\over u-u_{\min}}dy$$
for any $a\in\{u=u_{\min}\}\cap(y_1,y_2)\cap\{u''=\beta\}$ and $\varphi\in H_0^1(a-\delta,a+\delta)$.
Since ${(u''-\beta)\psi_k\varphi\over u-{c_k}}\to{(u''-\beta)\psi_0\varphi\over u-u_{\min}}$ in $C^0_{\text{loc}}((y_1,y_2)\setminus\{u=u_{\min}\})$,   taking limits in
$$\int_{y_1}^{y_2}(\psi_k'\varphi'+\alpha^2\psi_k\varphi)+{(u''-\beta)\psi_k\varphi\over u-c_k}dy=0$$
 for any $\varphi\in H_0^1(y_1,y_2)$ and $\text{supp} \ \varphi\subset(y_1,y_2)\setminus\{y\in(y_1,y_2):u(y)=u_{\min}, u''(y)\neq\beta\}$,
we get
$$\int_{y_1}^{y_2}(\psi_0'\varphi'+\alpha^2\psi_0\varphi)dy+p.v.\int_{y_1}^{y_2}{(u''-\beta)\psi_0\varphi\over u-u_{\min}}dy=0.$$
If $\psi_0$ is nontrivial, $u_{\min}$ is an embedding eigenvalue of $\mathcal{R}_{\alpha,\beta}$, which is a contradiction. Therefore, $\psi_k\rightharpoonup\psi_0\equiv0$ in $H^1(y_1,y_2)$, which contradicts that $\|\psi_k\|_{L^2(y_1,y_2)}=1$, $k\geq1$.
Thus, $\sharp(\sigma_d(\mathcal{R}_{\alpha,\beta})\cap(-\infty,u_{\min}))<\infty$. Theorem \ref{main-result} (2) can be proved similarly.
\end{proof}
\subsection{Rule out oscillation under   ``good" endpoints assumption}
We rule out the oscillation of $\lambda_n(c)$ under  the ``good" endpoints assumption (i.e. one of the conditions (i)--(iii) in Theorem \ref{traveling wave construction}). The statement is given in Theorem \ref{thm-flows good endpoints}. To this end, we need the following  two lemmas.

\begin{lemma}\label{lem1} Let $u\in C^2([y_1,y_2])$, $u(y_1)=u_{\min}$ and $u'(y_1)\neq0$. For fixed $ \gamma\in(0,1/2],$ there exist constants $C>0$ and $\delta\in(0,y_2-y_1)$ such that if $ \delta_1\in(0,\delta],\ z=y_1+\delta_1,\ 0<u_{\min}-c<1,$ $ \phi\in C^2([y_1,z])$ and $\phi''=F$, then
\begin{align}\label{phi1}&|\phi|_{L^{\infty}(z)}\leq C(\delta_1^{\gamma}|(u-c)^{2-\gamma}F|_{L^{\infty}(z)}+|\phi(z)|+\delta_1^{\gamma}|\phi'(z)|),\\& \label{phi2}|(u-c)^{1-\gamma}\phi|_{L^{\infty}(z)}+\delta_1^{\gamma}|(u-c)^{2-2\gamma}\phi'|_{L^{\infty}(z)}\\ \nonumber\leq& C(\delta_1^{\gamma}|(u-c)^{3-2\gamma}F|_{L^{\infty}(z)}+|\phi(z)|+\delta_1^{\gamma}|\phi'(z)|),\\& \label{phi3}|(u-c)^{\gamma-1}\phi|_{L^{\infty}(z)}\\ \leq& \nonumber C(|(u-c)^{\gamma+1}F|_{L^{\infty}(z)}+|(u_{\min}-c)^{\gamma-1}\phi(y_1)|+|\phi(z)|+|\phi'(z)|),\\&\label{phi4}|(u-c)^{-1}\phi|_{L^{\infty}(z)} \leq C(\delta_1^{\gamma}|(u-c)^{1-\gamma}F|_{L^{\infty}(z)}+|(u_{\min}-c)^{-1}\phi(y_1)|+|\phi'(z)|),
\end{align}
where $ |f|_{L^{\infty}(z)}:=\sup_{y\in[y_1,z]}|f(y)|$.
\end{lemma}
\begin{proof} Since $0<u_{\min}-c<1$ and $-C\leq u_{\min}\leq u(y)\leq C,$ we have $0<u(y)-c\leq C$ for $y\in [y_1,z].$ Let $A_{\mu}=|(u-c)^{\mu}F|_{L^{\infty}(z)}$
 and $B_{\mu}=|(u-c)^{\mu}\phi'|_{L^{\infty}(z)}$ for $\mu\in\R$. Let $\delta>0$ be small enough such that $u'(y)>{u'(y_1)\over2}>0$ for
 $y\in[y_1,z]\subset[y_1,y_1+\delta]\subset[y_1,y_2].$  Then
 \begin{align*}&\int_y^z(u(s)-c)^{-\mu}ds\leq \frac{2}{u'(y_1)}\int_y^zu'(s)(u(s)-c)^{-\mu}ds=\frac{2(u(s)-c)^{1-\mu}|_{s=z}^{s=y}}{u'(y_1)(\mu-1)}
 \leq\frac{2(u(y)-c)^{1-\mu}}{u'(y_1)(\mu-1)}
\end{align*}
for fixed $\mu>1$ and  $y\in[y_1,z]$, and thus\begin{align}\label{u-c1}
&\int_y^z(u(s)-c)^{-\mu}ds\leq \frac{2(u(y)-c)^{1-\mu}}{u'(y_1)(\mu-1)}\leq C(u(y)-c)^{1-\mu}.
\end{align}Similarly, for fixed $\mu<1$ and $y\in[y_1,z]$,  we have\begin{align}\label{u-c2}
&\int_{y_1}^y(u(s)-c)^{-\mu}ds\leq \frac{2(u(s)-c)^{1-\mu}|_{s=y_1}^{s=y}}{u'(y_1)(1-\mu)}\leq C(u(y)-c)^{1-\mu}.
\end{align}Since   $ u(s)-c\geq u(s)-u_{\min}\geq u'(y_1)(s-y_1)/2$ and $ u(s)-c\geq u(y)-c$ for $y_1\leq y\leq s\leq z,$  we have
for fixed $\mu\geq 1-\gamma,$
\begin{align*}\numberthis\label{u-c3}&\int_y^z(u(s)-c)^{-\mu}ds\leq \int_y^z(u'(y_1)(s-y_1)/2)^{\gamma-1}(u(y)-c)^{1-\gamma-\mu}ds\\ \leq& C(u(y)-c)^{1-\gamma-\mu}\int_y^z(s-y_1)^{\gamma-1}ds\leq C(u(y)-c)^{1-\gamma-\mu}(z-y_1)^{\gamma}\\=&C\delta_1^{\gamma}(u(y)-c)^{1-\gamma-\mu}.
\end{align*} For fixed $\mu>1,$ using \eqref{u-c1} and the definition of $A_{\mu}$, we have for $ y\in[y_1,z]$,
\begin{align*}&|\phi'(y)-\phi'(z)|\leq\int_y^z|\phi''(s)|ds=\int_y^z|F(s)|ds\leq \int_y^z(u(s)-c)^{-\mu}A_{\mu}ds\leq C(u(y)-c)^{1-\mu}A_{\mu},\end{align*}and
\begin{align*}&|\phi'(y)|\leq|\phi'(z)|+ C(u(y)-c)^{1-\mu}A_{\mu},\\&|(u(y)-c)^{\mu-1}\phi'(y)|
\leq(u(y)-c)^{\mu-1}|\phi'(z)|+ CA_{\mu}\leq C|\phi'(z)|+ CA_{\mu}.\end{align*}
Then by the definition of $B_{\mu}$, we have
\begin{align}\label{Bmu}&B_{\mu-1}
=\sup_{y\in[y_1,z]}|(u(y)-c)^{\mu-1}\phi'(y)|\leq CA_{\mu}+C|\phi'(z)| \ \ \text{for fixed}\ \mu>1.\end{align}
Similarly, for fixed $\mu\geq1-\gamma,$ using \eqref{u-c3} and the definition of $A_{\mu}$, we have for $ y\in[y_1,z]$,
\begin{align*}&|\phi(y)-\phi(z)|\leq\int_y^z|\phi'(s)|ds\leq \int_y^z(u(s)-c)^{-\mu}B_{\mu}ds\leq C\delta_1^{\gamma}(u(y)-c)^{1-\gamma-\mu}B_{\mu}.
\end{align*}
This implies
\begin{align}\label{umu}&|(u-c)^{\mu+\gamma-1}\phi|_{L^{\infty}(z)}\leq C\delta_1^{\gamma}B_{\mu}+ C|\phi(z)|\ \ \text{for fixed}\ \mu\geq1-\gamma.\end{align}
Using \eqref{umu} for $\mu=1-\gamma$ and \eqref{Bmu} for $\mu=2-\gamma$, we have\begin{align*}&|\phi|_{L^{\infty}(z)}\leq C\delta_1^{\gamma}B_{1-\gamma}+ C|\phi(z)|\leq C\delta_1^{\gamma}(A_{2-\gamma}+|\phi'(z)|)+ C|\phi(z)|,\end{align*}which implies \eqref{phi1} by recalling the definition of $A_{\mu}$.
Using \eqref{umu} for $\mu=2-2\gamma$ and \eqref{Bmu} for $\mu=3-2\gamma$, we have\begin{align*}&|(u-c)^{1-\gamma}\phi|_{L^{\infty}(z)}+\delta_1^{\gamma}B_{2-2\gamma}\leq C\delta_1^{\gamma}B_{2-2\gamma}+ C|\phi(z)|\leq C\delta_1^{\gamma}(A_{3-2\gamma}+|\phi'(z)|)+ C|\phi(z)|,\end{align*}which implies \eqref{phi2} by recalling the definition of $A_{\mu}$ and $B_{\mu}$.

For fixed $\mu<1,$ using \eqref{u-c2} and the definition of $A_{\mu} $, we have for $y\in[y_1,z]$,
\begin{align*}&|\phi(y)-\phi(y_1)|\leq\int_{y_1}^y|\phi'(s)|ds\leq \int_{y_1}^y(u(s)-c)^{-\mu}B_{\mu}ds\leq C(u(y)-c)^{1-\mu}B_{\mu},
\end{align*}and
\begin{align*}&|\phi(y)|\leq|\phi(y_1)|+ C(u(y)-c)^{1-\mu}B_{\mu},\\&|(u(y)-c)^{\mu-1}\phi(y)|\leq(u(y)-c)^{\mu-1}|\phi(y_1)|+
CB_{\mu}\leq (u_{\min}-c)^{\mu-1}|\phi(y_1)|+ CB_{\mu},\end{align*}
where we used $u(y)-c\geq u_{\min}-c>0$ and $\mu-1<0.$ Thus,
\begin{align}\label{umu1}&|(u-c)^{\mu-1}\phi|_{L^{\infty}(z)}\leq (u_{\min}-c)^{\mu-1}|\phi(y_1)|+ CB_{\mu} \ \ \text{for fixed}\ \mu<1.\end{align}
Using \eqref{umu1} for $\mu=\gamma$ and \eqref{Bmu} for $\mu=1+\gamma$, we have
\begin{align*}|(u-c)^{\gamma-1}\phi|_{L^{\infty}(z)}&\leq (u_{\min}-c)^{\gamma-1}|\phi(y_1)|+ CB_{\gamma}
\\&\leq (u_{\min}-c)^{\gamma-1}|\phi(y_1)|+ C(A_{1+\gamma}+|\phi'(z)|),\end{align*}which implies \eqref{phi3} by recalling the definition of $A_{\mu}$.
Using \eqref{u-c3} for $\mu=1-\gamma$ and the definition of $A_{\mu}$, we have for $y\in[y_1,z]$,
\begin{align*}&|\phi'(y)-\phi'(z)|\leq\int_y^z|\phi''(s)|ds=\int_y^z|F(s)|ds\leq \int_y^z(u(s)-c)^{\gamma-1}A_{1-\gamma}ds\leq C\delta_1^{\gamma}A_{1-\gamma}.\end{align*}
Then by the definition of $B_{\mu}$, we have
\begin{align}\label{Bmu1}&B_{0}
=\sup_{y\in[y_1,z]}|\phi'(y)|\leq C\delta_1^{\gamma}A_{1-\gamma}+C|\phi'(z)| .\end{align}Using \eqref{umu1} for $\mu=0$ and \eqref{Bmu1}, we have
\begin{align*}|(u-c)^{-1}\phi|_{L^{\infty}(z)}&\leq (u_{\min}-c)^{-1}|\phi(y_1)|+ CB_{0}
\leq (u_{\min}-c)^{-1}|\phi(y_1)|+ C(\delta_1^{\gamma}A_{1-\gamma}+|\phi'(z)|),\end{align*}which implies \eqref{phi4} by recalling the definition of $A_{\mu}$.
\end{proof}\begin{lemma}\label{lem2} Let $u\in C^2([y_1,y_2])$, $u(y_1)=u_{\min}$ and $u'(y_1)\neq0$.  For fixed $ \gamma\in(0,1/2],$ there exist constants $C>0$ and $\delta_1>0$ such that if $z=y_1+\delta_1,\ 0<u_{\min}-c<1,$ $ \phi\in C^2([y_1,z])$ and\begin{align}\label{phi5}
-\phi''+{u''-\beta\over u-c}\phi=-\alpha^2\phi-F \ \ \text{on}\ [y_1,z],
\end{align}then the inequalities \eqref{phi1}--\eqref{phi4} are still true.\end{lemma}
\begin{proof}Let $ \widetilde{F}={u''-\beta\over u-c}\phi+\alpha^2\phi+F$ and $\delta_1\in(0,\delta]$ be given in Lemma \ref{lem1}. Then $\phi''=\widetilde{F}$ on $[y_1,z]$. By Lemma \ref{lem1}, \eqref{phi1}--\eqref{phi4} are still true with $F$ replaced by $\widetilde{F}$. As $|u''-\beta|\leq C$ and $|u-c|\leq C$, we have $ |F-\widetilde{F}|\leq C|\phi/(u-c)|$ for $y\in[y_1,z]$. Thus, for $\mu\in\R$, we have \begin{align}\label{mu1}|(u-c)^{\mu}\widetilde{F}|_{L^{\infty}(z)}&\leq |(u-c)^{\mu}{F}|_{L^{\infty}(z)}+|(u-c)^{\mu}(F-\widetilde{F})|_{L^{\infty}(z)}\\ \nonumber&\leq |(u-c)^{\mu}{F}|_{L^{\infty}(z)}+C|(u-c)^{\mu-1}\phi|_{L^{\infty}(z)}.
\end{align}Using \eqref{phi1} with $F$ replaced by $\widetilde{F}$ and \eqref{mu1} for $\mu=2-\gamma$, we have\begin{align}\label{phi-f1}&|\phi|_{L^{\infty}(z)}\leq C(\delta_1^{\gamma}|(u-c)^{2-\gamma}\widetilde{F}|_{L^{\infty}(z)}+|\phi(z)|+\delta_1^{\gamma}|\phi'(z)|)\\ \nonumber \leq& C\delta_1^{\gamma}|(u-c)^{1-\gamma}\phi|_{L^{\infty}(z)}+C(\delta_1^{\gamma}|(u-c)^{2-\gamma}{F}|_{L^{\infty}(z)}+|\phi(z)|+\delta_1^{\gamma}|\phi'(z)|)\\
\nonumber \leq& C_1\delta_1^{\gamma}|\phi|_{L^{\infty}(z)}+C(\delta_1^{\gamma}|(u-c)^{2-\gamma}{F}|_{L^{\infty}(z)}+|\phi(z)|+\delta_1^{\gamma}|\phi'(z)|).
\end{align}Using \eqref{phi2} with $F$ replaced by $\widetilde{F}$ and \eqref{mu1} for $\mu=3-2\gamma$, we have\begin{align}\label{phi-f2}&|(u-c)^{1-\gamma}\phi|_{L^{\infty}(z)}+\delta_1^{\gamma}|(u-c)^{2-2\gamma}\phi'|_{L^{\infty}(z)}\\ \nonumber\leq& C(\delta_1^{\gamma}|(u-c)^{3-2\gamma}\widetilde{F}|_{L^{\infty}(z)}+|\phi(z)|+\delta_1^{\gamma}|\phi'(z)|)\\ \nonumber \leq& C(\delta_1^{\gamma}|(u-c)^{3-2\gamma}F|_{L^{\infty}(z)}+|\phi(z)|+\delta_1^{\gamma}|\phi'(z)|)+C\delta_1^{\gamma}|(u-c)^{2-2\gamma}\phi|_{L^{\infty}(z)}
\\ \nonumber \leq& C(\delta_1^{\gamma}|(u-c)^{3-2\gamma}F|_{L^{\infty}(z)}+|\phi(z)|+\delta_1^{\gamma}|\phi'(z)|)+
C_1\delta_1^{\gamma}|(u-c)^{1-\gamma}\phi|_{L^{\infty}(z)}.
\end{align}Using \eqref{phi4} with $F$ replaced by $\widetilde{F}$ and \eqref{mu1} for $\mu=1-\gamma$, we have\begin{align}\label{phi-f3}&|(u-c)^{-1}\phi|_{L^{\infty}(z)} \leq C(\delta_1^{\gamma}|(u-c)^{1-\gamma}\widetilde{F}|_{L^{\infty}(z)}+|(u_{\min}-c)^{-1}\phi(y_1)|+|\phi'(z)|)\\ \nonumber \leq& C(\delta_1^{\gamma}|(u-c)^{1-\gamma}F|_{L^{\infty}(z)}+|(u_{\min}-c)^{-1}\phi(y_1)|+|\phi'(z)|)
+C\delta_1^{\gamma}|(u-c)^{-\gamma}\phi|_{L^{\infty}(z)}\\ \nonumber \leq& C(\delta_1^{\gamma}|(u-c)^{1-\gamma}F|_{L^{\infty}(z)}+|(u_{\min}-c)^{-1}\phi(y_1)|+|\phi'(z)|)
+C_1\delta_1^{\gamma}|(u-c)^{-1}\phi|_{L^{\infty}(z)}.
\end{align}Here, $C_1>0$ is a constant depending only on $ \gamma,\ \alpha,\ \beta,\ u,\ \delta$ (and independent of $ \delta_1$). Taking $ \delta_1\in(0,\delta]$ small enough such that $C_1\delta_1^{\gamma}\leq 1/2$ in \eqref{phi-f1}--\eqref{phi-f3}, we obtain \eqref{phi1}, \eqref{phi2} and \eqref{phi4}.

Note that $\gamma>0$ and $2-\gamma\geq \gamma+1.$
Using \eqref{mu1} for $\mu=\gamma+1$ and \eqref{phi1}, we have\begin{align*}\numberthis\label{F1}& |(u-c)^{\gamma+1}\widetilde{F}|_{L^{\infty}(z)}\leq |(u-c)^{\gamma+1}F|_{L^{\infty}(z)}+C|(u-c)^{\gamma}\phi|_{L^{\infty}(z)}\\ \leq& |(u-c)^{\gamma+1}F|_{L^{\infty}(z)}+C|\phi|_{L^{\infty}(z)}\\ \leq& |(u-c)^{\gamma+1}F|_{L^{\infty}(z)}+C(\delta_1^{\gamma}|(u-c)^{2-\gamma}F|_{L^{\infty}(z)}+|\phi(z)|+\delta_1^{\gamma}|\phi'(z)|)\\ \leq& C(|(u-c)^{\gamma+1}F|_{L^{\infty}(z)}+|\phi(z)|+|\phi'(z)|).
\end{align*}Using \eqref{phi3} with $F$ replaced by $\widetilde{F}$ and \eqref{F1}, we have
\begin{align*}&|(u-c)^{\gamma-1}\phi|_{L^{\infty}(z)}
\\ \leq& C(|(u-c)^{\gamma+1}\widetilde{F}|_{L^{\infty}(z)}+|(u_{\min}-c)^{\gamma-1}\phi(y_1)|+|\phi(z)|+|\phi'(z)|)
\\ \leq& C(|(u-c)^{\gamma+1}F|_{L^{\infty}(z)}+|(u_{\min}-c)^{\gamma-1}\phi(y_1)|+|\phi(z)|+|\phi'(z)|).
\end{align*}Thus, \eqref{phi3} is also true.
\end{proof}

\if Let $u$ satisfy ${\bf (H1)}$ and  be analytic at two endpoints. First,
we prove that $\kappa_{\pm}=\pm\infty$ if and only if $\sharp(\sigma_d(\mathcal{R}_{\alpha,\beta})\cap\mathbf{R})<\infty$ for all $\beta\in\mathbf{R}^{\pm}$, where $\kappa_{\pm}$ are defined in \eqref{def-kappa+} and \eqref{def-kappa-}. Here the assumption that $u_{\min}$ or $u_{\max}$ is not an embedding eigenvalue of $\mathcal{R}_{\alpha,\beta}$ for small $\alpha>0$ is dropped.\fi

We are now in a position to prove Theorem \ref{thm-flows good endpoints}.
\begin{proof} [Proof of Theorem \ref{thm-flows good endpoints}] We only prove (1), and the proof of (2) is similar. If $\{u'=0\}\cap\{u=u_{\min}\}\neq\emptyset$, then
$0<\kappa_+<\infty$. By Theorem \ref{main-result} (3), $\sharp(\sigma_d(\mathcal{R}_{\alpha,\beta})\cap\mathbf{R})=\infty$ for $\beta>{9\over 8}\kappa_+$.
If $\{u'=0\}\cap\{u=u_{\min}\}=\emptyset$, then $\{u=u_{\min}\}\subset\{y_1,y_2\}$ (i.e. $u=u_{\min}$ can be achieved only at the endpoints).
We assume that $u(y_1)=u_{\min}$. Then $u(y_2)>u(y_1)$ and $u'(y_1)>0$. By taking $\delta\in(0,y_2-y_1)$ smaller, we can assume that $u'>{u'(y_1)\over2}$ on $y\in[y_1,y_1+\delta]$.
Let $\psi_c, c\in\mathbf{C}\setminus\text{Ran}(u)$, be the solution of
\begin{align}\label{orginalRayleigh-Kuo}
-\partial_y^2\psi_c+{u''-\beta\over u-c}\psi_c=-\alpha^2\psi_c\ \ \text{on} \ \ [y_1,y_2],\ \psi_c(y_2)=0,\ \partial_y\psi_c(y_2)=1.
\end{align} Note that for $c\in(-\infty,u_{\min})$,
 $c\in\sigma_d(\mathcal{R}_{\alpha,\beta})$ if and only if $\psi_c(y_1)=0$.  Suppose that $\sigma_d(\mathcal{R}_{\alpha,\beta})\cap(-\infty,u_{\min})=\{c_k\}_{k=1}^{\infty}$.
 Then $c_k\to u_{\min}^-$. Note that if (iii) is true for $i=1$, then $\beta_1\leq0$ and thus $\beta\neq\beta_1$ for all $\beta>0$. So we divide the discussion into two cases.

Case 1. $\beta=\beta_1$ and (i) holds for $i=1$.

Case 2.  $\beta=\beta_1$ and (ii) holds for $i=1$; or $\beta\neq\beta_1$.



 If Case 1 is true,  then $u''-\beta=0 $ on $ [y_1,y_1+\delta]$. By \eqref{orginalRayleigh-Kuo}, $\psi_c $ can be extended to an analytic function in $\mathbf{C}\setminus u([y_1+\delta,y_2]).$ Since $u_{\min}\not\in u([y_1+\delta,y_2]),$  $\psi_c(y_1)$ has a finite number of zeros in a neighborhood of $c=u_{\min}$, which contradicts that $\sharp(\sigma_d(\mathcal{R}_{\alpha,\beta})\cap(-\infty,u_{\min}))={\infty}$.

 Now, we assume Case 2 is true. If $\beta\neq\beta_1 $, define $m=0;$ if $\beta=\beta_1 $ and (ii) is true, define $m=m_1.$ Then $m\geq 0,$ by taking $\delta>0$ smaller, we can assume that \begin{align}\label{u''}C^{-1}|y-y_1|^{m}\leq |u''(y)-\beta|\leq C|y-y_1|^{m}\quad \text{for}\quad y\in[y_1,y_1+\delta].\end{align}
 As $|u''(y)-\beta|/|y-y_1|^{m}\in C([y_1+\delta,y_2]),$ $|u''(y)-\beta|\leq C|y-y_1|^{m} $ is also true for $ y\in[y_1+\delta,y_2]$.
 Since $u\in C^1([y_1,y_2])$, we have 
 $ u(y)-u_{\min}=u(y)-u(y_1)=\int_{y_1}^yu'(z)dz=(y-y_1)v(y)$, here $v(y)=\int_0^1u'(y_1+s(y-y_1))ds$ and $v\in C([y_1,y_2])$. If $y\in(y_1,y_2]$, we have
 $u(y)>u_{\min}$ and $ v(y)>0$. If $y=y_1$, we have $v(y)=u'(y_1)>0$. Thus, $v(y)>0$ in $ [y_1,y_2]$  and there exists a constant $C>1$ such that
 $C^{-1}\leq v(y)\leq C,$ which implies \begin{align}\label{umin}C^{-1}|y-y_1|\leq u(y)-u_{\min}\leq C|y-y_1|,\end{align} 
 \begin{align}\label{uminc}C^{-1}|y-y_1|\leq u(y)-c\quad \text{for}\quad c\leq u_{\min}.\end{align}

 Let $n\in\mathbf{N}$ and $\psi_{c,n}=\partial_c^n\psi_{c}.$ By Rolle's Theorem, there exists $\{c_{k,n}\}_{k=1}^{\infty}\subset(-\infty,u_{\min})$ such that $\psi_{c_{k,n},n}(y_1)=0 $ and $c_{k,n}\to u_{\min}^-$ as $k\to\infty$. For fixed $c<u_{\min},$ let $k>0$ be large enough such that $c_{k,n}\in(c,u_{\min}).$ Then $\psi_{c,n}(y_1)=\int_{c_{k,n}}^c\partial_s\psi_{s,n}(y_1)ds=\int_{c_{k,n}}^c\psi_{s,n+1}(y_1)ds,$ and\begin{align}\label{ckn}
|\psi_{c,n}(y_1)|\leq\int_c^{c_{k,n}}|\psi_{s,n+1}(y_1)|ds\leq \int_c^{u_{\min}}|\psi_{s,n+1}(y_1)|ds.
\end{align}
Moreover, $\psi_{c,n} $ satisfies \begin{align}\label{orginalpsi_{c,n}}
-\partial_y^2\psi_{c,n}+{u''-\beta\over u-c}\psi_{c,n}=-\alpha^2\psi_{c,n}-F_{c,n},\quad F_{c,n}=\sum_{k=1}^n{n!\over (n-k)!}{u''-\beta\over (u-c)^{k+1}}\psi_{c,n-k},
\end{align}
where $\psi_{c,0}=\psi_{c}$ and $  F_{c,0}=0$.
Note that $\psi_c(y) $ is continuous on $(\mathbf{C}\setminus\text{Ran}(u))\times[y_1,y_2] $, and analytic in $c$.  Let $u_+(y)=\inf u([y,y_2])$. Then $\psi_c(y) $ can be extended to a continuous function on $D_1:=\{(c,y):c< u_+(y),\ y\in[y_1,y_2]\},$ still satisfying \eqref{orginalRayleigh-Kuo} in $D_1.$ Moreover, $ u_+$ is increasing and continuous on $[y_1,y_2],$ $u_+(y_1)=u_{\min}$ and $u_+(y)>u_{\min}$ for $ y\in(y_1,y_2].$ By standard theory of ODE, $\partial_c^n\psi_{c}=\psi_{c,n}\in C^2(D_1)$ and $\psi_{c,n}|_{D_1} $ is real-valued   for $n\in\mathbf{N}$. Using this extension, $\psi_c|_{c=u_{\min}}$ is well-defined and satisfies \eqref{orginalRayleigh-Kuo} for $y\in(y_1,y_2].$
For fixed $ \delta_1\in(0,\delta]$ and $n\in\mathbf{N}$, we have\begin{align}\label{psicn}
|\psi_{c,n}(y)|+|\partial_y\psi_{c,n}(y)|\leq C(n,\delta_1)\ \text{for}\ y\in[y_1+\delta_1,y_2]\text{ and } 0\leq u_{\min}-c\leq1,
\end{align}
since a continuous function is bounded in a compact set. Let  $m_0\in\Z$ be such that $ m_0-1<m\leq m_0$ and $ \gamma=(m+1-m_0)/2$. Then $m_0\geq 0$
and $\gamma\in(0,1/2].$ We claim that the following uniform bounds\begin{align}\label{uniformsi_{c,n}}
|\psi_{c,n}|\leq C,\ |\psi_{c,0}|\leq C|u-c|,\ |\psi_{c,1}|\leq C|u-c|^{1-\gamma},\ |\partial_y\psi_{c,m_0+2}|\leq C|u-c|^{2\gamma-2}
\end{align}
hold {for} $0<u_{\min}-c<1,\ y\in[y_1,y_2]$ and  $n\in\Z\cap[0,m_0+1].$ Assume that the uniform bounds \eqref{uniformsi_{c,n}} are true, which will be verified later.  Let $W_{c,n}=\partial_y\psi_{c,n}\psi_{c}-\partial_y\psi_{c}\psi_{c,n}.$
Then we get by \eqref{orginalpsi_{c,n}} that $\partial_yW_{c,n}=F_{c,n}\psi_{c}. $ By  \eqref{uniformsi_{c,n}} and using $ u(y_1)=u_{\min}$, we have for $0<u_{\min}-c<1$,\begin{align*}
|\psi_{c}(y_1)|\leq \int_c^{u_{\min}}|\psi_{s,1}(y_1)|ds\leq C\int_c^{u_{\min}}|u_{\min}-s|^{1-\gamma}ds\leq C|u_{\min}-c|^{2-\gamma},
\end{align*}
and thus \begin{align*}
|\partial_y\psi_{c,m_0+2}\psi_{c}|(y_1)=|\partial_y\psi_{c,m_0+2}(y_1)||\psi_{c}(y_1)|\leq C|u_{\min}-c|^{2\gamma-2}|u_{\min}-c|^{2-\gamma}=C|u_{\min}-c|^{\gamma},
\end{align*}
which implies $\lim\limits_{c\to u_{\min}^-}\partial_y\psi_{c,m_0+2}\psi_{c}(y_1)=0.$ Since $\psi_{c_{k,m_0+2},m_0+2}(y_1)=0 $ for $k\geq1$ and $c_{k,m_0+2}\to u_{\min}^-$,
we have $\liminf\limits_{c\to u_{\min}^-}|\partial_y\psi_{c}\psi_{c,m_0+2}|(y_1)=0,$ and thus $\liminf\limits_{c\to u_{\min}^-}|W_{c,m_0+2}|(y_1)=0.$

Since $\psi_c(y_2)=0,\ \partial_y\psi_c(y_2)=1$ and recall that $\psi_{c,n}=\partial_c^n\psi_{c}, $ we have $\psi_{c,n}(y_2)=0,\ \partial_y\psi_{c,n}$ $(y_2)=0 $
and $W_{c,n}(y_2)=0 $ for $n>0.$ Thus, $-W_{c,n}(y_1)=\int_{y_1}^{y_2}\partial_yW_{c,n}(y)dy=\int_{y_1}^{y_2}F_{c,n}\psi_{c}(y)dy.$ Note that\begin{align*}
&F_{c,n}-n!{u''-\beta\over (u-c)^{n+1}}\psi_{c}=\sum_{k=1}^{n-1}{n!\over (n-k)!}{u''-\beta\over (u-c)^{k+1}}\psi_{c,n-k}.\end{align*}
If $n\in\Z\cap[2,m_0+2]$, by \eqref{uniformsi_{c,n}} and \eqref{u''}, we have for $0<u_{\min}-c<1$ and $ y\in[y_1,y_2]$,\begin{align*}&\left|F_{c,n}-{n!(u''-\beta)\psi_{c}\over (u-c)^{n+1}}\right|\leq\sum_{k=1}^{n-1}{n!\over (n-k)!}{|u''-\beta|\over (u-c)^{k+1}}|\psi_{c,n-k}|\\ \leq&C\sum_{k=1}^{n-2}{|y-y_1|^m\over (u-c)^{k+1}}+C{|y-y_1|^m|u-c|^{1-\gamma}\over (u-c)^{n}}\leq C{|y-y_1|^m\over |u-c|^{n-1+\gamma}},\end{align*}\begin{align*}&\left|F_{c,n}\psi_{c}-{n!(u''-\beta)\psi_{c}^2\over (u-c)^{n+1}}\right|\leq C{|y-y_1|^m|\psi_{c}|\over |u-c|^{n-1+\gamma}}\leq C{|y-y_1|^m|u-c|\over |u-c|^{n-1+\gamma}}=C{|y-y_1|^m\over |u-c|^{n-2+\gamma}},
\end{align*}
and thus using $m-m_0=2\gamma-1$ and \eqref{uminc}, we have
\begin{align*}&\left|F_{c,m_0+2}\psi_{c}-{(m_0+2)!(u''-\beta)\psi_{c}^2\over (u-c)^{m_0+3}}\right|\leq C{|y-y_1|^m\over |u-c|^{m_0+\gamma}}\leq C{|y-y_1|^m\over |y-y_1|^{m_0+\gamma}}= C|y-y_1|^{\gamma-1}.
\end{align*}
 Integrating it on $[y_1,y_2]$ and using $-W_{c,n}(y_1)=\int_{y_1}^{y_2}F_{c,n}\psi_{c}(y)dy$, we have for $0<u_{\min}-c<1$, \begin{align*}\left|\int_{y_1}^{y_2}{(m_0+2)!(u''-\beta)\psi_{c}^2\over (u-c)^{m_0+3}}dy\right|&\leq \left|\int_{y_1}^{y_2}F_{c,m_0+2}\psi_{c}dy\right|+ C\int_{y_1}^{y_2}|y-y_1|^{\gamma-1}dy\\&\leq |W_{c,m_0+2}|(y_1)+C,
\end{align*}
and as $\liminf\limits_{c\to u_{\min}^-}|W_{c,m_0+2}|(y_1)=0$, we have\begin{align*}&\liminf\limits_{c\to u_{\min}^-}\left|\int_{y_1}^{y_2}{(m_0+2)!(u''-\beta)\psi_{c}^2\over (u-c)^{m_0+3}}dy\right|\leq C,&\liminf\limits_{c\to u_{\min}^-}\left|\int_{y_1}^{y_2}{(u''-\beta)\psi_{c}^2\over (u-c)^{m_0+3}}dy\right|\leq C.
\end{align*}
Since $u''-\beta $ is continuous and real-valued, and $C^{-1}|y-y_1|^{m}\leq |u''(y)-\beta|$,  it does not change sign on
$y\in[y_1,y_1+\delta]$. Then for $0<u_{\min}-c<1$,
\begin{align*}&\left|\int_{y_1}^{y_1+\delta}{(u''-\beta)\psi_{c}^2\over (u-c)^{m_0+3}}dy\right|
=\int_{y_1}^{y_1+\delta}{|u''-\beta|\psi_{c}^2\over (u-c)^{m_0+3}}dy\geq C^{-1}\int_{y_1}^{y_1+\delta}{|y-y_1|^{m}\psi_{c}^2\over (u-c)^{m_0+3}}dy.\end{align*}
As $|\psi_{c}|\leq C,\ |u''-\beta|\leq C$, $ u-c\geq u-u_{\min}\geq C^{-1}$ for $ y\in[y_1+\delta,y_2],$ and $0<u_{\min}-c<1$, we have \begin{align*}&\left|\int_{y_1+\delta}^{y_2}{(u''-\beta)\psi_{c}^2\over (u-c)^{m_0+3}}dy\right|
\leq\int_{y_1+\delta}^{y_2}Cdy\leq C.\end{align*}Thus,\begin{align*}&
\liminf\limits_{c\to u_{\min}^-}\int_{y_1}^{y_1+\delta}{|y-y_1|^{m}\psi_{c}^2\over (u-c)^{m_0+3}}dy\leq
C\liminf\limits_{c\to u_{\min}^-}\left|\int_{y_1}^{y_1+\delta}{(u''-\beta)\psi_{c}^2\over (u-c)^{m_0+3}}dy\right|\\ \leq&
C\liminf\limits_{c\to u_{\min}^-}\left|\int_{y_1}^{y_2}{(u''-\beta)\psi_{c}^2\over (u-c)^{m_0+3}}dy\right|+
C\limsup\limits_{c\to u_{\min}^-}\left|\int_{y_1+\delta}^{y_2}{(u''-\beta)\psi_{c}^2\over (u-c)^{m_0+3}}dy\right|\leq C.\end{align*}
Since $\psi_{c}\in C(D_1)$, we have for fixed $y\in(y_1,y_1+\delta]$,\begin{align*}&
\lim\limits_{c\to u_{\min}^-}\psi_{c}(y)=\psi_{u_{\min}}(y),&\lim\limits_{c\to u_{\min}^-}{|y-y_1|^{m}\psi_{c}^2\over (u-c)^{m_0+3}}
={|y-y_1|^{m}\psi_{u_{\min}}^2\over (u-u_{\min})^{m_0+3}}.\end{align*}Thus, by Fatou's Lemma, we have\begin{align*}&
\int_{y_1}^{y_1+\delta}{|y-y_1|^{m}\psi_{u_{\min}}^2\over (u-u_{\min})^{m_0+3}}dy=
\int_{y_1}^{y_1+\delta}\!\!\lim\limits_{c\to u_{\min}^-}{|y-y_1|^{m}\psi_{c}^2\over (u-c)^{m_0+3}}dy\leq
\liminf\limits_{c\to u_{\min}^-}\int_{y_1}^{y_1+\delta}{|y-y_1|^{m}\psi_{c}^2\over (u-c)^{m_0+3}}dy\leq C.\end{align*}
By \eqref{umin}, we have\begin{align*}&
{|y-y_1|^{m}\psi_{u_{\min}}^2\over (u-u_{\min})^{m_0+3}}\geq {|y-y_1|^{m}\psi_{u_{\min}}^2\over (C(y-y_1))^{m_0+3}}
\geq {C^{-1}\psi_{u_{\min}}^2\over |y-y_1|^{m_0-m+3}}\geq {C^{-1}\psi_{u_{\min}}^2\over |y-y_1|^{3}}\end{align*}
for $y\in (y_1,y_1+\delta]$, where we used $m\leq m_0$. Thus,\begin{align*}&
\int_{y_1}^{y_1+\delta}{\psi_{u_{\min}}^2\over |y-y_1|^{3}}dy\leq C\int_{y_1}^{y_1+\delta}{|y-y_1|^{m}\psi_{u_{\min}}^2\over (u-u_{\min})^{m_0+3}}dy\leq C.\end{align*}
Now we take $ \varphi=\psi_{u_{\min}}$. Then $\varphi $ is real-valued  and  for $y\in(y_1,y_2]$, it satisfies \begin{align}\label{eq:phi}&
-\varphi''+{u''-\beta\over u-u_{\min}}\varphi=-\alpha^2\varphi,\ \varphi(y_2)=0,\ \varphi'(y_2)=1,\
\int_{y_1}^{y_1+\delta}{\varphi^2\over |y-y_1|^{3}}dy\leq C.\end{align}Thus, $\varphi,\varphi/|y-y_1|\in L^2(y_1,y_1+\delta). $ By \eqref{umin}, we have  for $y\in (y_1,y_1+\delta]$,\begin{align*}&
\left|{u''-\beta\over u-u_{\min}}\varphi\right|\leq {C|\varphi|\over u-u_{\min}}\leq {C|\varphi|\over |y-y_1|},\quad
{u''-\beta\over u-u_{\min}}\varphi\in L^2(y_1,y_1+\delta).\end{align*}Thus, $\varphi''\in L^2(y_1,y_1+\delta),\ \varphi\in H^2(y_1,y_1+\delta)
$ and $ \varphi\in C^1([y_1,y_1+\delta])$ by defining $ \varphi(y_1)=\lim_{y\to y_1^+}\varphi(y)$. If $ \varphi(y_1)\neq 0,$
then there exists $ \delta_1\in(0,\delta]$ such that $ |\varphi(y)|\geq |\varphi(y_1)|/2\geq C^{-1}|y-y_1|$ for $y\in (y_1,y_1+\delta_1]$. If $\varphi(y_1)=0$ and $ \varphi'(y_1)\neq 0,$
then there exists $ \delta_1\in(0,\delta]$ such that $ |\varphi'(y)|\geq |\varphi'(y_1)|/2$ for $y\in (y_1,y_1+\delta_1]$, and
$ |\varphi(y)|=|\varphi(y)-\varphi(y_1)|=|(y-y_1)\varphi'(\xi_y)|\geq |y-y_1||\varphi'(y_1)|/2$ for $\xi_y\in(y_1,y)$ and $y\in (y_1,y_1+\delta_1]$. Therefore,  if $\varphi(y_1)\neq 0$ or $ \varphi'(y_1)\neq 0,$
then there exists $ \delta_1\in(0,\delta]$ and $ C>0$ such that
$ |\varphi(y)|\geq C^{-1}|y-y_1|$ for $y\in (y_1,y_1+\delta_1]$, and\begin{align*}&
\int_{y_1}^{y_1+\delta}{\varphi^2\over |y-y_1|^{3}}dy\geq
C^{-2}\int_{y_1}^{y_1+\delta_1}{|y-y_1|^2\over |y-y_1|^{3}}dy=+\infty,\end{align*}which contradicts \eqref{eq:phi}.
Thus, we must have $\varphi(y_1)=\varphi'(y_1)= 0.$ Then by the proof of Lemma 3 in \cite{LYZ}, we have
 $\varphi\equiv0$ on $[y_1,y_2]$, which contradicts $\varphi'(y_2)=1 $. This proves (1)  for Case 2.

It remains to prove
 \eqref{uniformsi_{c,n}}.
Let $ \delta_1\in(0,\delta]$ be fixed such that Lemma \ref{lem2} is true. Recall that $z=y_1+\delta_1$ and $|f|_{L^{\infty}(z)}=\sup_{y\in[y_1,z]}|f(y)|. $
 By \eqref{psicn} we know that \eqref{uniformsi_{c,n}} is true for $y\in[z,y_2]$. Now we assume that $y\in[y_1,z], 0<u_{\min}-c<1,$ and that \eqref{phi1}--\eqref{phi4} are used for $F$ satisfying \eqref{phi5} (i.e. the condition in Lemma \ref{lem2}).
The proof of \eqref{uniformsi_{c,n}} for $y\in[y_1,z]$  is divided into 7 steps as follows.
\\
 {\bf Step 1.} $|\psi_{c,n}|\leq C$ for $y\in[y_1,z]$ and $n\in\Z\cap[0,m_0].$

 For $n=0,$ by \eqref{orginalpsi_{c,n}}, \eqref{phi1}, \eqref{psicn} and $F_{c,0}=0 $, we have $|\psi_{c,0}|_{L^{\infty}(z)}\leq C(|\psi_{c,0}(z)|+\delta_1^{\gamma}|\psi_{c,0}'(z)|)\leq C$, and thus $|\psi_{c,0}|\leq C$ for $y\in[y_1,z].$ Now, we prove the result by induction. Assume that $n\in\Z\cap(0,m_0]$ and $|\psi_{c,k}|\leq C$ for $k\in\Z\cap[0,n)$ and $y\in[y_1,z].$ Then by \eqref{orginalpsi_{c,n}}, \eqref{u''}, \eqref{uminc} and $m-m_0=2\gamma-1$, we have for $y\in[y_1,z]$,\begin{align}\label{Fcn}&|F_{c,n}|\leq C\sum_{k=1}^n{|y-y_1|^m|\psi_{c,n-k}|\over (u-c)^{k+1}}\leq C\sum_{k=1}^n{|y-y_1|^m\over (u-c)^{k+1}}\leq C{(u-c)^m\over (u-c)^{n+1}},\\&\label{FCn1}|(u-c)^{2-\gamma}F_{c,n}|\leq C(u-c)^{2-\gamma+m-n-1}\leq  C(u-c)^{m-m_0+1-\gamma}=C(u-c)^{\gamma}\leq C.
\end{align}
By \eqref{phi1}, \eqref{psicn} and \eqref{FCn1}, we have
  \begin{align}\label{Fcn2}|\psi_{c,n}|_{L^{\infty}(z)}\leq C(\delta_1^{\gamma}|(u-c)^{2-\gamma}F_{c,n}|_{L^{\infty}(z)}+|\psi_{c,n}(z)|+\delta_1^{\gamma}|\psi_{c,n}'(z)|)\leq C,\end{align}
   which means $|\psi_{c,n}|\leq C$ for $y\in[y_1,z].$ Thus, the result in { Step 1} is true.
\\
  {\bf Step 2.} $|\psi_{c,n}|\leq C|u-c|^{\gamma-1}$ for $y\in[y_1,z]$ and $n=m_0+1.$

 Let $n=m_0+1.$
  By { Step 1}, we know that $|\psi_{c,k}|\leq C$ for $y\in[y_1,z]$ and $k\in\Z\cap[0,n)$. Thus, \eqref{Fcn} is still true and for $y\in[y_1,z]$,\begin{align*}&|(u-c)^{3-2\gamma}F_{c,n}|\leq C(u-c)^{3-2\gamma+m-n-1}=  C(u-c)^{m-m_0+1-2\gamma}=C,
\end{align*}
 which, along with \eqref{phi2} and \eqref{psicn}, implies that \begin{align*}|(u-c)^{1-\gamma}\psi_{c,n}|_{L^{\infty}(z)}\leq C(\delta_1^\gamma|(u-c)^{3-2\gamma}F_{c,n}|_{L^{\infty}(z)}+|\psi_{c,n}(z)|+\delta_1^\gamma|\psi_{c,n}'(z)|)\leq C.\end{align*}Then $|(u-c)^{1-\gamma}\psi_{c,n}|\leq C$, and thus $|\psi_{c,n}|\leq C|u-c|^{\gamma-1}$ for $y\in[y_1,z].$
\\
  {\bf Step 3.} $|\psi_{c,0}|\leq C|u-c|^{1-\gamma}$ for $y\in[y_1,z]$.

  If $m_0=0,$ then $m=0$ and $\gamma=1/2.$ By { Step 2}, we have $|\psi_{c,1}|\leq C|u-c|^{\gamma-1}=C|u-c|^{-\gamma}$ for $y\in[y_1,z]$. If $m_0>0,$ then by { Step 1}, we have $|\psi_{c,1}|\leq C\leq C|u-c|^{-\gamma}$ for $y\in[y_1,z]$. Thus, $|\psi_{c,1}|\leq C|u-c|^{-\gamma}$ is always true for $y\in[y_1,z]$. Then by \eqref{ckn} and $u(y_1)=u_{\min}$, we have $|\psi_{c,0}(y_1)|\leq \int_c^{u_{\min}}|\psi_{s,1}(y_1)|ds\leq C\int_c^{u_{\min}}|u_{\min}-s|^{-\gamma}ds\leq C|u_{\min}-c|^{1-\gamma}$. By \eqref{phi3}, \eqref{psicn} and $F_{c,0}=0 $, we have $|(u-c)^{\gamma-1}\psi_{c,0}|_{L^{\infty}(z)}\leq C(|(u_{\min}-c)^{\gamma-1}\psi_{c,0}(y_1)|+|\psi_{c,0}(z)|+|\psi_{c,0}'(z)|)\leq C$. Then $|(u-c)^{\gamma-1}\psi_{c,0}|\leq C$, and thus $ |\psi_{c,0}|\leq C|u-c|^{1-\gamma}$ for $y\in[y_1,z].$
\\
 {\bf Step 4.} $|\psi_{c,n}|\leq C$ for $y\in[y_1,z]$ and $n=m_0+1.$

 Let $n=m_0+1.$ By \eqref{orginalpsi_{c,n}}, \eqref{u''}, \eqref{umin}, { Step 1}  and { Step 3}, we have for $y\in[y_1,z]$,  \begin{align*}&|F_{c,n}|\leq C\sum_{k=1}^n{|y-y_1|^m|\psi_{c,n-k}|\over (u-c)^{k+1}}\leq C\sum_{k=1}^{n-1}{|y-y_1|^m\over (u-c)^{k+1}}+C{|y-y_1|^m|\psi_{c,0}|\over (u-c)^{n+1}}\\ \leq& C{(u-c)^m\over (u-c)^{n}}+C{|y-y_1|^m|u-c|^{1-\gamma}\over (u-c)^{n+1}}\leq C(u-c)^{m-\gamma-n},\\&|(u-c)^{2-\gamma}F_{c,n}|\leq C(u-c)^{2-2\gamma+m-n}=  C(u-c)^{m-m_0+1-2\gamma}=C.
\end{align*}Here, we used $n=m_0+1$ and $m-m_0=2\gamma-1$. Thus, \eqref{Fcn2} is still true for $n=m_0+1,$ i.e. $|\psi_{c,n}|\leq C$ for $y\in[y_1,z]$.
\\
  {\bf Step 5.} $|\psi_{c,0}|\leq C|u-c|$ for $y\in[y_1,z]$.

  By { Step 1} and { Step 4}, we have $|\psi_{c,1}|\leq C$ for $y\in[y_1,z]$. Then by \eqref{ckn}, we have $|\psi_{c,0}(y_1)|\leq \int_c^{u_{\min}}|\psi_{s,1}(y_1)|ds\leq C\int_c^{u_{\min}}ds\leq C|u_{\min}-c|$. By \eqref{phi4}, \eqref{psicn} and $F_{c,0}=0 $, we have $|(u-c)^{-1}\psi_{c,0}|_{L^{\infty}(z)}\leq C(|(u_{\min}-c)^{-1}\psi_{c,0}(y_1)|+|\psi_{c,0}'(z)|)\leq C$, which gives $|(u-c)^{-1}\psi_{c,0}|\leq C$ and  $|\psi_{c,0}|\leq C|u-c|$ for $y\in[y_1,z].$
\\
 {\bf Step 6.} $|\partial_y\psi_{c,n}|\leq C|u-c|^{2\gamma-2}$ and  $|\psi_{c,n}|\leq C|u-c|^{\gamma-1}$ for $y\in[y_1,z]$ and $n=m_0+2.$

  Since $n=m_0+2$ and $m-m_0=2\gamma-1$,  we have by \eqref{orginalpsi_{c,n}}, \eqref{u''}, \eqref{uminc}, { Step 1} and { Steps 4-5} that for $y\in[y_1,z]$, \begin{align*}&|F_{c,n}|\leq C\sum_{k=1}^n{|y-y_1|^m|\psi_{c,n-k}|\over (u-c)^{k+1}}\leq C\sum_{k=1}^{n-1}{|y-y_1|^m\over (u-c)^{k+1}}+C{|y-y_1|^m|\psi_{c,0}|\over (u-c)^{n+1}}\\ \leq& C{(u-c)^m\over (u-c)^{n}}+C{|y-y_1|^m|u-c|\over (u-c)^{n+1}}\leq C(u-c)^{m-n}=C(u-c)^{m-m_0-2},\\&|(u-c)^{3-2\gamma}F_{c,n}|\leq C(u-c)^{1-2\gamma+m-m_0}=C.
\end{align*}Then by \eqref{phi2} and \eqref{psicn}, we have\begin{align*}&|(u-c)^{1-\gamma}\psi_{c,n}|_{L^{\infty}(z)}+\delta_1^{\gamma}|(u-c)^{2-2\gamma}\partial_y\psi_{c,n}|_{L^{\infty}(z)}\\ \leq& C(|(u-c)^{3-2\gamma}F_{c,n}|_{L^{\infty}(z)}+|\psi_{c,n}(z)|+|\partial_y\psi_{c,n}(z)|)\leq C.\end{align*}Therefore, $|\partial_y\psi_{c,n}|\leq C\delta_1^{-\gamma}|u-c|^{2\gamma-2}\leq C|u-c|^{2\gamma-2}$ and $|\psi_{c,n}|\leq C|u-c|^{\gamma-1}$ for $y\in[y_1,z].$
\\
  {\bf Step 7.} $|\psi_{c,1}|\leq C|u-c|^{1-\gamma}$ for $y\in[y_1,z]$.

   If $m_0=0,$ then $m=0$ and $\gamma=1/2.$ By {Step 6}, we have
  $|\psi_{c,2}|\leq C|u-c|^{\gamma-1}=C|u-c|^{-\gamma}$ for $y\in[y_1,z]$. If $m_0>0,$ then by {Step 1} and {Step 4}, we have $|\psi_{c,2}|\leq C\leq C|u-c|^{-\gamma}$ for $y\in[y_1,z]$.
  Thus, $|\psi_{c,2}|\leq C|u-c|^{-\gamma}$ is always true for $y\in[y_1,z]$. Then by \eqref{ckn} and $u(y_1)=u_{\min}$, we have $|\psi_{c,1}(y_1)|\leq \int_c^{u_{\min}}|\psi_{s,2}(y_1)|ds\leq C\int_c^{u_{\min}}|u_{\min}-s|^{-\gamma}ds\leq C|u_{\min}-c|^{1-\gamma}$. By \eqref{orginalpsi_{c,n}}, \eqref{u''} and { Step 5}, we have  for $y\in[y_1,z]$, \begin{align*}&|F_{c,1}|\leq C{|y-y_1|^m|\psi_{c,0}|\over (u-c)^{2}}\leq C{|y-y_1|^m|u-c|\over (u-c)^{2}}=C{|y-y_1|^m\over u-c},\\&|(u-c)^{\gamma+1}F_{c,1}|\leq C(u-c)^{\gamma}|y-y_1|^m\leq C.
\end{align*}Then by \eqref{phi3} and \eqref{psicn}, we have\begin{align*}&|(u-c)^{\gamma-1}\psi_{c,1}|_{L^{\infty}(z)}\\ \leq& C(|(u-c)^{\gamma+1}F_{c,1}|_{L^{\infty}(z)}+|(u_{\min}-c)^{\gamma-1}\psi_{c,1}(y_1)|+|\psi_{c,1}(z)|+|\partial_y\psi_{c,1}(z)|)\leq C,\end{align*}which gives $|(u-c)^{\gamma-1}\psi_{c,1}|\leq C$ and $|\psi_{c,1}|\leq C|u-c|^{1-\gamma}$ for $y\in[y_1,z].$

By { Step 1} and { Steps 4-7} we know that \eqref{uniformsi_{c,n}} is true for $y\in[y_1,z]=[y_1,y_1+\delta_1]$. This completes the proof of \eqref{uniformsi_{c,n}} and thus Case 2.
\end{proof}

\subsection{Rule out oscillation for flows in class $\mathcal{K}^+$}
We rule out the oscillation of $\lambda_n(c)$ for flows in class $\mathcal{K}^+$, which  is stated in Theorem \ref{number for sinus flow}.
The proof is based on Hamiltonian structure and index theory.

\begin{proof}[Proof of Theorem \ref{number for sinus flow}] The assumption ({\bf{H1}}) is satisfied for a flow $u$  in class $\mathcal{K}^+$.
By  Theorem \ref{main-result}, it suffices to prove $\sharp(\sigma_d(\mathcal{R}_{\alpha,\beta})\cap(-\infty,u_{\min}))<\infty$ for $0<\alpha^2\leq M_{\beta}$ and  $ 0<\beta\leq {9\over 8}\kappa_+$. Similar proof is valid for $ {9\over 8}\kappa_-\leq\beta<0$. First, we consider $\beta\in \text{Ran} (u)\cap(0,{9\over 8}\kappa_+]$. Define the non-shear space
$$X:=\{\omega\in L^2(D_T):\int_0^T\omega(x,y)dx=0, \text{T-periodic in }x  \}.$$
Note that as $\omega=\partial_{x}v_{2}-\partial_{y}v_{1}$, $\int_0^T\omega(x,y)dx=0$ is equivalent to $\int_0^Tv_1(x,y)dx=$constant. Thus, $\int_0^Tv_1(x,y)dx=0$ implies  $\int_0^T\omega(x,y)dx=0$.

The linearized equation (\ref{linearized Euler equation}) has a Hamiltonian structure
in the traveling frame $(x-u_\beta t,y,t)$:
\begin{align*}
\omega_{t}=-(\beta-u^{\prime\prime})\partial_{x}\left(  {\omega}/%
{K_{\beta}}-\psi\right)  =JL\omega,
\end{align*}
where
$
J=-(\beta-u^{\prime\prime})\partial_{x}:X^*\rightarrow X,L=
{1}/{K_{\beta}}-\left(  -\Delta\right)  ^{-1}:X\rightarrow X^*.
$
Let $
J_{\alpha}=-i\alpha(\beta-u^{\prime\prime})$ and  $L_{\alpha}=\frac
{1}{K_{\beta} }-(  -\frac{d^{2}}{dy^{2}}+\alpha^{2})^{-1}
$ on
$L_{\frac
{1}{K_{\beta} }}^{2}$.
It follows from Theorem 3 in \cite{LYZ}
  that
 $$k_c+k_r+k_i^{\leq0}=n^-(L_\alpha),$$
where $n^-(L_\alpha)$ is the Morse index of $L_\alpha$,
${k}%
_{r}$ is the sum of algebraic multiplicities of positive eigenvalues of
${J}_{\alpha}{L}_{\alpha}$, ${k}_{c}$ is the sum of algebraic multiplicities
of eigenvalues of ${J}_{\alpha}{L}_{\alpha}$ in the first and the fourth
quadrants and ${k}_{i}^{\leq0}$ is the total number of non-positive dimensions of
$\langle{L}_{\alpha}\cdot,\cdot\rangle$ restricted to the
generalized eigenspaces of nonzero purely imaginary eigenvalues of
${J}_{\alpha}{L}_{\alpha}$.

Suppose that $\sharp(\sigma_d(\mathcal{R}_{\alpha,\beta})\cap(-\infty,u_{\min}))=\infty$.
 Then it follows from  Theorem \ref{eigenvalue asymptotic behavior-bound} that there exists $m_{\beta}<n<N_{\beta}$ such that $\sharp(\{\lambda_n(c)=-\alpha^2, c<u_{\min}\})=\infty$.
  Let $c^*<u_{\min}$ be a solution of $\lambda_n(c)=-\alpha^2$ with eigenfunction $ \phi^*$.
  $c^*$ can be chosen  sufficiently close to $u_{\min}$. Then $-i\alpha(c^*-u_\beta)$ is a purely imaginary eigenvalue of $J_\alpha L_\alpha$ with eigenfunction $\omega^*=-{\phi^{*}}''+\alpha^2\phi^*$.
By Theorem 4 in \cite{LYZ},
$$\langle L_\alpha\omega^*,\omega^*\rangle=-(c^*-u_\beta)\lambda'_n(c^*).$$
Note that $c-u_\beta$ does not change sign when $c<u_{\min}$ is sufficiently close to $u_{\min}$.
Then
\begin{align*}
\sharp(\{-(c-u_\beta)\lambda'_n(c)\leq 0,c<u_{\min}\}\cap \{\lambda_n(c)=-\alpha^2, c<u_{\min}\})=\infty.
 \end{align*}
Hence, $k_i^{\leq0}=\infty$. This contradicts that
\begin{align*}
k_i^{\leq0}\leq n^-(L_\alpha)=n^-(\tilde L_0+\alpha^2)\leq n^-(\tilde L_0)<\infty,
\end{align*}
 where $\tilde L_0=-{d^2\over dy^2}-K_\beta: H^2\cap H_0^1\to L^2$. Therefore, $\sharp(\sigma_d(\mathcal{R}_{\alpha,\beta})\cap(-\infty,u_{\min}))<\infty$.

Then, we consider $\beta\in (0,{9\over 8}\kappa_+]\setminus\text{Ran} (u'')$. By Corollary 1 in \cite{LYZ}, $\lambda_n(c)$ is decreasing on
$c\in (-\infty,u_{\min})$ for any fixed $n\geq 1$. By Theorem \ref{eigenvalue asymptotic behavior-bound}, $\sharp(\sigma_d(\mathcal{R}_{\alpha,\beta})\cap(-\infty,u_{\min}))<N_{\beta}$.
\end{proof}

\section{Relations between a traveling wave family  and an isolated real eigenvalue}
\label{appendix-traveling wave concentration}

In this section, we establish the correspondence between a traveling wave family near a shear flow and an isolated real eigenvalue of $\mathcal{R}_{k\alpha,\beta}$.
For a given  isolated real eigenvalue $c_0$,
 we prove that there exists  a set of
 traveling wave solutions near $(u,0)$ with traveling speeds converging to  $c_0$, which is stated precisely in Lemma \ref{traveling wave construction-lemma}.
 We assume $k,k_0\in\Z$.

\begin{proof}[Proof of Lemma \ref{traveling wave construction-lemma}]
We assume that  $\beta>0$, and the case for $\beta<0$ is similar.
Since $c_0\in\sigma_d(\mathcal{R}_{k_0\alpha,\beta})\cap\mathbf{R}$ for some $k_0\geq1$, we have $c_0<u_{\min}$ and we choose $\delta_0>0$ such that $c_0+\delta_0<u_{\min}$.
By (\ref{vorticity-eqn}),  $\vec{u}\left(  x-c t,y\right)  $ is a solution of \eqref{Euler equation}--\eqref{boundary condition for euler} if
and only if $(\psi,c)$ solves
\begin{align}\label{j-traveling}
\frac{\partial\left(  \omega+\beta y,\psi-c y\right)  }{\partial\left(
x,y\right)  }=0
\end{align}
and $\psi$ takes constant values on $\left\{  y=y_i\right\}  $, where  $i=1,2$,
 $\omega=\operatorname{curl}%
\vec{u}$ and $\vec{u}=(\pa_y\psi,-\pa_x\psi)$.
Let $\psi_{0}$ be a stream function associated with
the shear flow $\left(  u,0\right)  $, i.e., $\psi_{0}^{\prime}=u$.
Since $u-c>0$ for  $c\in[c_0-\delta_0,c_0+\delta_0]$, $\psi_{0}-cy$ is increasing on $[y_1,y_2]$.
Let $I_c=\{\psi_{0}\left(  y\right)-cy:y\in[y_1,y_2]\}$  for  $c\in[c_0-\delta_0,c_0+\delta_0]$, and then we can define a function $\tilde  f_{c}\in
C^{2}(I_c)$ such that
\begin{equation}
\tilde f_{c}\left(  \psi_{0}\left(  y\right) -cy \right)  =\omega_{0}\left(  y\right)
+\beta y=-\psi_{0}^{\prime\prime}\left(  y\right)  +\beta y.
\label{eqn-f-psi-0}%
\end{equation}
Moreover,
\[
\tilde f'_{c}\left(  \psi_{0}\left(  y\right)-cy \right)  =\frac{\beta
-u^{\prime\prime}\left(  y\right)  }{u\left(  y\right)  -c}=:\mathcal{K}%
_{c}\left(  y\right)
\]
for  $c\in[c_0-\delta_0,c_0+\delta_0]$.
We extend $\tilde f_{c}$ to $f_c\in C_{0}^{2}\left(  \mathbf{R}\right)  $ such
that $f_c=\tilde f_{c}$ on $I_c$ and $\partial_z^2\partial_c f_c(z)$ is continuous  for   $c\in[c_0-\delta_0,c_0+\delta_0]$ and $z\in \mathbf{R}$. Taking
$c$ as the bifurcation parameter, we now construct steady solutions $\vec{u}  =\left(  \partial_y\psi,-\partial_x\psi\right)  $ near
$\left(  u,0\right)  $ by solving the elliptic equations
\begin{align}\label{eqn-psi-tilde}
-\Delta\psi+\beta y=f_c\left(  \psi-cy\right)
\end{align}
with the boundary conditions that ${\psi}$ takes constant values on
$\left\{  y=y_i\right\}  $, $i=1,2$.
  Define the perturbation of the stream function
by
\[
\phi\left(  x,y\right)  ={\psi}\left(  x,y\right)  -\psi_{0}\left(
y\right).
\]
Then by \eqref{eqn-f-psi-0}--\eqref{eqn-psi-tilde}, we have
\begin{equation*}
-\Delta \phi-\left(  f_c(\phi+\psi_{0}-cy)-f_c\left(  \psi_{0}-cy\right)  \right)
=0.
\end{equation*}
Define the spaces%
\begin{align*}
B= \{\varphi\in H^{4}(D_T):\text{ }\varphi(x
,y_i)=0,\,i=1,2,
\;
\varphi \text{ is even and }  T\text{-periodic in }x\}
\end{align*}
and
\[
C=\left\{  \varphi\in H^{2}(D_T):\text{ }%
T\text{-periodic in }x\right\},
\]
where $T={2\pi/ \alpha}$. Consider the mapping
\begin{align*}
&F:B\times[c_0-\delta_0,c_0+\delta_0]\longrightarrow C,\\
 &\;\;\;\;\;\;\;\;\;\;\;\;\;\;\;\;\;\;\;\;\;\;\;\;\;\;\;\;\;\;(\phi,c)\longmapsto-\Delta\phi-\left(  f_c(\phi+\psi_{0}-cy)-f_c\left(
\psi_{0}-cy\right)  \right)  .
\end{align*}
Then $F(0,c)=0$ for $c\in[c_0-\delta_0,c_0+\delta_0]$.  We study the bifurcation near the trivial solution $(0,c_0)$ of the equation
$F(\phi,c)=0$ in $B$, whose solutions give steady flows of \eqref{j-traveling}.\

For fixed $c\in[c_0-\delta_0,c_0+\delta_0]$, by linearizing $F$
around $\phi=0$, we have
\[
\partial_{\phi}F(0,c)=-\Delta-f_{c}^{\prime}(\psi
_{0}-cy)=-\Delta-\mathcal{K}_{c}=\mathcal{G}_c|_B,
\]
where $\mathcal{G}_c|_B$ is the restriction of $\mathcal{G}_c$ in $B$ and $\mathcal{G}_c$ is defined in \eqref{Glinearized elliptic operator}.
Then we divide the discussion of bifurcation near  $(0,c_0)$ of the equation
$F(\phi,c)=0$ into three cases.
Since $c_0\in\sigma_d(\mathcal{R}_{k_0\alpha,\beta})\cap\mathbf{R}$, there exists $n_0\geq1$ such that
 $(k_0\alpha)^2=-\lambda_{n_0}(c_0)$, where $\lambda_{n_0}(c_0)$ is the $n_0$-th eigenvalue of $\mathcal{L}_{c_0}$
and $\mathcal{L}_{c_0}$
 is defined in \eqref{sturm-Liouville}. Let
\begin{align}\label{def-k*-constructed traveling wave}
k_*=&\max\limits_{k\geq1}\{k:\text{there exists } n_k\geq1 \text{ such that } -(k\alpha)^2=\lambda_{n_k}(c_0)\}.
\end{align}
 Then $k_*$ exists by our assumption and $1\leq k_0\leq k_*< \infty$. Now we denote $n_*=n_{k_*}.$
\\
{\bf Case 1.} $\lambda_{n_*}'(c_0)\neq0$ (the transversal crossing condition)
and $c_0\notin\sigma_d(\mathcal{R}_{0,\beta})\cap\mathbf{R}$.

 In this case, we have $0\notin\sigma(\mathcal{L}_{c_0})$.
Let $B_*=\{\varphi\in B:{2\pi\over k_*\alpha}\text{-periodic in } x\}$ and
$C_*=\{\varphi\in C:{2\pi\over k_*\alpha}\text{-periodic in } x\}$.
Consider the restriction $F|_{B_*}$ and $\mathcal{G}_c|_{B_*}$. Then
by the definition of $k_*$, we have
\begin{align}\label{dim-ker-G-c}
\ker(%
{\mathcal{G}_{c_0}|_{B_*}})={\rm{span}}\{\phi_{c_0,n_*}(y)\cos(k_*\alpha x)\}\quad\text{and}\quad
\dim(\ker(\mathcal{G}_{c_0}|_{B_*}))=1,
\end{align}
where $\phi_{c_0,n_*}$ is a real-valued eigenfunction of
$\lambda_{n_*}(c_0)\in\sigma(\mathcal{L}_{c_0})$.
Note that
\begin{align*}
\partial_{c}\partial_{\phi}F(0,c_0)\left(  \phi_{c_0,n_*}%
(y)\cos(k_*\alpha x)\right)  =-{\beta-u''\over (u-c_0)^2}  \phi
_{c_0,n_*}(y)\cos(k_*\alpha x).
\end{align*}
Then by Lemma $11$ in \cite{LYZ}, we have
\begin{align*}
&\int_0^T\int_{y_1}^{y_2} \phi_{c_0,n_*}(y)\cos(k_*\alpha x)\left[\partial_{c}\partial_{\phi}F(0,c_0)\left(  \phi_{c_0,n_*}%
(y)\cos(k_*\alpha x)\right)\right] dydx\\
 =&-\int_0^T\int_{y_1}^{y_2}{\beta-u''\over (u-c_0)^2}  |\phi
_{c_0,n_*}(y)|^2\cos^2(k_*\alpha x)dydx
={\pi\over\alpha}\lambda_{n_*}'(c_0)\neq0,
\end{align*}
where we used that $\phi_{c_0,n_*}$ is real-valued. By \eqref{dim-ker-G-c}, we have  $\phi_{c_0,n_*}(y)\cos(k_*\alpha x)\in\ker(%
{\mathcal{G}_{c_0}|_{B_*}})$ and thus, $\partial_{c}\partial_{\phi}F(0,c_0)\left(  \phi_{c_0,n_*}%
(y)\cos(k_*\alpha x)\right)\notin$ ${\rm Ran}\,(%
{\mathcal{G}_{c_0}|_{B_*}})$. Then by Theorem $1.7$ in \cite{CR71}, there exist $\delta>0$ and
a nontrivial $C^1$ bifurcating curve $\{\left(  \phi_{\gamma},c(\gamma)\right), \gamma\in(-\delta,\delta)\}$
of $F(\phi,c)=0$, which intersects the trivial curve $\left(
0,c\right)  $ at $c=c_0$, such that
\begin{align*}
\phi_{\gamma}(x,y)=\gamma\phi_{c_0,n_*}(y)\cos(k_*\alpha x)+o(|\gamma|).
\end{align*}
 So the stream functions take the form\begin{align*}
&\psi_{\gamma}(x,y)=\psi_0(y)+\phi_{\gamma}(x,y)=\psi_{0}(y)+\gamma\phi_{c_0,n_*}(y)\cos(k_*\alpha x)+o(|\gamma|).
\end{align*}Let the velocity $\vec{u}_{\gamma}=(u_{(\gamma)},v_{(\gamma)})  =\left(  \partial_y\psi_{\gamma},-\partial_x\psi_{\gamma}\right) . $  Since $c_0<u_{\min}$, we have
\begin{align}\label{horizontal-velocity does not change sign}
&u_{(\gamma)}(x,y)-c{(\gamma)}= \partial_y\psi_{\gamma}(x,y)-c{(\gamma)}\\\nonumber
=&u(y)-c{(\gamma)}+\gamma\phi'_{c_0,n_*}(y)\cos(k_*\alpha x)+o(|\gamma|)>0,
\end{align}
and
\begin{align}\label{vertical velocity}
&v_{(\gamma)}(x,y)=-\partial_x\psi_{\gamma}(x,y)=k_*\alpha\gamma\phi_{c_0,n_*}(y)\sin(k_*\alpha x)+o(|\gamma|)\neq0
\end{align}
when $ \gamma$
 is small. Moreover, $ \|(u_{(\gamma)},v_{(\gamma)})-(u,0)\|_{H^3(D_T)} +|c({\gamma})-c_0|\leq
C_0\gamma$ for some constant $C_0>0$ large enough. Thus,  we can take $\delta>0$ smaller and $\varepsilon_0=C_0\delta$ such that for $\varepsilon\in(0,\varepsilon_0)$, $(u_{\varepsilon},v_{\varepsilon},c_{\varepsilon}) :=(u_{(\gamma)},v_{(\gamma)},c({\gamma}))|_{\gamma=\varepsilon/C_0}$ satisfies that $\|(u_{\varepsilon},v_{\varepsilon})-(u,0)\|_{H^3(D_T)} \leq
\varepsilon$, $c_\varepsilon\to c_0$, $u_{\varepsilon}(x,y)-c_{\varepsilon}>0$ and $\|v_{\varepsilon}\|_{L^2(D_T)}\neq0$.
By \eqref{vertical velocity},
${\tilde{v}_{\varepsilon}}\longrightarrow \sqrt{\alpha/\pi} \phi_{c_0,n_*}(y)\sin(k_*\alpha x)$ in  ${H^{2}\left(
D_T\right)  }$, where $\tilde v_\varepsilon=v_\varepsilon/\|v_\varepsilon\|_{L^2(D_T)}$.
\\
{\bf Case 2.}
$\lambda_{n_*}'(c_0)=0$ and
$c_0\notin\sigma_d(\mathcal{R}_{0,\beta})\cap\mathbf{R}$.


In this case, there exist $\delta_1\in(0,\delta_0]$ and $a\in\{\pm1\}$ such that $a\lambda_{n_*}$ is increasing in $[c_0, c_0+\delta_1]$, and thus,
\begin{align}\label{monotone in half interval}
a\lambda_{n_*}(c)>a\lambda_{n_*}(c_0)=-a(k_*\alpha)^2,\quad \forall\ c\in(c_0, c_0+\delta_1].
\end{align}
Let $\zeta_1\in C^{\infty}([y_1,y_2])$ be a positive function, $u_1$ be a solution of the regular ODE
\begin{align}\label{constructed ode for perturbation of original shear flow}
u_1''(u-c_0)-(u''-\beta)u_1=\zeta_1\quad\text{on}\quad[y_1,y_2],
\end{align}
and $\tau_0>0$ be such that $[c_0,c_0+\delta_1]\cap\text{Ran}(u+\tau u_1)=\emptyset$ for $\tau\in[-\tau_0,\tau_0]$. Since $u\in H^4(y_1,y_2)$ and $\zeta_1\in C^\infty([y_1,y_2])$, we have $u_1\in H^4(y_1,y_2)$. Let $\lambda_n(c,\tau)$ denote the $n$-th eigenvalue of $\mathcal{L}_{c,\tau}:H^2\cap H_0^1(y_1,y_2)\longrightarrow L^2(y_1,y_2)$ defined by
\begin{align*}
\mathcal{L}_{c,\tau}\phi=-\phi''+{u''+\tau u_1''-\beta\over u+\tau u_1-c}\phi
\end{align*}
for $c\in[c_0,c_0+\delta_1]$ and $\tau\in[-\tau_0,\tau_0]$.
Then by \eqref{constructed ode for perturbation of original shear flow} and the fact that $\zeta_1$ is a positive function, we have
\begin{align*}
\partial_\tau\lambda_{n_*}(c_0,0)=&\int_{y_1}^{y_2}\partial_{\tau}\left({u''+\tau u_1''-\beta\over u+\tau u_1-c_0}\right)\big|_{\tau=0}\phi_{n_*,c_0}^2dy\\
=&\int_{y_1}^{y_2}{u_1''(u-c_0)-(u''-\beta)u_1\over (u-c_0)^2}\phi_{n_*,c_0}^2dy=\int_{y_1}^{y_2}{\zeta_1\over (u-c_0)^2}\phi_{n_*,c_0}^2dy>0,
\end{align*}
where $\phi_{n_*,c_0}$ is a $L^2$ normalized eigenfunction of $\lambda_{n_*}(c_0)\in\sigma(\mathcal{L}_{c_0})$. By the definition of $k_*$, $-(k\al)^2\notin\sigma(\mathcal{L}_{c_0,0})$ for $k>k_*$.
Since $c_0\notin\sigma_d(\mathcal{R}_{0,\beta})\cap\mathbf{R}$, we have $0\notin\sigma(\mathcal{L}_{c_0,0})$. By the continuity of $\partial_\tau\lambda_{n_*}$ and the small perturbation of $\sigma(\mathcal{L}_{c,\tau})$, we can take $\tau_0>0$ and $\delta_1>0$ smaller such that $\partial_\tau\lambda_{n_*}(c,\tau)>0$ and
\begin{align}\label{0-no-spectrum}
0\notin \sigma(\mathcal{L}_{c,\tau})\quad\text{and}\quad -(k\al)^2\notin \sigma(\mathcal{L}_{c,\tau}),\quad \forall\ k>k_*
\end{align}
{for} $(c,\tau)\in[c_0,c_0+\delta_1]\times[-\tau_0,\tau_0]$.
By taking $\delta_1>0$ smaller and the Implicit Function Theorem, there exists $ \widetilde\gamma\in C^1([c_0,c_0+\delta_1])$  such that
$\lambda_{n_*}(c,\widetilde\gamma(c))=\lambda_{n_*}(c_0,0)=-(k_*\alpha)^2,\ \widetilde\gamma(c_0)=0$ and $|\widetilde\gamma(c)|\leq \tau_0$
 for $c\in[c_0,c_0+\delta_1]$. By \eqref{monotone in half interval}, we have
 $\lambda_{n_*}(c,\widetilde\gamma(c))=\lambda_{n_*}(c_0,0)\neq\lambda_{n_*}(c,0)$ and $\widetilde\gamma(c)\neq0$ for $c\in(c_0,c_0+\delta_1]$.
%
 Then for fixed $\tau\in(0,\tau_0]$, there exists $c_{\tau}\in[c_0,c_0+\delta_1]$  such that $\widetilde\gamma'(c_{\tau})\neq0$ and $|\widetilde\gamma(c_{\tau})|\leq\tau$.
 Note that $0=\frac{d}{dc}[\lambda_{n_*}(c,\widetilde\gamma(c))]=\partial_c\lambda_{n_*}(c,\widetilde\gamma(c))+
 \widetilde\gamma'(c)\partial_{\tau}\lambda_{n_*}(c,\widetilde\gamma(c))$.
 Let $ \tau_1=\widetilde\gamma(c_{\tau})$. Then we have $\partial_c\lambda_{n_*}(c_{\tau},\tau_1)=-\widetilde\gamma'(c_{\tau})\partial_{\tau}\lambda_{n_*}(c_{\tau},\tau_1)\neq0.$

 Fix any $\varepsilon\in(0,1)$. Then we can choose $\tau\in(0,\tau_0]$ and $\delta_1>0$ smaller such that for $ \tau_1=\widetilde\gamma(c_{\tau})$,
\begin{align}\label{shear flow perturbation u+tau-varepsilon-u-1-case 2a}
\|(u+\tau_1 u_1,0)-(u,0)\|_{H^3(y_1,y_2)}\leq \tau_1\|u_1\|_{H^3(y_1,y_2)}< {\varepsilon\over2}\quad \text{and} \quad |c_{\tau}-c_0|<\delta_1<{\varepsilon\over2}.
\end{align}
By \eqref{0-no-spectrum}, $\lambda_{n_*}(c_\tau,\tau_1)=-(k_*\alpha)^2$ and $\partial_c\lambda_{n_*}(c_{\tau},\tau_1)\neq0$, we can apply Case 1 to the shear flow $(u+\tau_1u_1,0)$: there exists a traveling wave solution $(u_\varepsilon(x-c_\varepsilon t,y),v_\varepsilon(x-c_\varepsilon t,y))$ to
\eqref{Euler equation}--\eqref{boundary condition for euler} which has
 period $T={2\pi}/{\alpha}$ in $x$,
\begin{align}\label{traveling wave to u+tau-varepsilon-u-1-case 2a}
 \|(u_\varepsilon,v_\varepsilon)-(u+\tau_1 u_1,0)\|_{H^3(D_T)} \leq{\varepsilon\over2} \quad \text{and} \quad |c_{\varepsilon}-c_{\tau}|\leq{\varepsilon\over2},
\end{align}
$u_{\varepsilon}\left(  x,y\right)-c_\varepsilon  \neq0$ and $\|v_\varepsilon\|_{L^2\left(  D_T\right)}\neq0$. Then by \eqref{shear flow perturbation u+tau-varepsilon-u-1-case 2a}--\eqref{traveling wave to u+tau-varepsilon-u-1-case 2a}, we have
$
 \|(u_\varepsilon,v_\varepsilon)-(u,0)\|_{H^3(D_T)}$ $<{\varepsilon}$ {and}  $|c_{\varepsilon}-c_{0}|<{\varepsilon}.
$
\\
{\bf Case 3.} $c_0\in\sigma_d(\mathcal{R}_{0,\beta})\cap\mathbf{R}$.

In this case, $0\in\sigma(\mathcal{L}_{c_0})$ and there exists $j_0>n_0\geq n_*$ such that $\lambda_{j_0}(c_0)=0$. There exist $\delta_1\in(0,\delta_0]$ and $a,b\in\{\pm1\}$ such that both $a\lambda_{n_*}$  and $b\lambda_{j_0}$ are decreasing in $[c_0,c_0+\delta_1]$.

 Since $\phi_{n_*,c_0}^2$ is linearly independent of $\phi_{j_0,c_0}^2$, there exists $\xi_1\in C^{\infty}([y_1,y_2])$ such that
 \begin{align}\label{int-ab}
 \int_{y_1}^{y_2}\xi_1{\phi_{n_*,c_0}^2\over(u-c_0)^2}dy=a \quad \text{and} \quad \int_{y_1}^{y_2}\xi_1{\phi_{j_0,c_0}^2\over(u-c_0)^2}dy=-b.
 \end{align}
 Let $u_1$ be a solution of \eqref{constructed ode for perturbation of original shear flow} with $\zeta_1=\xi_1$, and $\tau_0>0$ be such that
$[c_0,c_0+\delta_1]\cap \text{Ran}(u+\tau u_1)=\emptyset$ for $\tau\in[-\tau_0,\tau_0]$.
  Then by \eqref{int-ab}, we have
$
 a\partial_\tau\lambda_{n_*}(c_0,0)=a^2>0$
 and
$
 b\partial_\tau\lambda_{j_0}(c_0,0)=-b^2<0.
$
As in Case 2, we can take $\tau_0>0$ and $\delta_1>0$ smaller such that \begin{align*}
a\partial_\tau\lambda_{n_*}(c,\tau)>0\quad\text{and}\quad -(k\al)^2\notin \sigma(\mathcal{L}_{c,\tau}),\quad \forall\ k>k_*
\end{align*} for $(c,\tau)\in[c_0,c_0+\delta_1]\times[-\tau_0,\tau_0]$.
Note that $\lambda_{j_0-1}(c_0,0)<\lambda_{j_0}(c_0,0)=0<\lambda_{j_0+1}(c_0,0)$. By the continuity of $\partial_\tau\lambda_{j_0}$, $\lambda_{j_0-1}$ and $\lambda_{j_0+1}$, we can choose $\tau_0>0$ and $\delta_1>0$ smaller such that
\begin{align}\label{j0-1j0+1}
 b\partial_{\tau}\lambda_{j_0}(c,\tau)<0\quad\text{and}\quad \lambda_{j_0-1}(c,\tau)<0<\lambda_{j_0+1}(c,\tau)
\end{align}
for $(c,\tau)\in[c_0,c_0+\delta_1]\times [-\tau_0,\tau_0]$.

As $ a\partial_\tau\lambda_{n_*}(c_0,0)>0$ and $a\lambda_{n_*}(\cdot,0)$ is decreasing in $[c_0,c_0+\delta_1]$,
   we can choose $\tau_1\in(0,\tau_0]$ such that
 $a\lambda_{n_*}(c_0+\delta_1,\tau_1)<a\lambda_{n_*}(c_0,0)=-a(k_*\alpha)^2<a\lambda_{n_*}(c_0,\tau_1)$.
 Then there exists $c_{\tau_1}\in(c_0,c_0+\delta_1)$ such that $\lambda_{n_*}(c_{\tau_1},\tau_1)=-(k_*\alpha)^2. $

Since $b\partial_{\tau}\lambda_{j_0}(c,\tau)<0$ and $b\lambda_{j_0}(\cdot,0)$ is decreasing in $[c_0,c_0+\delta_1]$, we have
$
b\lambda_{j_0}(c,\tau)<b\lambda_{j_0}(c,0)<b\lambda_{j_0}(c_0,0)=0,
$
 which, along with \eqref{j0-1j0+1}, gives
 \begin{align*}
 b\lambda_{j_0+b}(c,\tau)>0>b\lambda_{j_0}(c,\tau),\quad (c,\tau)\in(c_0,c_0+\delta_1]\times(0,\tau_0].
 \end{align*}
 Since $(c_{\tau_1},\tau_1)\in(c_0,c_0+\delta_1]\times(0,\tau_0]$, we have $0\notin\sigma(\mathcal{L}_{c_{\tau_1},\tau_1})$.

Now, we can construct  a desired traveling wave solution $(u_\varepsilon(x-c_\varepsilon t,y),v_\varepsilon(x-c_\varepsilon t,y))$    by first  perturbing the shear flow $(u,0)$ to $(u+\tau_1u_1,0)$ and then applying Case 1 or Case 2 to $(u+\tau_1u_1,0)$
 as in \eqref{shear flow perturbation u+tau-varepsilon-u-1-case 2a}--\eqref{traveling wave to u+tau-varepsilon-u-1-case 2a}.
\end{proof}

To prove Corollary \ref{traveling wave construction-corollary}, we only need to modify the spaces $B$ and $C$ from $H^{4}$ and $H^{2}$ to $H^{s+1}$ and $H^{s-1}$
 in the proof of Lemma \ref{traveling wave construction-lemma}. We also use the fact  that  $\tilde{f}_c\in C^{\infty}(I_c)$, $f_c\in C_0^{\infty}(\mathbf{R})$
  and $u_1\in C^{\infty}([y_1,y_2])$ due to the assumption that $u\in C^{\infty}([y_1,y_2])$.

Conversely,  for  a set of
 traveling wave solutions near $(u,0)$ with traveling speeds converging to  $c_0$,  we show that $c_0$ is an isolated real eigenvalue besides $u_{\min}$ and $u_{\max}$, which is given in Lemma \ref{traveling wave concentration}.

\begin{proof}[Proof of Lemma \ref{traveling wave concentration}] It suffices to show that if $c_0\notin\{u_{\max},u_{\min}\}$, then
$c_0\in\bigcup_{k\geq1}(\sigma_d(\mathcal{R}_{k\alpha,\beta})\cap\mathbf{R})$ and \eqref{eigenfunction-convergence} holds.
 Note that $(u_\varepsilon,v_\varepsilon)$
 solves
 \begin{equation}\label{nonlinear}(u_\varepsilon-c_\varepsilon)\partial_x\omega_\varepsilon+v_\varepsilon\partial_y\omega_\varepsilon+\beta v_\varepsilon=0.
 \end{equation}
 Moreover,
 \begin{align*}
 \|\omega_\varepsilon-\omega_0\|_{H^2(D_T)}\leq C\|(u_\varepsilon,v_\varepsilon)-(u,0)\|_{H^3(D_T)}\leq C\varepsilon.
\end{align*}
By taking $\varepsilon_0>0$ smaller,
 \begin{align}\label{u-c-estimate}|u_\varepsilon-c_\varepsilon|\geq|u-c_\varepsilon|-|u-u_\varepsilon|\geq C^{-1}\end{align}
  for $\varepsilon\in(0,\varepsilon_0)$ and $y\in[y_1,y_2]$.
Note that ${\pi\over y_2-y_1}\|v_\varepsilon\|_{L^2(D_T)}\leq \|\nabla v_\varepsilon\|_{L^2(D_T)}$. By Sobolev embedding, we have
\begin{align}\label{sobolev1}&\|v_\varepsilon\|_{L^4(D_T)}\leq C\| v_\varepsilon\|_{H^1(D_T)}\leq C\| \nabla v_\varepsilon\|_{L^2(D_T)},\\\label{sobolev2}
&\|\partial_y(\omega_\varepsilon-\omega_0)\|_{L^4(D_T)}\leq C
\|\partial_y(\omega_\varepsilon-\omega_0)\|_{H^1(D_T)}
\leq C
\|\omega_\varepsilon-\omega_0\|_{H^2(D_T)}\leq C\varepsilon.
\end{align}

Since $\partial_x\omega_\varepsilon=\partial_x(\partial_xv_\varepsilon-\partial_yu_\varepsilon)=\Delta v_\varepsilon$, we get by \eqref{nonlinear} that
\begin{align}\label{equation-for-normalized vertical velocity}
\Delta\tilde v_\varepsilon+{\partial_y(\omega_\varepsilon-\omega_0)\over u_\varepsilon-c_\varepsilon}\tilde v_\varepsilon+{\beta-u''\over u_\varepsilon-c_\varepsilon}\tilde v_\varepsilon=0,
\end{align}where
 $\tilde v_\varepsilon={v_\varepsilon/\|v_\varepsilon\|_{L^2(D_T)}}$.
By \eqref{u-c-estimate}, we have  $\left|{\beta-u''\over u_\varepsilon-c_\varepsilon}\right|\leq C$ for $y\in[y_1,y_2]$ and
\begin{align*}
&\|\Delta\tilde v_\varepsilon\|_{L^2(D_T)}\leq C\|\partial_y(\omega_\varepsilon-\omega_0)\|_{L^4(D_T)}\|\tilde v_\varepsilon\|_{L^4(D_T)}+C\|\tilde v_\varepsilon\|_{L^2(D_T)}\\
\leq& C\varepsilon\|\tilde v_\varepsilon\|_{H^1(D_T)}+C\|\tilde v_\varepsilon\|_{L^2(D_T)}\leq C\varepsilon\|\tilde v_\varepsilon\|_{H^2(D_T)}^{1/2}+C,
\end{align*}
where we used \eqref{sobolev1}--\eqref{sobolev2} and $\|\tilde v_\varepsilon\|_{H^1(D_T)}\leq C\|\tilde v_\varepsilon\|_{H^2(D_T)}^{1/2}\|\tilde v_\varepsilon\|_{L^2(D_T)}^{1/2}$. Since $v_\varepsilon(x,y_i)=0$ for $i=1,2$, we have $\|\tilde v_\varepsilon\|_{H^2(D_T)}\leq C\|\Delta\tilde v_\varepsilon\|_{L^2(D_T)}.$ Thus, $\|\tilde v_\varepsilon\|_{H^2(D_T)}\leq C$ and
\begin{align}\label{c0norm normalized vertical velocity}
&\|\tilde v_\varepsilon\|_{C^0([0,T]\times[y_1,y_2])}\leq C\|\tilde v_\varepsilon\|_{H^2(D_T)}\leq C,\\\label{L4norm normalized vertical velocity}
&\|\partial_x\tilde v_\varepsilon\|_{L^4(D_T)}+\|\partial_y\tilde v_\varepsilon\|_{L^4(D_T)}\leq
C\|\tilde v_\varepsilon\|_{H^2(D_T)}\leq C.
\end{align}Up to a subsequence, there exists $\tilde v_0\in{H^2(D_T)}$ such that $\tilde v_\varepsilon\rightharpoonup\tilde v_0$ in $H^2(D_T)$, $\tilde v_\varepsilon\rightarrow\tilde v_0$ in ${H^1(D_T)}$ and $\|\tilde v_0\|_{L^2(D_T)}=1$.
Taking derivative in \eqref{equation-for-normalized vertical velocity} with respect to $x$ and $y$, we get by \eqref{sobolev2} and \eqref{c0norm normalized vertical velocity}--\eqref{L4norm normalized vertical velocity} that
\begin{align*}
&\|\partial_x\Delta\tilde v_\varepsilon\|_{L^2(D_T)}\\
\leq& \left\|\partial_x\left({\partial_y(\omega_\varepsilon-\omega_0)\over u_\varepsilon-c_\varepsilon}\right)\tilde v_\varepsilon+{\partial_y(\omega_\varepsilon-\omega_0)\over u_\varepsilon-c_\varepsilon}\partial_x\tilde v_\varepsilon+\partial_x\left({\beta-u''\over u_\varepsilon-c_\varepsilon}\right)\tilde v_\varepsilon+{\beta-u''\over u_\varepsilon-c_\varepsilon}\partial_x\tilde v_\varepsilon\right\|_{L^2(D_T)}\\
\leq& C\left(\|\partial_{xy}(\omega_\varepsilon-\omega_0)\|_{L^2(D_T)}+\|\partial_{y}(\omega_\varepsilon-\omega_0)\|_{L^2(D_T)}\right)\|\tilde v_\varepsilon\|_{C^0([0,T]\times[y_1,y_2])}+\\
&C\|\partial_y(\omega_\varepsilon-\omega_0)\|_{L^4(D_T)}\|\partial_x\tilde v_\varepsilon\|_{L^4(D_T)}
+C\|\tilde v_\varepsilon\|_{L^2(D_T)}+C\|\partial_x\tilde v_\varepsilon\|_{L^2(D_T)}\leq C,
\end{align*}
and
\begin{align*}
&\|\partial_y\Delta\tilde v_\varepsilon\|_{L^2(D_T)}\\
\leq& \left\|\partial_y\left({\partial_y(\omega_\varepsilon-\omega_0)\over u_\varepsilon-c_\varepsilon}\right)\tilde v_\varepsilon+
{\partial_y(\omega_\varepsilon-\omega_0)\over u_\varepsilon-c_\varepsilon}\partial_y\tilde v_\varepsilon
+\partial_y\left({\beta-u''\over u_\varepsilon-c_\varepsilon}\right)\tilde v_\varepsilon+{\beta-u''\over u_\varepsilon-c_\varepsilon}\partial_y\tilde v_\varepsilon\right\|_{L^2(D_T)}\\
\leq& C\left(\|\partial_y^2(\omega_\varepsilon-\omega_0)\|_{L^2(D_T)}
+\|\partial_y(\omega_\varepsilon-\omega_0)\|_{L^2(D_T)}\right)\|\tilde v_\varepsilon\|_{C^0([0,T]\times[y_1,y_2])}+
\\& C\|\partial_y(\omega_\varepsilon-\omega_0)\|_{L^4(D_T)}
\|\partial_y\tilde v_\varepsilon\|_{L^4(D_T)}+
C\|\tilde v_\varepsilon\|_{L^2(D_T)}+
C\|\partial_y\tilde v_\varepsilon\|_{L^2(D_T)}\leq C,
\end{align*}
which implies that $\|\tilde v_\varepsilon\|_{H^3(D_T)}\leq C$ and thus, $\tilde v_\varepsilon\longrightarrow\tilde v_0 $ in $H^2(D_T)$.
For any $\phi\in H^1(D_T)$ with periodic boundary condition in $x$ and Dirichlet boundary condition in $y$, we have
\begin{align}\label{equa-in weak sense}
\int_0^T\int_{y_1}^{y_2}\left(-\nabla \tilde v_\varepsilon\cdot \nabla\phi+{\partial_y(\omega_\varepsilon-\omega_0)\over u_\varepsilon-c_\varepsilon}\tilde v_\varepsilon\phi+{\beta-u''\over u_\varepsilon-c_\varepsilon}\tilde v_\varepsilon\phi\right)dydx=0.
\end{align}
Since  $\|\tilde v_\varepsilon\|_{L^4(D_T)}\leq C\|\tilde v_\varepsilon\|_{H^1(D_T)}\leq C$,  we have by \eqref{u-c-estimate} and \eqref{sobolev2} that
\begin{align*}
&\left|\int_0^T\int_{y_1}^{y_2}{\partial_y(\omega_\varepsilon-\omega_0)\over u_\varepsilon-c_\varepsilon}\tilde v_\varepsilon\phi dydx\right|\leq
C\int_0^T\int_{y_1}^{y_2}|\partial_y(\omega_\varepsilon-\omega_0)||\tilde v_\varepsilon||\phi| dydx\\
\leq &C\|\partial_y(\omega_\varepsilon-\omega_0)\|_{L^4(D_T)}\|\tilde v_\varepsilon\|_{L^4(D_T)}\|\phi\|_{L^2(D_T)}
\leq C\varepsilon\|\phi\|_{L^2(D_T)}\longrightarrow0 \ \ \text{ as } \varepsilon\to0^+.
\end{align*}
Noting that $\tilde v_\varepsilon\longrightarrow\tilde v_0$ in ${H^2(D_T)}$ and  sending $\varepsilon\to0^+$ in \eqref{equa-in weak sense}, we have
\begin{align*}
\int_0^T\int_{y_1}^{y_2}\left(-\nabla \tilde v_0\cdot \nabla\phi+{\beta-u''\over u-c_0}\tilde v_0\phi\right)dydx=0.
\end{align*}
Thus, $\tilde v_0\in H^2(D_T)$ is a weak solution of
\begin{align}\label{limit-equa}
\mathcal{G}_{c_0}\tilde v_0=-\Delta \tilde v_0-{\beta-u''\over u-c_0}\tilde v_0=0.
\end{align}
Since $c_0\notin \text{Ran}(u)$, we have $\left|{\beta-u''\over u-c_0}\right|\leq C$ for $y\in[y_1,y_2]$.   Then by elliptic regularity theory,
we have $\tilde v_0$ is a classical solution of \eqref{limit-equa}. Thus, $\varphi_{c_0}:=\tilde v_0\in\ker(\mathcal{G}_{c_0})$. Since $-\Delta \phi=0$
has no nontrivial solutions satisfying the boundary conditions, we have $|c_0|<\infty$.
Since
 $\varphi_{c_0}=\sum_{k\in\mathbf{Z}}\widehat{\varphi}_{{c_0},k}(y)e^{ik\alpha x}\neq0$
  solves \eqref{limit-equa}, there exists $k_0\in\mathbf{Z}$ such that
$\widehat{\varphi}_{{c_0},k_0}\neq0$ solves
\begin{align*}
-\widehat{\varphi}_{{c_0},k_0}''+(k_0\alpha)^2\widehat{\varphi}_{{c_0},k_0}-{\beta-u''\over u-c_0}\widehat{\varphi}_{{c_0},k_0}=0
\end{align*}
with $\widehat{\varphi}_{{c_0},k_0}(y_1)=\widehat{\varphi}_{{c_0},k_0}(y_2)=0$. Now we show that $k_0\neq 0$. Let $P_0f(x,y)={1\over T}\int_0^Tf(x,y)dx$ for $f\in L^2(D_T)$. Then $P_0$ is a bounded linear operator on $L^2(D_T)$.
Since $\tilde v_\varepsilon=v_\varepsilon/\|v_\varepsilon\|_{L^2(D_T)}=-\partial_x\psi_\varepsilon/\|v_\varepsilon\|_{L^2(D_T)}$, we have $P_0\tilde v_\varepsilon=P_0v_\varepsilon=0$. Taking limit as $\varepsilon\to0^+$, we have $P_0\varphi_{c_0}=P_0\tilde v_0=0$ and thus, $\widehat{\varphi}_{{c_0},0}(y)=P_0\varphi_{c_0}(x,y)\equiv0$, which implies that $k_0\neq 0$.
Thus, $c_0\in\bigcup_{k\neq0}(\sigma_d(\mathcal{R}_{k\alpha,\beta})\cap\mathbf{R})=\bigcup_{k\geq1}(\sigma_d(\mathcal{R}_{k\alpha,\beta})\cap\mathbf{R})$, where we used the fact that $\sigma_d(\mathcal{R}_{k\alpha,\beta})=\sigma_d(\mathcal{R}_{-k\alpha,\beta}).$
\end{proof}

We give two remarks to Lemma
 \ref{traveling wave concentration}: the first is to study the Fourier expansion  of the limit function $\varphi_{c_0}$, and the second is to  show that the asymptotic behavior
 of $L^2$ normalized vertical velocities $\tilde v_\varepsilon$ might be complicated if $c_{\varepsilon}\rightarrow c_0\in\{u_{\min},u_{\max}\}$.

\begin{remark}\label{The function varphi c0 superposition of finite normal modes}
The function $\varphi_{c_0}$ in
 Lemma
 \ref{traveling wave concentration} is a superposition of finite normal modes.
In fact,
since $0\neq\varphi_{c_0}(x,y)=\sum_{k\in\mathbf{Z}}\widehat{\varphi}_{{c_0},k}(y)e^{ik\alpha x}\in\ker{(\mathcal{G}_{c_0})}$ and $\inf\sigma(\mathcal{L}_{c_0})>-\infty$, we have $n_*:=\sharp(\{k\in\mathbf{Z}:-(k\alpha)^2\in\sigma(\mathcal{L}_{c_0})\})\in[1,\infty)$.
Let $\{k\in\mathbf{Z}:-(k\alpha)^2\in\sigma(\mathcal{L}_{c_0})\}=\{k_n:1\leq n\leq n_*\}$. Then
$\varphi_{c_0}(x,y)=\sum_{n=1}^{n_*}\widehat{\varphi}_{{c_0},k_n}(y)e^{ik_n\alpha x}$ and  $\widehat{\varphi}_{{c_0},k_n}$ is an eigenfunction of $-(k_n\alpha)^2\in\sigma(\mathcal{L}_{c_0})$.
\end{remark}

\begin{remark}\label{asymptotic behavior of vertical velocities}
 Consider a flow $u$ satisfying $(\bf{H1})$, $\{u'=0\}\cap\{u=u_{\min}\}\neq \emptyset$, $\alpha>0$ and $\beta>{9\over 8} \kappa_+$.
By Theorem $\ref{eigenvalue asymptotic behavior-bound}$ $(3)$,  there exists $\{c_n\}_{n=1}^\infty=\sigma_d(\mathcal{R}_{\alpha,\beta})\cap\mathbf{R}$ such that $c_n\to u_{\min}^-$, $c_{n+1}>c_n$ and $\alpha^2=-\lambda_n(c_n)$  for $n\geq1$. By Lemma
 $\ref{traveling wave construction-lemma}$, we can choose $\varepsilon_n\to0^+$ and nearby traveling wave solutions
 $\vec{u}_{\varepsilon_n}\left(
x-c_{\varepsilon_n}t,y\right)  =(  u_{\varepsilon_n}\left(  x-c_{\varepsilon_n
}t,y\right),$ $  v_{\varepsilon_n}\left(  x-c_{\varepsilon_n}t,y\right)  )  $
with period ${2\pi/\alpha}$ in $x$ such that
 $\|(u_{\varepsilon_n},v_{\varepsilon_n})-(u,0)\|_{H^3(D_T)} \leq\varepsilon_n,$
$|c_{\varepsilon_n}-c_n|\to0^+$,  $\|v_{\varepsilon_n}\|_{L^2\left(  D_T\right)}\neq0$
and there exists $\varphi_{c_n}\in\ker(\mathcal{G}_{c_n})$ such that
  $\left\|\tilde{v}_{\varepsilon_n}-\varphi_{c_n}\right\|_{C^0}$ is small enough for large $n$, where $\tilde{v}_{\varepsilon_n}={{v}_{\varepsilon_n}/\|{v}_{\varepsilon_n}\|_{L^2\left( D_T\right)}}$.
 Then $c_{\varepsilon_n}\to u_{\min}^-$.
A possible case is that  there exists a subsequence $\{n_j\}_{j=1}^\infty$ such that
$c_{n_j}\notin \bigcup_{k>1}(\sigma_d(\mathcal{R}_{k\alpha,\beta})\cap\mathbf{R})$ for $j\geq1$. In this case,
$\varphi_{c_{n_j}}(x,y)=\sqrt{\alpha/\pi}\phi_{c_{n_j}}(y)\sin(\alpha x)$ since it is odd in $x$ (see the construction in Lemma \ref{traveling wave construction-lemma}),
where $\phi_{c_{n_j}}$ is a $L^2$ normalized eigenfunction of $\lambda_{n_j}(c_{n_j})=-\alpha^2\in \sigma(\mathcal{L}_{c_{n_j}})$.
  Since  $ \phi_{c_{n_j}}$ has $n_j-1$ sign-changed zeros in $(y_1,y_2)$, $\tilde{v}_{\varepsilon_{n_j}}$ oscillates frequently in the $y$-direction for large $j$.
\end{remark}

The minimal period of any nearby traveling wave solution  in $x$ can be determined under the following condition.

\begin{lemma}\label{traveling wave concentration ad}
 Under the assumption of Lemma \ref{traveling wave concentration},
  if
\begin{align}\label{c0-ad}
c_0\in\sigma_d(\mathcal{R}_{\alpha,\beta})\cap\mathbf{R}\quad\text{and}\quad c_0\notin\bigcup_{k\geq2}(\sigma_d(\mathcal{R}_{k\alpha,\beta})\cap\mathbf{R}),\end{align}
then
$\vec{u}_{\varepsilon}\left(
x-c_{\varepsilon}t,y\right)$ has minimal period ${2\pi/\alpha}$ in $x$ for $\varepsilon>0$ small enough.
\end{lemma}

\begin{proof}
Let the minimal horizontal period  of the traveling wave solution $\vec{u}_{\varepsilon}\left(
x-c_{\varepsilon}t,y\right)$  be $T/n_{\varepsilon}$ for $\varepsilon\in(0,\varepsilon_0)$, where $n_\varepsilon\in\mathbf{Z}^+$ and $T={{2\pi}/\alpha}$.
Fix $\varepsilon\in(0,\varepsilon_0)$ and let $\tilde v_\varepsilon=v_\varepsilon/\|v_\varepsilon\|_{L^2(D_{T})}$.
Since $\tilde v_\varepsilon(x,y)=\tilde v_\varepsilon(x+T/n_{\varepsilon},y) $ for $x\in\mathbf{R}$ and $y\in[y_1,y_2]$,
we have $\int_{0}^{T}e^{i\alpha x}\tilde v_\varepsilon(x,y) dx=e^{i\alpha T/n_{\varepsilon}}\int_{0}^{T}
e^{i\alpha x}\tilde v_\varepsilon(x,y) dx$.
Thus, if $n_{\varepsilon}>1$, then $e^{i\alpha T/n_{\varepsilon}}=e^{2\pi i/n_{\varepsilon}}\neq 1 $ and $\int_{0}^{T}e^{i\alpha x}\tilde v_\varepsilon(x,y) dx=0$ for $y\in[y_1,y_2]$.

Suppose that  there exists a sequence $\{\varepsilon_k:k\geq1\}\subset(0,\varepsilon_0)$ such that $\varepsilon_k\to0^+$ and    $\vec{u}_{\varepsilon_k}\left(
x-c_{\varepsilon_k}t,y\right)$ has minimal period ${T/n_{\varepsilon_k}}<T$ in $x$ for $k\geq1$. Then
$\int_{0}^{T}e^{i\alpha x}\tilde v_{\varepsilon_k}(x,y) dx=0$  for $k\geq1$ and $y\in[y_1,y_2]$.
By Lemma \ref{traveling wave concentration},
 ${\tilde{v}_{\varepsilon_k}}\longrightarrow\varphi_{c_0}$ {in}  ${H^{2}\left(D_{T}\right)}$, where $\varphi_{c_0}\in \ker(\mathcal{G}_{c_0})$ and $\mathcal{G}_{c_0}$ is defined in \eqref{Glinearized elliptic operator}.
 Then
\begin{align}\label{x0-x0+Tn0}
\int_{0}^{T}
e^{i\alpha x}\varphi_{c_0}(x,y) dx=
\lim_{k\to\infty}\int_{0}^{T}
e^{i\alpha x}\tilde v_{\varepsilon_k}(x,y) dx=0
\end{align}
 for  $y\in [y_1,y_2]$.
 By \eqref{c0-ad} and $\widehat{\varphi}_{c_0,0}=0$, we have $\varphi_{c_0}(x,y)=\widehat{\varphi}_{{c_0},1}(y)e^{i\alpha x}+
 \overline{\widehat{\varphi}_{{c_0},1}(y)}e^{-i\alpha x}
 $, where  $\widehat{\varphi}_{{c_0},1}\neq0$ is an eigenfunction of
 $-\alpha^2\in\sigma(\mathcal{L}_{c_0})$. On the other hand, we have
 \begin{align*}
 &\int_{0}^{T}e^{i\alpha x}\varphi_{c_0}(x,y)dx=T\overline{\widehat{\varphi}_{{c_0},1}(y)}\not\equiv0,
 \end{align*}
which contradicts \eqref{x0-x0+Tn0}.
\end{proof}

\begin{remark}\label{rem-Prop 7-LYZ}
$(1)$ Let $c_0\in\bigcup_{k\geq1}(\sigma_d(\mathcal{R}_{k\alpha,\beta})\cap\mathbf{R})$ and $k_*$ be defined in \eqref{def-k*-constructed traveling wave}. It follows from Lemma $\ref{traveling wave concentration ad}$ that under the assumption of Lemma $\ref{traveling wave concentration}$, if $\vec{u}_{\varepsilon}\left(x-c_{\varepsilon}t,y\right)$ has period ${2\pi}/({k_*\alpha})$ in $x$ for $\varepsilon\in(0,\varepsilon_0)$, then the period  ${2\pi}/({k_*\alpha})$ is minimal for  $\varepsilon>0$ small enough.

$(2)$
In Proposition $7$ of \cite{LYZ}, it should be corrected  that the minimal period of constructed traveling wave solutions in $x$ might be less than $2\pi/\alpha_0$,
since it is possible that  $(c_0,k\alpha_0, \beta, \phi_k)$ is a non-resonant neutral mode  for some  $k\geq2$ and $\phi_k\in H_0^1\cap H^2(y_1,y_2)$.
   Consequently,
   the   minimal period of constructed  traveling wave solutions near the sinus profile in Theorem $7$ $(\rm{i})$ of \cite{LYZ}  might be less than $2\pi/\alpha_0$, see Example $\ref{example-sinus flow}$ for systematic study of traveling wave families near the sinus profile.
   If \eqref{c0-ad} holds true for $\alpha=\alpha_0$, then   the   minimal period of these  traveling wave solutions in  Proposition $7$ and  Theorem $7$ $(\rm{i})$ of \cite{LYZ} is $2\pi/\alpha_0$.
\end{remark}

\section{The number of traveling wave families  near a shear flow}

In this section, we prove the main theorems-Theorems \ref{traveling wave construction} and \ref{traveling wave construction-classK+}.
The proof is based on the study on the number of isolated real eigenvalues of the linearized Euler operator in Sections 3-4, and correspondence between a traveling wave family near the shear flow and an isolated real eigenvalue in Section 5. We only prove Theorem \ref{traveling wave construction}, since the other is similar.

\begin{proof}[Proof of Theorem \ref{traveling wave construction}.]
Let the number of traveling wave families near $(u,0)$ be denoted by $\theta$. By Theorem \ref{the explicit number of traveling wave families},
$\theta=\sharp(\bigcup_{k\geq1}(\sigma_d(\mathcal{R}_{k\alpha,\beta})\cap\mathbf{R}))$. Here $\alpha=2\pi/T$.
\\
{\it Proof of (1):} Since $\{u'=0\}\cap\{u=u_{\min}\}\neq\emptyset$, we have $0<\kappa_+<\infty$. First, let $\{u=u_{\min}\}\cap(y_1,y_2)\neq\emptyset$ and we divide the discussion into two cases.
\\
{\bf Case 1a.} $\beta\in(0,\min\{{9\over8}\kappa_+,\mu_+\})$.

By Corollary \ref{sturm-liouville-eigen-asym} (1),
\begin{align}\label{first-eigen-bounded}
\inf_{c\in(-\infty,u_{\min})}\lambda_1(c)>-\infty.
\end{align}
Thus, there exists $1\leq k_0<\infty$ such that
\begin{align}\label{large-k-alpha-empty}
\bigcup_{k> k_0}(\sigma_d(\mathcal{R}_{k\alpha,\beta})\cap\mathbf{R})=\emptyset.
\end{align}
By Theorem \ref{main-result} (1), we have
\begin{align}\label{small-k-alpha-finte}
\theta=\sharp\big(\bigcup_{k\geq1}(\sigma_d(\mathcal{R}_{k\alpha,\beta})\cap\mathbf{R})\big)\leq
\sum_{k\geq1}\sharp(\sigma_d(\mathcal{R}_{k\alpha,\beta})\cap\mathbf{R})=
\sum_{k=1}^{k_0}\sharp(\sigma_d(\mathcal{R}_{k\alpha,\beta})\cap\mathbf{R})<\infty.
\end{align}
\\
{\bf Case 1b.} $\beta\in(\min\{{9\over8}\kappa_+,\mu_+\},\infty)$.

By Corollary \ref{sturm-liouville-eigen-asym} (1) for $\beta\in(\min\{{9\over8}\kappa_+,\mu_+\},{9\over8}\kappa_+]$ and
Theorem \ref{eigenvalue asymptotic behavior-bound} (3) for $\beta\in({9\over8}\kappa_+,\infty)$, we have
\begin{align}\label{first-eigen-asym-infty}
\lim_{c\to u_{\min}^-}\lambda_1(c)=-\infty.
\end{align}
Thus, there exists $c_k\in\sigma_d(\mathcal{R}_{k\alpha,\beta})\cap\mathbf{R}$ such that
  $c_k<c_{k+1}<u_{\min}$ for every $k\geq1$, and $c_k\to u_{\min}^-$.
Then $\theta=\sharp(\bigcup_{k\geq1}(\sigma_d(\mathcal{R}_{k\alpha,\beta})\cap\mathbf{R}))=\infty$.

 Next, let $\{u=u_{\min}\}\cap(y_1,y_2)=\emptyset$ and we separate the proof into two cases.
 \\
 {\bf Case 2a.} $\beta\in(0,{9\over 8}\kappa_+)$.

 By Corollary \ref{sturm-liouville-eigen-asym} (3), we obtain \eqref{first-eigen-bounded}.
Thus, there exists $1\leq k_0<\infty$ such that \eqref{large-k-alpha-empty} holds.
 By Theorem  \ref{main-result} (1), we obtain
\eqref{small-k-alpha-finte}.
 \\
 {\bf Case 2b.} $\beta\in({9\over 8}\kappa_+,\infty)$.

By Theorem \ref{eigenvalue asymptotic behavior-bound} (3), we have
\begin{align}\label{all-eigen-asym-infty}
\lim_{c\to u_{\min}^-}\lambda_n(c)=-\infty,\;\; n\geq1.
\end{align}
Using  \eqref{all-eigen-asym-infty} for $n=1$ and the fact that  $\theta=\sharp(\bigcup_{k\geq1}(\sigma_d(\mathcal{R}_{k\alpha,\beta})\cap\mathbf{R}))$, it can be proved that $\theta=\infty$ by a similar way as in Case 1b. This completes the proof of (1). The proof of (2) is similar.
 \smallskip
 \\
{\it Proof of (3):} Since $\{u'=0\}\cap\{u=u_{\min}\}= \emptyset$, we have $\{u=u_{\min}\}\cap(y_1,y_2)= \emptyset$ and $\kappa_+=\infty$.
  Fix $\beta\in(0,\infty)$.
  By Corollary \ref{sturm-liouville-eigen-asym} (3), we obtain \eqref{first-eigen-bounded}, and thus,
there exists $1\leq k_0<\infty$ such that \eqref{large-k-alpha-empty} holds.
By Theorem \ref{thm-flows good endpoints} (1),  we obtain
\eqref{small-k-alpha-finte}.
 This completes the proof of (3). The proof of (4) is similar.
\end{proof}
\begin{remark}
In Cases $1b$ and $2b$ of the above proof, the infinitely many traveling wave families are produced by
 the asymptotic behavior of  the first eigenvalue $\lambda_1(c)$ of $\mathcal{L}_c$, see \eqref{first-eigen-asym-infty}.
There could be many other traveling wave families in general, which are produced by the asymptotic behavior of $\lambda_n(c)$ for $n\geq2$, see \eqref{all-eigen-asym-infty}. In fact, if $\beta\in({9\over 8}\kappa_+,\infty)$,
for fixed $n\geq2$, there exists $c_{k,n}\in\sigma_d(\mathcal{R}_{k\alpha,\beta})\cap\mathbf{R}$ such that $\lambda_n(c_{k,n})=-(k\alpha)^2$  for every $k\geq1$. If $\{c_{k,n}:k\geq1,n\geq2\}\setminus\{c_{k}:k\geq1\}\neq\emptyset$, then a simple application of Theorem $\ref{the explicit number of traveling wave families}$ yields other traveling wave families.
\end{remark}
\section{Application to the sinus profile}

In this section, we apply our main results to the sinus profile. Moreover, we calculate  the explicit  number  of  isolated real eigenvalues of $\mathcal{R}_{\alpha,\beta}$ and traveling wave families near sinus profile.

\begin{example}\label{example-sinus flow}
The  sinus profile is $u(y)={\frac{1+\cos(\pi y)}{2}},$ $ y\in\lbrack-1,1]$.
 We determine $\sharp\big(\sigma_d(\mathcal{R}_{\alpha,\beta})\cap$ $\mathbf{R}\big)$ and the number  of traveling wave families for
the sinus profile
  on the $(\alpha,\beta)$'s region.
 For the sinus profile, we have $u_{\min}=0$, $u_{\max}=1$, $\{u'=0\}\cap\{u=u_{\min}\}=\{\pm1\}$, $\{u'=0\}\cap\{u=u_{\max}\}=\{0\}$,
 $\kappa_+=u''(\pm 1)={1\over2}\pi^2$ and $\kappa_-=u''(0)=-{1\over2}\pi^2$.
  We
  divide the plane into nine parts as follows.

\begin{center}
 \begin{tikzpicture}[scale=0.58]
 \draw [->](-12, 0)--(12, 0)node[right]{$\beta$};
 \draw [->](0,0)--(0,9) node[above]{$\alpha$};
   \draw (-7.8, 0).. controls (-7.8, 2) and (-7.8, 4)..(-7.8,5);
 \draw (0, 0).. controls (3, 4) and (8, 4)..(9,6.2);
 \draw (2.4, 2.27).. controls (4,4.01) and (5, 4.02)..(7.8,5);
  \draw (7.8, 0).. controls (8.3,1) and (9.05, 2.5)..(9,3);
  \draw (-7.8, 0).. controls (-8.3,1) and (-9.05, 2.5)..(-9,3);
  \draw (-7.8, 5).. controls (-7.78,3) and (-6.8, 2.5)..(-6.5,0);
  \draw (-9, 0).. controls (-9,3) and (-9, 4)..(-9,7);
   \draw (9, 0).. controls (9,3) and (9, 4)..(9,7);
    \path (-7.8, 5)  edge [-,dotted](-7.8, 7) [line width=0.7pt];
   \path (7.8, 0)  edge [-,dotted](7.8, 5) [line width=0.7pt];
   \path (2.4, 0)  edge [-,dotted](2.4, 2.2) [line width=0.7pt];
    \node (a) at (-9.2,-0.5) {\tiny$-{9\over 16}\pi^2$};
    \node (a) at (-7.7,-0.5) {\tiny$-{1\over 2}\pi^2$};
    \node (a) at (-6.5,-0.5) {\tiny$\beta_l$};
    \node (a) at (0,-0.5) {\tiny$0$};
    \node (a) at (7.7,-0.5) {\tiny${1\over 2}\pi^2$};
    \node (a) at (9.2,-0.5) {\tiny${9\over 16}\pi^2$};
   \node (a) at (2.4,-0.5) {\tiny${\sqrt{3}-1\over 4}\pi^2$};
    \node (a) at (-10.5,4) {\tiny$I$};
  \node (a) at (10.5,4) {\tiny$VIII$};
    \node (a) at (-8.5,0.5) {\tiny$II$};
     \node (a) at (-8.5,3) {\tiny$III$};
      \node (a) at (-7.2,1) {\tiny$IV$};
      \node (a) at (-1,3) {\tiny$IX$};
      \node (a) at (4.45,3.57) {\tiny$V$};
       \node (a) at (5,2) {\tiny$VII$};
        \node (a) at (8.5,0.5) {\tiny$VI$};
 \end{tikzpicture}
\end{center}\vspace{-0.2cm}
 \begin{center}\vspace{-0.2cm}
   {\small {\bf Figure 3.} }
  \end{center}

In Figure $3$,
\begin{align*}
I&=\{(\alpha,\beta)|\alpha>0,\beta<-{9\over 16}\pi^2\},\\
II&=\{(\alpha,\beta)|0<\alpha<\pi\sqrt{-r^2-r+{3\over4}},-{9\over 16}\pi^2\leq\beta<-{1\over 2}\pi^2,r\in[{1\over4},{1\over2})\},\\
III&=\{(\alpha,\beta)|\alpha\geq\pi\sqrt{-r^2-r+{3\over4}},-{9\over 16}\pi^2\leq\beta<-{1\over 2}\pi^2,r\in[{1\over4},{1\over2})\}\cup\\
&\{(\alpha,\beta)|0<\alpha<{\sqrt{3}\over2}\pi,\beta=-{1\over2}\pi^2 \},\\
IV&=\{(\alpha,\beta)|0<\alpha<\sqrt{\Lambda_\beta},-{1\over 2}\pi^2<\beta<\beta_l\},\\
V&=\{(\alpha,\beta)|\pi\sqrt{1-r^2}<\alpha<\sqrt{\Lambda_\beta},{\sqrt{3}-1\over 4}\pi^2<\beta<{1\over2}\pi^2\},\\
VI&=\{(\alpha,\beta)|0<\alpha<\pi\sqrt{-r^2-r+{3\over4}},{1\over 2}\pi^2<\beta\leq{9\over 16}\pi^2,r\in[{1\over4},{1\over2})\},\\
VII&=\{(\alpha,\beta)|0<\alpha<\sqrt{\Lambda_\beta},0<\beta\leq{\sqrt{3}-1\over 4}\pi^2\}\cup\\
&\{(\alpha,\beta)|0<\alpha\leq\pi\sqrt{1-r^2},{\sqrt{3}-1\over 4}\pi^2<\beta<{1\over2}\pi^2, r\in({1\over2},{\sqrt{3}\over2})\}\cup\\
&\{(\alpha,\beta)|\pi\sqrt{-r^2-r+{3\over4}}\leq\alpha<\pi\sqrt{1-r^2},{1\over 2}\pi^2\leq\beta\leq{9\over 16}\pi^2,r\in[{1\over4},{1\over2}]\},\\
VIII&=\{(\alpha,\beta)|\alpha>0,\beta>{9\over 16}\pi^2\},\\
IX&=\{(\alpha,\beta)|\alpha\geq{\sqrt{3}\over2}\pi,\beta=-{1\over 2}\pi^2\}\cup\{(\alpha,\beta)|\alpha>\sqrt{\Lambda_\beta},-{1\over 2}\pi^2<\beta<{1\over 2}\pi^2\}\cup\{(\alpha,\beta)|\\
&\alpha=\sqrt{\Lambda_\beta},0<\beta\leq{\sqrt{3}-1\over 4}\pi^2\}\cup\{(\alpha,\beta)|\alpha\geq\pi\sqrt{1-r^2},{1\over 2}\pi^2\leq\beta\leq{9\over 16}\pi^2,r\in[{1\over4},{1\over2}]\},
\end{align*}
where $\Lambda_\beta=\sup_{c\notin(0,1)}\max\{-\lambda_1(c),0\},$ $r={1\over4}+\sqrt{{9\over16}+{\beta\over\pi^2}}$ for $-{9\over 16}\pi^2\leq\beta<0$,
$r={1\over4}+\sqrt{{9\over16}-{\beta\over\pi^2}}$ for $0<\beta\leq{9\over 16}\pi^2$,  and $\beta_{l}$ is given by Theorem $6$ of \cite{LYZ}.
Moreover,  $m_{\beta}=0$ and $\max\{M_{\beta},0\}=\Lambda_{\beta}$ for $\beta\in[-{1\over2}\pi^2,0)\cup(0,{9\over16}\pi^2]$, and $m_{\beta}=1$ and $M_{\beta}=(-r^2-r+{3\over4})\pi^2$ for $\beta\in[-{9\over16}\pi^2,-{1\over2}\pi^2)$.

The explicit number  $\sharp\left(\sigma_d(\mathcal{R}_{\alpha,\beta})\cap\mathbf{R}\right)$ is given as follows:
\begin{align}\nonumber
(\alpha,\beta)\in IX&\Longrightarrow\sharp\left(\sigma_d(\mathcal{R}_{\alpha,\beta})\cap\mathbf{R}\right)=0;\\\nonumber
(\alpha,\beta)\in III\cup VII&\Longrightarrow\sharp\left(\sigma_d(\mathcal{R}_{\alpha,\beta})\cap\mathbf{R}\right)=1; \\\label{explicit number}
(\alpha,\beta)\in IV\cup V\cup VI&\Longrightarrow\sharp\left(\sigma_d(\mathcal{R}_{\alpha,\beta})\cap\mathbf{R}\right)=2;\\\nonumber
(\alpha,\beta)\in II&\Longrightarrow\sharp\left(\sigma_d(\mathcal{R}_{\alpha,\beta})\cap\mathbf{R}\right)=3;\\\nonumber
(\alpha,\beta)\in I\cup VIII&\Longrightarrow\sharp\left(\sigma_d(\mathcal{R}_{\alpha,\beta})\cap\mathbf{R}\right)=\infty.
\end{align}
In addition, $\sharp\left(\sigma_d(\mathcal{R}_{\alpha,\beta})\cap\mathbf{R}\right)=1$ if $(\alpha,\beta)\in\Gamma:=\{(\alpha,\beta)|\alpha=\sqrt{\Lambda_\beta}, -{1\over2}\pi^2<\beta<\beta_l\text{ or }$ ${\sqrt{3}-1\over 4}\pi^2<\beta<{\pi^2\over2}\}$. Now we fix $\alpha=2\pi/T$.

The number (denoted by $\theta$) of traveling wave families  near the sinus profile is given by
\begin{align}\nonumber
(\alpha,\beta)\in IX&\Longrightarrow\theta=0;\\\label{number of traveling wave families for sinus flow}
(\alpha,\beta)\in  VII&\Longrightarrow1\leq\theta<\infty;\\\nonumber
(\alpha,\beta)\in IV\cup V\cup VI &\Longrightarrow2\leq\theta<\infty; \\\nonumber
(\alpha,\beta)\in I\cup II\cup III\cup VIII,\ \beta\neq -\frac{\pi^2}{2}&\Longrightarrow\theta=\infty.
\end{align}
In addition, $\theta=\sharp(\sigma_d(\mathcal{R}_{\alpha,\beta})\cap\mathbf{R})=1$ if $(\alpha,\beta)\in\Gamma$.
Moreover,
\begin{align}\label{number of traveling wave families for sinus flow 2}
(\alpha,\beta)\in III\cup IV\cup V\cup VII \Longrightarrow\theta=\sum_{k\geq1}\sharp(\sigma_d(\mathcal{R}_{k\alpha,\beta})\cap\mathbf{R}).\end{align}
If $(\alpha,\beta)\in III\cup IV\cup V\cup VII\cup\Gamma$, then
\begin{align}\label{number of traveling wave families for sinus flow with minimal period T}
 c_0\in\sigma_d(\mathcal{R}_{\alpha,\beta})\cap\mathbf{R}\Longrightarrow
\vec{u}_{\varepsilon}\left(
x-c_{\varepsilon}t,y\right) \text{ has minimal period } {2\pi/\alpha} \text{ in } x
\end{align}
for $\varepsilon>0$ small enough,  where $\vec{u}_{\varepsilon}\left(
x-c_{\varepsilon}t,y\right)$ satisfies the assumption of Lemma $\ref{traveling wave concentration}$.


To prove \eqref{explicit number}--\eqref{number of traveling wave families for sinus flow with minimal period T}, we need  the following asymptotic behavior, signatures and monotonicity of $\lambda_n$.

$(\rm{1})$ $\lim_{c\to\pm\infty}\lambda_{n}(c)={n^2\over 4}\pi^2>0$  for $\beta\in\mathbf{R}$;

$(\rm{2})$ $\lim_{c\to0^-}\lambda_n(c)=-\infty$  and $\lambda_n$ is decreasing on $(-\infty,0)$ for $\beta\in({9\over16}\pi^2,\infty)$;

$(\rm{3})$ $\lim_{c\to0^-}\lambda_n(c)=\lambda_n(0)$, $\lambda_1(0)<\lambda_2(0)\leq0$, $\lambda_3(0)>0$  and $\lambda_n$ is decreasing on $(-\infty,0)$ for $\beta\in[{1\over2}\pi^2,{9\over16}\pi^2]$;

$(\rm{4})$ $\lim_{c\to0^-}\lambda_n(c)=\lambda_n(0)$,  $\lambda_1(0)<0$, $\lambda_2>0$ on $(-\infty,0]$, there exist $c_1<c_2\in (-\infty,0)$ such that $\lambda_1(c)\geq\lambda_1(c_1)=0$ for $c\in(-\infty,c_1)$,  $\lambda_1$ is decreasing on $(c_1,c_2)$, $ \lambda_1(c_2)=\inf_{c\in(-\infty,0)}
\lambda_1(c)$ and $\lambda_1$ is increasing on $(c_2,0)$ for $\beta\in({\sqrt{3}-1\over 4}\pi^2,{1\over2}\pi^2)$;

$(\rm{5})$ $\lim_{c\to0^-}\lambda_n(c)=\lambda_n(0)$, $\lambda_1(0)<0$,  $\lambda_2>0$ on $(-\infty,0]$, there exists $c_1\in (-\infty,0)$ such that $\lambda_1(c)\geq\lambda_1(c_1)=0$ for $c\in(-\infty,c_1)$ and $\lambda_1$ is decreasing on $(c_1,0)$ for $\beta\in(0,{\sqrt{3}-1\over 4}\pi^2]$;

$(\rm{6})$ $\lambda_1\geq0$ on $(1,\infty)$ for $\beta\in[\beta_l,0)$;

$(\rm{7})$ $\lim_{c\to1^+}\lambda_n(c)=\lambda_n(1)$, $\lambda_1(1)>0$,  $\lambda_2>0$ on $(1,\infty)$, there exist $c_1<c_2<c_3\in (1,\infty)$ such that $\lambda_1(c)\geq\lambda_1(c_1)=\lambda_1(c_3)=0$ for $c\in(1,c_1)\cup(c_3,\infty)$,  $\lambda_1$ is decreasing on $(c_1,c_2)$, $ \lambda_1(c_2)=\inf_{c\in(1,\infty)}\lambda_1(c)$ and $\lambda_1$ is increasing on $(c_2,c_3)$   for $\beta\in(-{1\over2}\pi^2,\beta_l)$;

 $(\rm{8})$ $\lim_{c\to1^+}\lambda_n(c)=\lambda_{n}(1)$, $\lambda_1(1)<0$, $\lambda_2(1)=0$ and $\lambda_n$ is increasing on $(1,\infty)$ for $\beta=-{1\over2}\pi^2$;

  $(\rm{9})$ $\lim_{c\to1^+}\lambda_{1}(c)=-\infty$, $\lim_{c\to1^+}\lambda_{n+1}(c)=\lambda_{n}(1)$, $\lambda_1(1)=\lambda_2(1)<0$, $\lambda_3 (1)>0$ and $\lambda_n$ is increasing on $(1,\infty)$ for $\beta\in[-{9\over16}\pi^2,-{1\over2}\pi^2)$;

  $(\rm{10})$ $\lim_{c\to1^+}\lambda_{n}(c)=-\infty$ and $\lambda_n$ is increasing on $(1,\infty)$ for $\beta\in(-\infty,-{9\over16}\pi^2)$,
  \\
  where  $n\geq 1$, $\lambda_n(0)=\left(\left(r+{n-1\over2}\right)^2-1\right)\pi^2$ for $\beta\in(0,{9\over16}\pi^2]$,
$\lambda_n(1)= \left(\left(r-{1\over2}+\lceil{n\over2}\rceil\right)^2-1\right)$ $\pi^2$ for $\beta\in[-{9\over16}\pi^2,-{1\over2}\pi^2)\cup(-{1\over2}\pi^2,0)$, and $\lambda_n(1)=({n^2\over4}-1)\pi^2$ for $\beta=-{1\over2}\pi^2$ by Proposition $1$ in \cite{LYZ}.

Assertions $(1)$--$(10)$ provide pictures of the negative eigenvalues of $\mathcal{L}_c$ for fixed $\beta$.
Assume that $(1)$--$(10)$ are true. Note that
$\sigma_d(\mathcal{R}_{\alpha,\beta})\cap(1,\infty)=\emptyset$ if $\alpha\in\mathbf{R}$ and $\beta>0$, and
$\sigma_d(\mathcal{R}_{\alpha,\beta})\cap(-\infty,0)=\emptyset$ if $\alpha\in\mathbf{R}$ and $\beta<0$.
 Then we claim that
 \begin{align}\label{non-intersect}
&(\sigma_d(\mathcal{R}_{k\alpha,\beta})\cap\mathbf{R})\cap(\sigma_d(\mathcal{R}_{j\alpha,\beta})\cap\mathbf{R})=\emptyset,\\ \label{Rk}
&\sigma_d(\mathcal{R}_{k\alpha,\beta})\cap\R=\{c\in\R\setminus[0,1]|\lambda_1(c)=-(k\alpha)^2\}
 \end{align}
for   $k,j\in\mathbf{Z}^+$, $k\neq j$ and $(\alpha,\beta)\in III\cup IV\cup V\cup VII\cup\Gamma$.
In fact, \eqref{Rk} implies \eqref{non-intersect}.
If $(\alpha,\beta)\in III$, then by $(8)$--$(9)$ we have $\lambda_2>\lambda_2(1)\geq -\alpha^2$ on $(1,\infty)$, which gives \eqref{Rk}.
If  $(\alpha,\beta)\in IV\cup\Gamma$ with $ \beta<0$,
then by  $(7)$ we have
$\lambda_2>0$ on $(1,\infty)$ and thus, $\mathcal{L}_{c}$ has at most  one negative eigenvalue for   $c\in(1,\infty)$, which gives \eqref{Rk}.
Similarly, we can prove \eqref{Rk} for  $(\alpha,\beta)\in V \cup VII\cup\Gamma$ with $0<\beta<{\pi^2\over2}$
 by $(4)$--$(5)$.
If $(\alpha,\beta)\in VII$ with $\beta\geq {\pi^2\over2}$, then by $(3)$ we have $\lambda_2>\lambda_2(0)\geq -\alpha^2$ on $(-\infty,0)$,
which gives \eqref{Rk}.

By applying Theorem $\ref{the explicit number of traveling wave families}$,
we get \eqref{explicit number}--\eqref{number of traveling wave families for sinus flow 2}.
\eqref{number of traveling wave families for sinus flow with minimal period T} is a direct consequence of Lemma $\ref{traveling wave concentration ad}$. Here, \eqref{non-intersect} is  used in the proof of \eqref{number of traveling wave families for sinus flow 2}--\eqref{number of traveling wave families for sinus flow with minimal period T}.

Using \eqref{number of traveling wave families for sinus flow 2} we can evaluate $ \theta$ for $(\alpha,\beta)\not\in VI$ as follows.
\begin{align}\nonumber
\beta\in (-\infty,-{1\over2}\pi^2)\cup({9\over16}\pi^2,+\infty)&\Longrightarrow\theta=\infty;\\\label{number1}
\beta=-{1\over2}\pi^2&\Longrightarrow\theta=\lceil\frac{\sqrt{3}\pi}{2\alpha}\rceil-1;\\\nonumber
-{1\over 2}\pi^2<\beta<\beta_l&\Longrightarrow\theta=\lfloor\frac{\sqrt{\Lambda_{\beta}}}{\alpha}\rfloor+\lceil\frac{\sqrt{\Lambda_{\beta}}}{\alpha}\rceil-1; \\\nonumber
\beta_l\leq \beta\leq0&\Longrightarrow\theta=0;\\\nonumber
0<\beta\leq{\sqrt{3}-1\over 4}\pi^2&\Longrightarrow\theta=\lceil\frac{\sqrt{\Lambda_{\beta}}}{\alpha}\rceil-1;\\\nonumber
{\sqrt{3}-1\over 4}\pi^2<\beta<{1\over2}\pi^2&\Longrightarrow\theta=\lfloor\frac{\sqrt{\Lambda_{\beta}}}{\alpha}\rfloor+\lceil\frac{\sqrt{\Lambda_{\beta}}}{\alpha}\rceil
-\lfloor\frac{\pi\sqrt{1-r^2}}{\alpha}\rfloor-1;\\\nonumber\pi\sqrt{-r^2-r+{3\over4}}\leq\alpha,\ {1\over 2}\pi^2\leq\beta\leq{9\over 16}\pi^2&\Longrightarrow\theta=\lceil\frac{{\pi\sqrt{1-r^2}}}{\alpha}\rceil-1.
\end{align}
The case $(\alpha,\beta)\in VI$ is more complicated. By $(3)$, we have $ \theta=\sharp(A_1\cup A_2)$  with $A_i=\{c<0|\lambda_i(c)=-(k\alpha)^2,\ k\in\mathbf{Z}^+\}$. Then
 by $(3)$ and the expression of $\lambda_i(0)$, $i=1,2$, we have $\sharp(A_1)=\lceil\frac{{\pi\sqrt{1-r^2}}}{\alpha}\rceil-1 $, $\sharp(A_2)=\lceil\frac{\pi\sqrt{-r^2-r+{3\over4}}}{\alpha}\rceil-1 $, and $\sharp(A_1)\leq \theta\leq \sharp(A_1)+\sharp(A_2)$, i.e. $\lceil\frac{{\pi\sqrt{1-r^2}}}{\alpha}\rceil-1\leq \theta\leq \lceil\frac{{\pi\sqrt{1-r^2}}}{\alpha}\rceil+\lceil\frac{\pi\sqrt{-r^2-r+{3\over4}}}{\alpha}\rceil-2$. In fact, $\theta= \sharp(A_1)+\sharp(A_2)-\sharp(A_1\cap A_2)$, but it seems difficult to give an explicit formula if $A_1\cap A_2\neq\emptyset.$

Now, we prove $(1)$--$(10)$. $(1)$ and $(4)$--$(7)$ are a summary of spectral results in Section $4$ of \cite{LYZ}. Monotonicity of $\lambda_n$ for $\beta\in(-\infty,-{1\over2}\pi^2]\cup[{1\over2}\pi^2,\infty)$ is due to Corollary $1$ in \cite{LYZ}. Asymptotic behavior of $\lambda_n$ in $(2)$ and $(10)$ is obtained by Corollary $\ref{sturm-liouville-eigen-asym}$. Signatures of $\lambda_n$ in $(3)$ and $(8)$--$(9)$ are due to Proposition $1$ in \cite{LYZ} and simple computation.

The rest is to prove the asymptotic behavior of $\lambda_n$ in $(3)$ and $(8)$--$(9)$.
First, we consider $\beta=\pm{1\over 2}\pi^2$. We only prove that $\lim_{c\to0^-}\lambda_n(c)=\lambda_n(0)$ for $\beta={1\over2}\pi^2$ and $n\geq1$. Note that ${\beta-u''\over u}=\pi^2$ and
$\left\|{\beta-u''\over u-c}-{\beta-u''\over u}\right\|_{L^1(-1,1)}=\pi^2\left\|{c\over u-c}\right\|_{L^1(-1,1)}$ for $c<0$. Let $0<\delta<1$.
Then
\begin{align*}
\left\|{c\over u-c}\right\|_{L^1(1-\delta,1)}={2\over\pi}\int_0^{\cos({\pi\over2}(1-\delta))}{-c\over z^2-c}{1\over \sqrt{1-z^2}}dz\leq C_\delta\sqrt{-c}\to0
\end{align*}
as $c\to0^-$. Similarly,  $\left\|{c\over u-c}\right\|_{L^1(-1,-1+\delta)}\to0$. Clearly, $\left\|{c\over u-c}\right\|_{L^1(-1+\delta,1-\delta)}\to 0$. Then $\left\|{c\over u-c}\right\|_{L^1(-1,1)}\to 0$.
It follows from Theorem $2.1$ in \cite{KWZ} that $\lim_{c\to0^-}\lambda_n(c)=\lambda_n(0)$ for  $n\geq1$.

We then consider $\beta\in({1\over2}\pi^2,{9\over16}\pi^2)$. We use the eigenfunctions of $\lambda_n(\beta,0)$ in Proposition $1$ $(\rm{iii})$ of \cite{LYZ}. Here, we rewrite
 $\lambda_n(\beta,c)=\lambda_n(c)$ to indicate its dependence on $\beta$ if necessary.
There exist $ \phi_n(y)=\phi_n^{(\beta,0)}(y)=\cos^{2r}(\frac{\pi}{2}y)P_{n-1}(\sin(\frac{\pi}{2}y))$, $n\geq 1$, satisfying\begin{align*}
&-\phi_n''-\frac{\beta-u''}{u}\phi_n=\lambda_n(0)\phi_n\quad \text{on}\quad(-1,1),\quad\phi_n(\pm 1)=0.
\end{align*}
Here, $\lambda_n(\beta,0)=\left(\left(r+{n-1\over2}\right)^2-1\right)\pi^2$, $r={1\over4}+\sqrt{{9\over16}-{\beta\over\pi^2}}$ and $P_{n-1}(\cdot)$ is a polynomial
 with order $n-1$.
Moreover, $\phi_n\in H_0^1(-1,1)$ is real-valued, and we normalize it such that $\|\phi_n\|_{L^2(-1,1)}=1.$ Then we have for $m, n\geq1$, \begin{align*}
&\int_{-1}^{ 1}\phi_n\phi_mdy=\delta_{mn}=\left\{\begin{array}{l@{\ \text{if}\ }l}1&n= m,\\0&n\neq m,\end{array}\right.\quad\int_{-1}^{1}\left(\phi_n'\phi_m'+{u''-\beta\over u}\phi_n\phi_m\right)dy=\lambda_n(0)\delta_{mn}.
\end{align*} Note that $\phi_n $ has $ n-1$ zeros in $(-1,1)$, and we denote $Z_n:=\{y\in(-1,1):\phi_n(y)=0\}=\{a_{n,1},\cdots,a_{n,n-1}\}$.
For any $n$-dimensional subspace $V=\text{span}\{\psi_1,\cdots,\psi_{n}\}$ in $H_0^1(-1,1)$,
there exists  $0\neq(\xi_1,\cdots,\xi_{n})\in\mathbf{R}^{n}$ such that
$
\xi_1\psi_1(a_{n,i})+\cdots+\xi_{n}\psi_{n}(a_{n,i})=0, i=1,\cdots,n-1.
$
Define
$
\tilde\psi=\xi_1\psi_1+\cdots+\xi_{n}\psi_{n}.
$
Then $\tilde\psi(a_{n,i})=0$, $i=1,\cdots,n-1$, i.e.  $\tilde\psi|_{Z_n}=0.$ We normalize $\tilde\psi$ such that $\|\tilde\psi\|_{L^2(-1,1)}=1$. Since $ \tilde\psi\in H_0^1(-1,1)$, we  have $\tilde\psi(\pm1)=0$. Similar to \eqref{phiyto0}, we have $|\tilde\psi(y)|^2\phi_n'(y)/\phi_n(y)\to 0 $ as $y\to a_{n,i}$ or $ y\to -1^+$ or $ y\to 1^-$, where $1\leq i\leq n-1$. Integration by parts gives \begin{align*}
&\left\|\tilde\psi'-\tilde\psi{\phi_n'\over\phi_n}\right\|_{L^2(-1,1)}^2=\int_{-1}^{ 1}\left(|\tilde\psi'|^2+{\phi_n''\over \phi_n}|\tilde\psi|^2\right)dy=\int_{-1}^{ 1}\left(|\tilde\psi'|^2-\frac{\beta-u''}{u}|\tilde\psi|^2-\lambda_n(0)|\tilde\psi|^2\right)dy.
\end{align*} If $c<0$, then using $ \beta-u''>0$ and $u-c>0$ for $y\in[-1,1]$, we have\begin{align*}
&\int_{-1}^{ 1}\left(|\tilde\psi'|^2-\frac{\beta-u''}{u-c}|\tilde\psi|^2\right)dy\geq\int_{-1}^{ 1}\left(|\tilde\psi'|^2-\frac{\beta-u''}{u}|\tilde\psi|^2\right)dy\geq\int_{-1}^{ 1}\lambda_n(0)|\tilde\psi|^2dy=\lambda_n(0).
\end{align*}This, along with \eqref{variation2}, yields that
$
\inf_{c\in(-\infty,0)}\lambda_{n}(c)\geq \lambda_n(0).
$ Now, we consider the upper bound.
 Let
$V_{n}=\text{span}\{\phi_1,\cdots,\phi_n\}.$ Then $V_{n}\subset H_0^1(-1,1)$.  By \eqref{variation2}, there exist $b_{i,c}\in\mathbf{R}$, $i=1,\cdots,n$, with $\sum_{i=1}^{n}|b_{i,c}|^2=1 $ such that $\varphi_c=\sum_{i=1}^{n}b_{i,c}\phi_i\in V_{n}$ with $\|\varphi_c\|^2_{L^2}=1$, and
\begin{align}\nonumber
\lambda_{n}(c)\leq &\sup_{\|\phi\|_{L^2}=1,\phi\in V_{n}}\int_{-1}^{1}\left(|\phi'|^2+{u''-\beta\over u-c}|\phi|^2\right)dy=
\int_{-1}^{1}\left(|\varphi_c'|^2+{u''-\beta\over u-c}|\varphi_c|^2\right)dy\\ \nonumber
=&\sum_{i=1}^{n}|b_{i,c}|^2\int_{-1}^{1}\left(|\phi_i'|^2+{u''-\beta\over u-c}|\phi_i|^2\right)dy+\!\!\!\!\sum_{1\leq i<j\leq n}\!\!\!2b_{i,c}b_{j,c}\int_{-1}^{1}\left(\phi_i'\phi_j'+{u''-\beta\over u-c}\phi_i\phi_j\right)dy\\ \nonumber
\leq&\max_{1\leq i\leq n}\int_{-1}^{1}\left(|\phi_i'|^2+{u''-\beta\over u-c}|\phi_i|^2\right)dy+\!\!\!\!\sum_{1\leq i<j\leq n}\!\left|\int_{-1}^{1}\left(\phi_i'\phi_j'+{u''-\beta\over u-c}\phi_i\phi_j\right)dy\right|\\ \nonumber\to&\max_{1\leq i\leq n}\int_{-1}^{1}\left(|\phi_i'|^2+{u''-\beta\over u}|\phi_i|^2\right)dy+\!\!\!\!\sum_{1\leq i<j\leq n}\!\left|\int_{-1}^{1}\left(\phi_i'\phi_j'+{u''-\beta\over u}\phi_i\phi_j\right)dy\right|\\ \nonumber =&\max_{1\leq i\leq n}\lambda_i(0)+\!\!\sum_{1\leq i<j\leq n}0=\lambda_n(0),\;\text{as}\;\;c\to 0^-.
\end{align}Combining the upper and lower bounds, we have $\lim_{c\to0^-}\lambda_n(c)=\lambda_n(0)$ for $\beta\in({1\over2}\pi^2,{9\over16}\pi^2).$

Now, we consider $\beta={9\over16}\pi^2$. By Corollary $1$ $(\rm{i})$ in \cite{LYZ}, we have for fixed $c<0$, $\lambda_n(\beta,c)\leq \lambda_n(\beta',c) $ if $ \beta'<\beta$. As $\lambda_n(\beta',c)\geq \lambda_n(\beta',c')$ if $c<c'<0$ (see Corollary $1$ $(\rm{iv})$ in \cite{LYZ}) and $\lim_{c\to0^-}\lambda_n(\beta',c)=\lambda_n(\beta',0)$,
we have for fixed $\beta'\in({1\over2}\pi^2,{9\over16}\pi^2)$, $\lambda_n(\beta',c)\geq \lambda_n(\beta',0)$ if $c<0$.
 Then\begin{align*}
&\lim_{c\to0^-}\lambda_n(\beta,c)\leq \liminf_{\beta'\to\beta^-}\lim_{c\to0^-}\lambda_n(\beta',c)=\liminf_{\beta'\to\beta^-}\lambda_n(\beta',0)=\lambda_n(\beta,0),
\\
&\lim_{c\to0^-}\lambda_n(\beta,c)=\lim_{c\to0^-}\lim_{\beta'\to\beta^-}\lambda_n(\beta',c)\geq\lim_{\beta'\to\beta^-}\lambda_n(\beta',0)
=\lambda_n(\beta,0).
\end{align*}Here, we used the left continuity of $ \lambda_n(\cdot,0)$ at $\beta={9\over16}\pi^2. $
Thus, $\lim_{c\to0^-}\lambda_n(\beta,c)=\lambda_n(\beta,0). $

Next, we consider $\beta\in(-\frac{9}{16}\pi^2,-\frac{1}{2}\pi^2)$.
By Proposition $1$ $(\rm{iv})$ in \cite{LYZ},
there exists $ \phi_n(y)=\phi_n^{(\beta,1)}(y)=|\sin(\frac{\pi}{2}y)|^{2r}P_{n}(\cos(\frac{\pi}{2}y))$ if $n$ is  odd; there exists $ \phi_n(y)=\phi_n^{(\beta,1)}(y)=\text{sign}(y)$ $|\sin(\frac{\pi}{2}y)|^{2r}P_{n-1}(\cos(\frac{\pi}{2}y))$ if $n$ is even; and for $n\geq1$,  \begin{align*}
&-\phi_n''-\frac{\beta-u''}{u-1}\phi_n=\lambda_n(1)\phi_n\quad \text{on}\quad(-1,1)\setminus\{0\},\quad\phi_n(\pm 1)=0.
\end{align*}
Here, $\lambda_n(\beta,1)=\left(\left(r-{1\over2}+\lceil{n\over2}\rceil\right)^2-1\right)\pi^2$ and $r={1\over4}+\sqrt{{9\over16}+{\beta\over\pi^2}}$.
Moreover,
 $\phi_n\in H_0^1(-1,1)$ is real-valued and we  normalize it such that $\|\phi_n\|_{L^2(-1,1)}=1 .$ Then for $m,n\geq1$,\begin{align*}
&\int_{-1}^{ 1}\phi_n\phi_mdy=\delta_{mn},\quad\int_{-1}^{1}\left(\phi_n'\phi_m'+{u''-\beta\over u-1}\phi_n\phi_m\right)dy=\lambda_n(1)\delta_{mn}.
\end{align*} If $n\geq 1$ is odd, then $\phi_n $ has $ n$ zeros in $(-1,1)$, and $0\in Z_n=\{y\in(-1,1):\phi_n(y)=0\}=\{a_{n,1},\cdots,a_{n,n}\}$.
 For any $(n+1)$-dimensional subspace $V=\text{span}\{\psi_1,\cdots, \psi_{n+1}\}$ in $H_0^1(-1,1)$,
there exists $0\neq(\xi_1,\cdots,\xi_{n+1})\in\mathbf{R}^{n+1}$ such that
$
\xi_1\psi_1(a_{n,i})+\cdots+\xi_{n+1}\psi_{n+1}(a_{n,i})=0, i=1,\cdots,n.
$
Define
$
\tilde\psi=\xi_1\psi_1+\cdots+\xi_{n+1}\psi_{n+1}.
$
Then $\tilde\psi(a_{n,i})=0$, $i=1,\cdots,n$, i.e.  $\tilde\psi|_{Z_n}=0.$ We normalize $\tilde\psi$ such that $\|\tilde\psi\|_{L^2(-1,1)}=1$. Since $ \tilde\psi\in H_0^1(-1,1)$, we have $\tilde\psi(\pm1)=0$. Integration by parts gives\begin{align*}
&\left\|\tilde\psi'-\tilde\psi{\phi_n'\over\phi_n}\right\|_{L^2(-1,1)}^2=\int_{-1}^{ 1}\left(|\tilde\psi'|^2+{\phi_n''\over \phi_n}|\tilde\psi|^2\right)dy=\int_{-1}^{ 1}\left(|\tilde\psi'|^2-\frac{\beta-u''}{u-1}|\tilde\psi|^2-\lambda_n(1)|\tilde\psi|^2\right)dy.
\end{align*}
If $c>1$, then using $ \beta-u''<0$ and $u-c<0$ for $y\in[-1,1]$, we have\begin{align*}
&\int_{-1}^{ 1}\left(|\tilde\psi'|^2-\frac{\beta-u''}{u-c}|\tilde\psi|^2\right)dy\geq\int_{-1}^{ 1}\left(|\tilde\psi'|^2-\frac{\beta-u''}{u-1}|\tilde\psi|^2\right)dy\geq\int_{-1}^{ 1}\lambda_n(1)|\tilde\psi|^2dy=\lambda_n(1).
\end{align*}This, along with \eqref{variation2}, yields that
$
\inf_{c\in(1,+\infty)}\lambda_{n+1}(c)\geq \lambda_n(1).
$ If $n\geq 1$ is even, then $\lambda_{n+1}(c)\geq\lambda_{n}(c)\geq\lambda_{n-1}(1)=\lambda_{n}(1)$. Thus, $
\inf_{c\in(1,+\infty)}\lambda_{n+1}(c)\geq \lambda_n(1)
$ is always true.
 Now, we consider the upper bound. Let
$V_{n+1}=\text{span}\{\phi_0,\phi_1,\cdots,\phi_n\},$  here $\phi_0=\eta|_{[-1,1]} $ is defined in \eqref{def-eta}, and $\phi_1,\cdots,\phi_n $ are $L^2$ normalized eigenfunctions. Then $V_{n+1}\subset H_0^1(-1,1)$. By \eqref{variation2}, for $c>1,$ there exists $\varphi_c\in V_{n+1}$ with $\|\varphi_c\|^2_{L^2}=1$, and \begin{align*}
\lambda_{n+1}(c)\leq &\sup_{\|\phi\|_{L^2}=1,\phi\in V_{n+1}}\int_{-1}^{1}\left(|\phi'|^2+{u''-\beta\over u-c}|\phi|^2\right)dy=
\int_{-1}^{1}\left(|\varphi_c'|^2+{u''-\beta\over u-c}|\varphi_c|^2\right)dy.
\end{align*}Since $V_{n+1}\subset H_0^1(-1,1) $ is finite dimensional, there exist $ \varphi_1\in V_{n+1}$ and $c_m\to1^+$ such that $\varphi_{c_m}\to\varphi_1$ in $ H_0^1(-1,1).$ Then  $\|\varphi_{1}\|^2_{L^2}=1,$ $\int_{-1}^{1}|\varphi_{c_m}'|^2dy\to\int_{-1}^{1}|\varphi_{1}'|^2dy$ and \begin{align*}
\lambda_{n+1}(c_m)\leq &
\int_{-1}^{1}\left(|\varphi_{c_m}'|^2+{u''-\beta\over u-c_m}|\varphi_{c_m}|^2\right)dy.
\end{align*} Since $u''-\beta>0$ and $u-c<0 $ for $y\in[-1,1]$ and $c>1$,   by Fatou's Lemma, we have\begin{align*}
&\limsup_{m\to\infty}
\int_{-1}^{1}{u''-\beta\over u-c_m}|\varphi_{c_m}|^2dy\leq\int_{-1}^{1}\limsup_{m\to\infty}{u''-\beta\over u-c_m}|\varphi_{c_m}|^2dy=\int_{-1}^{1}{u''-\beta\over u-1}|\varphi_{1}|^2dy.
\end{align*}In particular, if $\varphi_{1}(0)\neq 0$, then\begin{align*}
&\int_{-1}^{1}{u''-\beta\over u-1}|\varphi_{1}|^2dy=-\infty,\quad\limsup_{m\to\infty}
\int_{-1}^{1}{u''-\beta\over u-c_m}|\varphi_{c_m}|^2dy=-\infty,\quad\limsup_{m\to\infty}\lambda_{n+1}(c_m)=-\infty.
\end{align*}If $\varphi_{1}(0)= 0$, then $\varphi_{1}\in\text{span}\{\phi_1,\cdots,\phi_n\}$ and\begin{align*}
&\limsup_{m\to\infty}\lambda_{n+1}(c_m)\leq\int_{-1}^{1}\left(|\varphi_{1}'|^2+{u''-\beta\over u-1}|\varphi_{1}|^2\right)dy.
\end{align*}As $\|\varphi_{1}\|^2_{L^2}=1 ,$ there exist $b_{i}\in\mathbf{R}$, $i=1,\cdots,n$, with $\sum_{i=1}^{n}|b_{i}|^2=1 $ such that $\varphi_1=\sum_{i=1}^{n}b_{i}\phi_i\in V_{n+1}$ and
\begin{align*}
\int_{-1}^{1}\left(|\varphi_1'|^2+{u''-\beta\over u-1}|\varphi_1|^2\right)dy&=\sum_{i=1}^{n}|b_{i}|^2\int_{-1}^{1}\left(|\phi_i'|^2+{u''-\beta\over u-1}|\phi_i|^2\right)dy
=\sum_{i=1}^{n}|b_{i}|^2\lambda_i(1)\\&\leq\max_{1\leq i\leq n}\lambda_i(1)=\lambda_n(1).
\end{align*}Therefore, if $\varphi_{1}(0)= 0$, then $
\limsup\limits_{m\to\infty}\lambda_{n+1}(c_m)\leq\lambda_n(1);
$ if $\varphi_{1}(0)\neq 0 $,  this is clearly true since the limit is $-\infty$ in this case. By monotonicity of $\lambda_n$, we have $
\lim\limits_{c\to1^+}\lambda_{n+1}(c)\leq\lambda_n(1).$
Combining the upper and lower bounds, we have $
\lim\limits_{c\to1^+}\lambda_{n+1}(c)=\lambda_n(1)$ for $\beta\in(-\frac{9}{16}\pi^2,-\frac{1}{2}\pi^2)$.

For $\beta=-\frac{9}{16}\pi^2 $, the limits $
\lim\limits_{c\to1^+}\lambda_{n+1}(c)=\lambda_n(1)$, $n\geq1$, can be proved similarly as in the case $\beta=\frac{9}{16}\pi^2.$
Finally, the limit $\lim_{c\to1^+}\lambda_{1}(c)=-\infty$ for $\beta\in[-\frac{9}{16}\pi^2,-\frac{1}{2}\pi^2)$ follows from Theorem $\ref{eigenvalue asymptotic behavior-bound}$ $(2)$.
\end{example}

\section*{Acknowledgement}
Z. Lin is
 partially supported by the NSF grants DMS-1715201 and DMS-2007457.
Z. Zhang is partially supported by NSF
of China under Grant 11425103.

\end{CJK*}


\begin{thebibliography}{99}
\bibitem {BD01} A. Barcilon, P. G. Drazin,  Nonlinear waves of vorticity, Stud. Appl. Math., 106 (2001),  437-479.
\bibitem {Balmforth-Piccolo} N. J. Balmforth, C. Piccolo, The onset of meandering in a barotropic jet, J. Fluid Mech., 449 (2001),
85-114.
\bibitem {BM15} J. Bedrossian, N. Masmoudi, Inviscid damping and the asymptotic stability of planar shear flows in the
2D Euler equations, Publ. Math. Inst. Hautes \'{E}tudes Sci., 122 (2015), 195-300.
\bibitem {BD93} M. Buchanan, J. J.  Dorning,  Superposition of nonlinear plasma waves,
Phys. Rev. Lett., 70 (1993), 3732-3735.
\bibitem {BD94} M. Buchanan, J. J.  Dorning, Near equilibrium multiple-wave plasma
states, Phys. Rev. E, 50 (1994), 1465-1478.
\bibitem {Case} K.M. Case, Stability of inviscid plane Couette flow, Phys. Fluids 3,  (1960), 143-148.
\bibitem {CR71}M. Crandall, P. Rabinowitz, Bifurcation from simple
eigenvalues, J. Funct. Anal., 8 (1971), 321-340.

\bibitem {CEW20}
M. Coti Zelati, T. M. Elgindi, K. Widmayer,
Stationary Structures near the Kolmogorov and Poiseuille Flows in the 2d Euler Equations, Preprint,
arXiv:2007.11547.
\bibitem {DZ90} L. Demeio, P. F. Zweifel,   Numerical simulations of perturbed Vlasov
equilibria, Phys. Fluids B, 2 (1990), 1252-1255.
\bibitem {Drazin}
 P. G. Drazin, Introduction to hydrodynamic stability, Cambridge Texts in Applied Mathematics, Cambridge
University Press, Cambridge, 2002.
\bibitem{EW} T. M. Elgindi, K. Widmayer,  Long time stability for solutions of a $\beta$-plane equation,
Comm. Pure Appl. Math., 70 (2017), 1425-1471.
\bibitem{Engevik} L. Engevik, A note on the barotropic instability of the Bickley jet, J. Fluid Mech., 499 (2004),  315-326.
\bibitem{Fjortoft} R. Fj{\o}rtoft, Application of integral theorems in deriving criteria of stability of laminar flow
and for baroclinic circular vortex, Geofys. Publ. Norske Vid.-Akad. Oslo, 17 (1950),
1-52.
\bibitem{GNRS} E. Grenier, T.T. Nguyen, F. Rousset and A. Soffer, Linear inviscid damping and enhanced viscous dissipation of shear flows by using the conjugate operator
method, J. Funct. Anal., 278 (2020), 108339.
\bibitem {Godfrey-Hardy-Littlewood} G. H. Hardy, J. E. Littlewood, G. P\'{o}lya, Inequalities, Cambridge Mathematical Library, Cambridge University Press,Cambridge (1988); reprint of the 1952 edition.
  \bibitem{Howard-Drazin}  L. N. Howard, P. G. Drazin, On instability of a parallel flow of inviscid fluid in a rotating system with
variable Coriolis parameter, J. Math. Phys., 43 (1964), 83-99.
\bibitem {Ionescu-Jia20} A. Ionescu, H. Jia, Inviscid damping near the Couette flow in a channel, Comm. Math. Phys., 374 (2020), 2015-2096.
\bibitem {Ionescu-Jia2001} A. Ionescu, H. Jia, Nonlinear inviscid damping near monotonic shear flows, Preprint, arXiv:2001.03087.
\bibitem {Jia2020-1} H. Jia, Linear inviscid damping in Gevrey spaces, Arch. Ration. Mech. Anal., 235 (2020), 1327-1355.
 \bibitem {Jia2020-2} H. Jia, Linear inviscid damping near monotone shear flows, SIAM J. Math. Anal., 52 (2020), 623-652.
    \bibitem {Kato1980}T. Kato, Perturbation theory for linear operators,
second edition, Springer-Verlag, Heidelberg, 1980.
   \bibitem {KWZ} Q. Kong, H. Wu, A. Zettl, Dependence of the $n$-th Sturm-Liouville eigenvalue on the problem, J. Differential Equations, 156 (1999), 328-354.
   \bibitem {Kuo1949}H. L. Kuo, Dynamic instability of two-dimensional
non-divergent flow in a barotropic atmosphere, J. Meteor., 6 (1949),  105-122.
    \bibitem {Kuo1974}H. L. Kuo, Dynamics of quasi-geostrophic flows and
instability theory, Adv. Appl. Mech., 13 (1974),  247-330.
\bibitem {LD09} C. Lancellotti, J. J. Dorning,  Nonlinear Landau damping,
Transport Theor. Stat.,  38 (2009), 1-146.
\bibitem {Li-Lin} C. Y. Li, Z. Lin, A resolution of the Sommerfeld Paradox, SIAM J. Math. Anal., 43 (2011), 1923-1954.
\bibitem {Lin03} Z. Lin, Instability of some ideal plane flows, SIAM J. Math. Anal., 35 (2003), 318-356.
\bibitem {LYZ} Z. Lin, J. Yang, H. Zhu,  Barotropic instability of shear flows, Stud. Appl. Math., 144 (2020), 289-326.
\bibitem{LZ} Z. Lin, C. Zeng, Inviscid dynamic structures near Couette flow, Arch. Rat. Mech. Anal., 200 (2011), 1075-1097.
\bibitem{Lipps} F. B. Lipps, The barotropic stability of the mean winds in the atmosphere, J. Fluid Mech., 12 (1962),
397-407.
\bibitem{Majda} A. Majda, Introduction to PDEs and waves for the atmosphere and ocean, Courant Lecture Notes in Mathematics, vol. 9, New York University, Courant Institute of Mathematical Sciences, New York; American
Mathematical Society, Providence, RI, 2003.
\bibitem{Maslowe} S. A. Maslowe, Barotropic instability of the Bickley jet, J. Fluid Mech., 229 (1991),  417-426.
\bibitem {Masmoudi} N. Masmoudi, About the Hardy inequality. In: Schleicher D., Lackmann M. (eds) An Invitation to Mathematics. Springer, Berlin, Heidelberg.
 \bibitem {Masmoudi-Zhao2001}   N. Masmoudi, W. Zhao, Nonlinear inviscid damping for a class of monotone shear flows in finite channel, Preprint, arXiv:2001.08564.
 \bibitem {McWilliams}     J. C. McWilliams, Fundamentals of geophysical fluid dynamics, Cambridge University Press, 2006.
    \bibitem {Pedlosky1987}J. Pedlosky, Geophysical fluid dynamics, 2nd
edn. Springer, New York (1987).
\bibitem{PW} F. Pusateri, K. Widmayer, On the global stability of a $\beta$-plane equation,
 Anal. PDE, 11 (2018), 1587-1624.
\bibitem{Orr}   W. Orr, Stability and instability of steady motions of a perfect liquid, Proc. Ir. Acad. Sect. A: Math
Astron. Phys. Sci., 27(1907), 9-66.
\bibitem{Rayleigh1880} L. Rayleigh, On the stability or instability of certain fluid motions, Proc. London Math.
Soc., 9 (1880),  57-70.
\bibitem{RS}  M. Reed, B. Simon,  Methods of modern mathematical physics. IV. Analysis of operators. Academic Press [Harcourt Brace Jovanovich, Publishers], New York-London, 1978. xv+396 pp.
 \bibitem{Rosencrans-Sattinger} S. I. Rosencrans, D. H. Sattinger,  On the spectrum of an operator occuring in the theory of hydrodynamic stability,  J. Math. Phys.,  45 (1966), 289-300.
    \bibitem {Rossby1939} C. G. Rossby, Relation between variations in the
intensity of the zonal circulation of the atmosphere and the displacements of
the semi-permanent centers of action, J. Mar. Res., 2 (1939),  38-55.
\bibitem{Tung1981} K. K. Tung,  Barotropic instability of zonal flows, J.
Atmos. Sci., 38 (1981), 308-321.
\bibitem{WZZ-1} D. Wei, Z. Zhang and W. Zhao, Linear inviscid damping for a class of monotone shear flow in Sobolev spaces, Commun. Pure Appl. Math., 71 (2018), 617-687.
\bibitem{WZZ-2} D. Wei, Z. Zhang and W. Zhao, Linear inviscid damping and vorticity depletion for shear flows, Ann. PDE, 5 (2019), 3, 101 pp.
\bibitem{WZZ-3} D. Wei, Z. Zhang and W. Zhao, Linear inviscid damping and enhanced dissipation for the Kolmogorov flow, Adv. Math., 362 (2020), 106963, 103 pp.
\bibitem{WZZ} D. Wei, Z. Zhang, H. Zhu, Linear inviscid damping for the $\beta$-plane equation,
 Comm. Math. Phys., 375 (2020), 127-174.
\bibitem{Zillinger-1} C. Zillinger, Linear inviscid damping for monotone shear flows, Trans. Am. Math. Soc., 369 (2017), 8799-8855.
\bibitem{Zillinger-2}  C. Zillinger, Linear inviscid damping for monotone shear flows in a finite periodic channel, boundary effects, blow-up and critical Sobolev regularity, Arch. Ration. Mech. Anal., 221 (2016), 1449-1509.
\end{thebibliography}
\end{document}